\documentclass[12pt,letterpaper,leqno]{amsart}
\usepackage{microtype}
\usepackage{amssymb}
\usepackage{amsfonts}
\usepackage{geometry}
\usepackage{amsxtra}
\usepackage{graphicx}
\usepackage{tikz}
\usepackage{hyperref}
\usepackage{amsthm}
\usepackage[normalem]{ulem}
\usepackage{comment}

\usetikzlibrary{trees}
\renewcommand{\Im}{\operatorname{Im}}

\newcommand{\sgn}{\operatorname{sgn}}
\newcommand\encircle[1]{
	\text{
  		\tikz[baseline=(X.base)] 
    		\node (X) [draw, shape=circle, inner sep=0] {${\strut #1}$};
    	}
 }
\newtheorem{theorem}{Theorem}
\theoremstyle{plain}

\newtheorem{algorithm}{Algorithm}

\newtheorem{claim}[theorem]{Claim}

\newtheorem{corollary}[theorem]{Corollary}

\newtheorem{definition}[theorem]{Definition}
\newtheorem{example}{Example}

\newtheorem{lemma}[theorem]{Lemma}

\newtheorem{proposition}[theorem]{Proposition}
\newtheorem{remark}[theorem]{Remark}

\numberwithin{equation}{section}
\numberwithin{theorem}{section} 
\input{tcilatex}
\begin{document}
\title[$H^{1}$ Unconditional Uniqueness for $\mathbb{T}^{4}$ Cubic NLS]{The
Unconditional Uniqueness for the Energy-critical Nonlinear Schr\"{o}dinger
Equation on $\mathbb{T}^{4}$}
\author{Xuwen Chen}
\address{Department of Mathematics, University of Rochester, Rochester, NY
14627}
\email{chenxuwen@math.umd.edu}
\urladdr{http://www.math.rochester.edu/people/faculty/xchen84/}
\author{Justin Holmer}
\address{Department of Mathematics, Brown University, 151 Thayer Street,
Providence, RI 02912}
\email{justin$\underline{\;\,}$holmer@brown.edu}
\urladdr{http://www.math.brown.edu/\symbol{126}holmer/}
\date{05/20/2020}

\begin{abstract}
We consider the $\mathbb{T}^{4}$ cubic NLS which is energy-critical. We
study the unconditional uniqueness of solutions to the NLS via the cubic
Gross-Pitaevskii hierarchy, an uncommon method, which does not require the
existence of a solution in the Strichartz type spaces. We prove $U$-$V$
multilinear estimates to replace the previously used Sobolev multilinear
estimates, which fail on $\mathbb{T}^{4}$. To incorporate the weaker
estimates, we work out new combinatorics from scratch and compute, for the
first time, the time integration limits, in the recombined Duhamel-Born
expansion. The new combinatorics and the $U$-$V$ estimates then seamlessly
conclude the $H^{1}$ unconditional uniqueness for the NLS under the infinite
hierarchy framework. This work establishes a unified scheme to prove $H^{1}$
uniqueness for the $\mathbb{R}^{3}/\mathbb{R}^{4}/\mathbb{T}^{3}/\mathbb{T}%
^{4}$ energy-critical Gross-Pitaevskii hierarchies and thus the
corresponding NLS.
\end{abstract}

\maketitle
\tableofcontents

\section{Introduction}

The cubic nonlinear Schr\"{o}dinger equation (NLS) in four dimensions 
\begin{eqnarray}
i\partial _{t}u &=&-\Delta u\pm \left\vert u\right\vert ^{2}u\text{ in }%
\mathbb{R}\times \Lambda   \label{NLS:T^4 cubic NLS} \\
u(0,x) &=&u_{0},  \notag
\end{eqnarray}%
where $\Lambda =\mathbb{R}^{4}$ or $\mathbb{T}^{4}$, is called
energy-critical as it is invariant under the $\dot{H}^{1}$ scaling 
\begin{equation*}
u(t,x)\mapsto u_{\lambda }(t,x)=\frac{1}{\lambda }u(\frac{t}{\lambda ^{2}},%
\frac{x}{\lambda })
\end{equation*}%
if $\Lambda =\mathbb{R}^{4}$. The large datum global well-posedness of the
defocusing case of (\ref{NLS:T^4 cubic NLS}), was first proved for $\Lambda =%
\mathbb{R}^{4}$ in \cite{RV} after the breakthrough on the defocusing $%
\mathbb{R}^{3}$ quintic problem \cite{CKSTT}. After that, the
energy-critical defocusing $\mathbb{T}^{3}$ quintic problem's global
well-posedness is settled in \cite{HTT,IP}, by partially invoking the $%
\mathbb{R}^{3}$ result \cite{CKSTT}. Such a problem for (\ref{NLS:T^4 cubic
NLS}) when $\Lambda =\mathbb{T}^{4}$ is then subsequently proved in \cite%
{HTT1,KV,Y}. The goal of this paper is to establish the $H^{1}$
unconditional uniqueness for (\ref{NLS:T^4 cubic NLS}) on $\mathbb{T}^{4}$.

\begin{theorem}
\label{THM:MainUniquenessTHM}There is at most one $C_{[0,T]}^{0}H_{x}^{1}%
\cap \dot{C}_{[0,T]}^{1}H_{x}^{-1}$ solution\footnote{%
A $C_{[0,T]}^{0}H_{x}^{1}$ distributional solution is automatically a $%
C_{[0,T]}^{0}H_{x}^{1}\cap \dot{C}_{[0,T]}^{1}H_{x}^{-1}$ solution. We wrote
the latter here as it is a more direct space for (\ref{NLS:T^4 cubic NLS}).}
to (\ref{NLS:T^4 cubic NLS}) on $\mathbb{T}^{4}$.
\end{theorem}

The unconditional uniqueness problems, even in the $H^{1}$-critical setting,
are often overlooked as solving them in $\mathbb{R}^{n}$ after proving the
well-posedness is relatively simple.\footnote{%
See, for example, \cite[\S 16]{CKSTT}.} For NLS on $\mathbb{T}^{n}$, such
problems are delicate as estimates on $\mathbb{T}^{n}$, especially the $%
\mathbb{T}^{n}$ Strichartz estimates, are weaker than their $\mathbb{R}^{n}$
counterparts. For example, for the $\mathbb{R}^{n}$ case, one can easily use
the existence of a better solution in Strichartz spaces to yield the
unconditional uniqueness. But such a technique does not work well in the $%
\mathbb{T}^{n}$ case. In fact, Theorem \ref{THM:MainUniquenessTHM} for the $%
\mathbb{T}^{3}$ quintic case at $H^{1}$ regularity was not known until
recently \cite{CJInvent}.

On the other hand, away from answering the original mathematical problem%
\footnote{%
See a discussion after Lemma \ref{Lem:R4TrilinearWithFreqLocal} of this
problem using $\mathbb{T}^{3}$ as an example.} that there could be multiple
solutions coming from different spaces in which (\ref{NLS:T^4 cubic NLS}) is
wellposed, the unconditional uniqueness problems on $\mathbb{T}^{n}$ have
practical applications. An example is the control problem for the
Lugiato--Lefever system, first formulated in \cite{LL}, which could be
considered as a NLS with forcing: 
\begin{eqnarray}
i\partial _{t}u_{f} &=&-\Delta u_{f}\pm \left\vert u_{f}\right\vert
^{p-1}u_{f}+f\text{ in }\mathbb{R}\times \mathbb{T}^{n},  \label{eqn:LLE} \\
u_{f}(0,x) &=&u_{0},  \notag
\end{eqnarray}%
the problem is to find $f$ and $u_{0}$ such that $u_{f}\in X$, for some
space $X$ in which (\ref{eqn:LLE}) is well-posed, minimizes some given
functional $Z(u)$. For some experimental and engineering purposes, the
spatial domain has to be $\mathbb{T}^{n}$. The space $X$, in which one looks
for the minimizer, largely determines the difficulty. If $X=L_{x}^{2}$ or $%
H_{x}^{1}$, there are techniques readily available to hunt for minimizers.
However, how to search for minimizers when $X$ is a proper subspace of $%
H_{x}^{1}$ like $H_{x}^{2}$ or $H_{x}^{1}\cap L_{t}^{p}L_{x}^{q}$, a common
space for well-posedness, remains open. Such a dilemma can be resolved if
one has unconditional uniqueness results like Theorem \ref%
{THM:MainUniquenessTHM}. The recent work \cite{SZ} has also brought
attention to the analysis of PDEs over $\mathbb{T}^{4-m}\times \mathbb{R}%
^{m} $ as it is shown that such domains arise on the necks of some K3
surfaces which are Calabi-Yau manifolds.

To prove Theorem \ref{THM:MainUniquenessTHM}, we will use the cubic
Gross-Pitaevskii (GP) hierarchy on $\mathbb{T}^{4}$, which is uncommon in
the analysis of NLS. Let $\mathcal{L}_{k}^{1}$ denote the space of trace
class operators on $L^{2}(\mathbb{T}^{4k})$. The cubic GP hierarchy on $%
\mathbb{T}^{4}$ is a sequence $\left\{ \gamma ^{(k)}(t)\right\} \in \oplus
_{k\geqslant 1}C\left( \left[ 0,T\right] ,\mathcal{L}_{k}^{1}\right) $ which
satisfies the infinitely coupled hierarchy of equations: 
\begin{equation}
i\partial _{t}\gamma ^{(k)}=\sum_{j=1}^{k}\left[ -\Delta _{x_{j}},\gamma
^{(k)}\right] \pm b_{0}\sum_{j=1}^{k}\limfunc{Tr}\nolimits_{k+1}\left[
\delta (x_{j}-x_{k+1}),\gamma ^{(k+1)}\right]  \label{Hierarchy:T^4 cubic GP}
\end{equation}%
where $b_{0}$ is some coupling constant, $\pm $ denotes defocusing /
focusing. Given any solution $u$ of (\ref{NLS:T^4 cubic NLS}), it generates
a solution to (\ref{Hierarchy:T^4 cubic GP}) by letting 
\begin{equation}
\gamma ^{(k)}=\left\vert u\right\rangle \left\langle u\right\vert ^{\otimes
k},  \label{eqn:solution to GP generated by NLS}
\end{equation}%
in operator form or 
\begin{equation*}
\gamma ^{(k)}(t,\mathbf{x}_{k},\mathbf{x}_{k}^{\prime
})=\dprod\limits_{j=1}^{k}u(t,x_{j})\bar{u}(t,x_{j}^{\prime })
\end{equation*}%
in kernel form if we write $\mathbf{x}_{k}=(x_{1},...,x_{k})\in \mathbb{T}%
^{4k}$.

Hierarchy (\ref{Hierarchy:T^4 cubic GP}) arises in the derivation of NLS as
a $N\rightarrow \infty $ limit of quantum $N$-body dynamics. It was first
derived in the work of Erd\"{o}s, Schlein, and Yau \cite%
{E-S-Y2,E-S-Y5,E-S-Y3} for the $\mathbb{R}^{3}$ defocusing cubic case around
2005.\footnote{%
See also \cite{AGT} for the 1D defocusing cubic case around the same time.}
They proved delicatedly that there is a unique solution to the $\mathbb{R}
^{3}$ cubic GP hierarchy in a $H^{1}$-type space (unconditional uniqueness)
in \cite{E-S-Y2} with a sophisticated Feynman graph analysis. This first
series of ground breaking papers have motivated a large amount of work.

In 2007, Klainerman and Machedon \cite{KM}, inspired by \cite%
{E-S-Y2,KMNullForm}, proved the uniqueness of solutions regarding the $%
\mathbb{R}^{3}$ cubic GP hierarchy in a Strichartz-type space (conditional
uniqueness). They proved a collapsing type estimate, which implies a
multilinear estimate when applied to factorized solutions like (\ref%
{eqn:solution to GP generated by NLS}), to estimate the inhomogeneous term,
and provided a different combinatorial argument, the now so-called
Klainerman-Machedon (KM) board game, to combine the inhomogeneous terms
effectively reducing their numbers. At that time, it was unknown how to
prove that the limits coming from the $N$-body dynamics are in the
Strichartz type spaces even though the solutions to (\ref{Hierarchy:T^4
cubic GP}) generated by the $\mathbb{R}^{3}$ cubic NLS naturally lie in both
the $H^{1}$-type space and the Strichartz type space. Nonetheless, \cite{KM}
has made the analysis of (\ref{Hierarchy:T^4 cubic GP}) approachable to PDE
analysts and the KM board game has been used in every work involving
hierarchy (\ref{Hierarchy:T^4 cubic GP}).\footnote{%
The analysis of the Boltzmann hierarchy is also using the KM board game,
see, for example, \cite{TCRDNP}.} When Kirkpatrick, Schlein, and Staffilani 
\cite{KSS} derived (\ref{Hierarchy:T^4 cubic GP}) and found that the
Klainerman-Machedon Strichartz-type bound can be obtained via a simple trace
theorem for the defocusing case in $\mathbb{R}^{2}$ and $\mathbb{T}^{2}$ in
2008, many works \cite{TCNP2,ChenAnisotropic,C-HFocusing,C-HFocusingII,Xie,S}
then followed such a scheme for the uniqueness of GP hierarchies. However,
how to check the Klainerman-Machedon Strichartz type bound in the 3D cubic
case remained fully open at that time.

T. Chen and Pavlovic laid the foundation for the 3D quintic defocusing
energy-critical case by studying the 1D and 2D defocusing quintic case, in
their late 2008 work \cite{TCNP2}, in which they proved that the 2D quintic
case, a case usually considered equivalent to the 3d cubic case, does
satisfy the Klainerman-Machedon Strichartz-type bound though proving it for
the 3D cubic case was still open.

T. Chen and Pavlovic also initiated the study of the well-posedness theory
of (\ref{Hierarchy:T^4 cubic GP}) with general initial datum as an
independent subject away from the quantum $N$-body dynamics in \cite%
{TCNP1,TCNP3,TCNP4}. (See also \cite{TCNPNT,MNPS,MNPRS1,MNPRS2,S,SS}.) On
the one hand, generalizing the problem could help to attack the
Klainerman-Machedon Strichartz type bound problem. On the other hand, it
leads one to consider whether hierarchy (\ref{Hierarchy:T^4 cubic GP}), the
general equation, could hold more in store than its special solution, NLS (%
\ref{NLS:T^4 cubic NLS}).\footnote{%
Private communication.} Then in 2011, T. Chen and Pavlovic proved that the
3D cubic Klainerman-Machedon Strichartz type bound does hold for the
defocusing $\beta <1/4$ case in \cite{TCNP5}. The result was quickly
improved to $\beta \leqslant 2/7$ by X.C. in \cite{Chen3DDerivation} and to
the almost optimal case, $\beta <1,$ by X.C. and J.H. in \cite{C-H2/3,C-H<1}%
, by lifting the $X_{1,b}$ space techniques from NLS theory into the field.

Around the same period of time, Gressman, Sohinger, and Staffilani \cite{GSS}
studied the uniqueness of (\ref{Hierarchy:T^4 cubic GP}) in the $\mathbb{T}
^{3}$ setting and found that the sharp collapsing estimate on $\mathbb{T}
^{3} $ needs $\varepsilon $ more derivatives than the $\mathbb{R}^{3}$ case
in which one derivative is needed. Herr and Sohinger later generalized this
fact to all dimensions in \cite{HS1}. That is, collapsing estimates on $%
\mathbb{T}^{n}$ always need $\varepsilon $ more derivatives than the $%
\mathbb{R}^{n}$ case proved in \cite{ChenAnisotropic}.\footnote{%
Except the 1D case, as shown in \cite{C-HFocusing}, this $\varepsilon $-loss
also happens in $\mathbb{R}^{1}$.}

In 2013, T. Chen, Hainzl, Pavlovic, and Seiringer, introduced the quantum de
Finetti theorem, from \cite{LMR}, to the derivation of the time-dependent
power-type NLS and provided, in \cite{CHPS}, a simplified proof of the $%
\mathbb{R}^{3}$ unconditional uniqueness theorem regarding (\ref%
{Hierarchy:T^4 cubic GP}) in \cite{E-S-Y2}. The application of the quantum
de Finetti theorem allows one to replace the collapsing estimates by the
multilinear estimates. The scheme in \cite{CHPS}, which consists of the
Klainerman-Machedon board game, the quantum de Finetti theorem, and the
multilinear estimates, is robust. Sohinger used this scheme in \cite{S1} to
address the aforementioned $\varepsilon $-loss problem for the defocusing $%
\mathbb{T}^{3}$ cubic case. Hong, Taliaferro, and Xie used this scheme to
obtain unconditional uniqueness theorems for (\ref{Hierarchy:T^4 cubic GP})
in $\mathbb{R}^{n}$, $n=1,2,3$, with regularities matching the NLS analysis,
in \cite{HTX}, and $H^{1}$ small solution uniqueness for the $\mathbb{R}^{3}$
quintic case in \cite{HTX1}. (See also \cite{C-HFocusingII,C-PUniqueness}.)

The analysis of GP hierarchy did not yield new NLS results with regularity
lower than the NLS analysis until \cite{HS2, CJInvent}. (See also \cite{K}
for recent development using NLS analysis.) In \cite{HS2}, with the scheme
in \cite{CHPS}, Herr and Sohinger generalized the usual Sobolev multilinear
estimates, to Besov spaces and obtained new unconditional uniqueness results
regarding (\ref{Hierarchy:T^4 cubic GP}) and hence NLS (\ref{NLS:T^4 cubic
NLS}) on $\mathbb{T}^{n}$. The result has pushed the regularity requirement
for uniqueness of (\ref{NLS:T^4 cubic NLS}) lower than the number coming
from the NLS analysis. Moreover, their result has covered the whole
subcritical region for $n\geqslant 4$, which includes Theorem \ref%
{THM:MainUniquenessTHM} with $H^{1+\varepsilon }$ regularity.

In \cite{CJInvent}, by discovering the new hierarchical uniform frequency
localization (HUFL) property for the GP hierarchy, which reduces to a new
statement even for NLS, X.C. and J.H. established a new $H^{1}$-type
uniqueness theorem for the $\mathbb{T}^{3}$ quintic energy-critical GP
hierarchy. The new uniqueness theorem, though neither conditional nor
unconditional for the GP hierarchy, implies the $H^{1}$ unconditional
uniqueness result for the $\mathbb{T}^{3}$ quintic energy-critical NLS. It
is then natural to consider the $\mathbb{T}^{4}$ cubic energy-critical case
in this paper. However, the key Sobolev multilinear estimates in \cite%
{CJInvent}, fail for the $\mathbb{T }^{4}$ cubic case here, and it turns
out, surprisingly, that $\mathbb{T}^{4}$ is \emph{unique / special} when
compared to $\mathbb{R}^{3}/\mathbb{R}^{4}/ \mathbb{T}^{3}$.

\subsection{Outline of the Proof of Theorem \protect\ref%
{THM:MainUniquenessTHM}}

We will prove Theorem \ref{THM:MainUniquenessTHM} as a corollary of Theorem %
\ref{Thm:GP uniqueness}, a GP hierarchy uniqueness theorem stated in \S \ref%
{Sec:GP Uniqueness}. As Theorem \ref{Thm:GP uniqueness} requires the HUFL
condition, we prove that any $C_{[0,T]}^{0}H_{x}^{1}\cap \dot{C}%
_{[0,T]}^{1}H_{x}^{-1}$ solution to (\ref{NLS:T^4 cubic NLS}) on $\mathbb{T}%
^{4}$ satisfies uniform in time frequency localization (UTFL) with Lemma \ref%
{Cor:UTFLforNLS}. That is, solutions to (\ref{Hierarchy:T^4 cubic GP})
generated from (\ref{NLS:T^4 cubic NLS}) via formula (\ref{eqn:solution to
GP generated by NLS}) satisfy the HUFL condition. Thus we will have
established Theorem \ref{THM:MainUniquenessTHM} once we have proved Theorem %
\ref{Thm:GP uniqueness}.

As Theorem \ref{Thm:GP uniqueness} is an energy-critical case, due to the
known similarities between the $\mathbb{R}^{3}$ quintic and the $\mathbb{R}%
^{4}$ cubic cases, one would guess that the proof of the $\mathbb{T}^{3}$
quintic case goes through for the $\mathbb{T}^{4}$ cubic case as well. It
does \emph{not}. As mentioned before, the key Sobolev multilinear estimates
in \cite{CJInvent}, fail here. Interested readers can see Appendix \ref%
{Sec:OldEstimatesFail} for the proof that they fail. In this $H^{1}$%
-critical setting, the next in the line replacement would be the weaker $U$-$%
V$ multilinear estimates. The $U$-$V$ trilinear estimates do hold on $%
\mathbb{T}^{4}$. This is where we start.

In \S \ref{Sec:UVandTrilinear}, we first give a short introduction to the $U$
-$V$ space referring the standard literature \cite{HTT,IP,KTV,KV}, then
prove the $U$-$V$ version of the $\mathbb{T}^{4}$ trilinear estimates,
Lemmas \ref{Lem:T4TrilinearWithUV} and \ref{Lem:T4TrilinearWithUVH1}. The
proof of the $U$-$V$ trilinear estimates is less technical and simpler than
the proof of the Sobolev multilinear estimates in \cite{CJInvent}, as they
are indeed weaker.\footnote{%
The stronger Sobolev multilinear estimates hold for $\mathbb{R}^{4}$. See
Appendix \ref{sec:R4TrilinearEstimate}.} But these $U$-$V$ trilinear
estimates still highly rely on the scale invariant Stichartz estimates / $%
l^{2}$-decoupling theorem in \cite{BD,KV}.

Though the $U$-$V$ trilinear estimates hold in $\mathbb{T}^{4}$, there is no
method available to use them to prove uniqueness for GP hierarchies. This is
why estimates in the hierarchy framework have always been about $%
L_{t}^{p}H_{x}^{s}$. Even in \cite{C-H2/3,C-H<1} in which the $X_{s,b}$
techniques were used, they were used only once in the very end of the
iteration instead of every step of the iteration to yield smallness.
Conceptually speaking, while it is easy to bound the $L_{t}^{\infty
}H_{x}^{s}$ norm by the $U$-$V$ norms, one has to pay half a derivative in
time to come back. On the one hand, we are proving an unconditional
uniqueness theorem, we have to come back to the Sobolev spaces in the end of
the proof. On the other hand, we are proving a critical result, we do not
have an extra half derivative in time to spare. To fix this problem, we
adjust how the multilinear estimates apply to the Duhamel-Born expansion of $%
\gamma ^{(k)}$ after the application of the Klainerman-Machedon board game,
so that the $U$- $V$ trilinear estimates only lands on \textquotedblleft
Duhamel-like" integral.

The main problem now surfaces. The time integration domain $D_{m}$ of the
aforementioned \textquotedblleft Duhamel-like" integrals, coming from the
Klainerman-Machedon board game, is a union of a very large number of high
dimensional simplexes under the action of a proper subset of the permutation
group $S_{k}$ specific to every integrand. To, at least, have a chance to
use space-time norms like $X_{s,b}$ and $U$-$V,$ which are very sensitive to
the irregularity of time domain as they involve taking time derivatives, one
would have to know what $D_{m}$ is. It turns out, $D_{m}$ coming from the
original Klainerman-Machedon board game is not fully compatible with the $U$%
- $V$ trilinear estimates. To this end, we establish an extended
Klainerman-Machedon board game which is compatible in \S \ref%
{Sec:ElaboratedKMBoardGame}.

We first develop, as a warm up, in \S \ref{Sec:Find Integration Limits}, via
a detailed tree\footnote{%
This is the 3rd type of tree used in the analysis of GP hiearchies. The 1st
two are the Feymann graphs in \cite{E-S-Y2} and the binary trees in \cite%
{CHPS}. They are coded differently and serve different purposes.} diagram
representation, a more elaborated proof of the original Klainerman-Machedon
board game, which yields, for the first time, an algorithm to directly
compute $D_{m}$ and domains like that. Graphically speaking, under our tree
representation, the original Klainerman-Machedon board game combines all the
trees with the same skeletons into a \textquotedblleft upper echelon" class
which can be represented by a upper echelon tree.\footnote{%
It is possible to write \S \ref{Sec:ElaboratedKMBoardGame} without trees (or
matrices), but we would lose this graphical explanation. Due to the
coupling, recursive, and iterative features of the hiearchies, algorithm
terminologies happen to be helpful.} The time integration domain $D_{m}$ for
each upper echelon class can be directly read off from the upper echelon
tree representing the class.

We then introduce, in \S \ref{Sec:Wild1}-\ref{Sec:Wild5}, the wild moves,
which allow us to uncover more integrals in the Duhamel-Born expansion with
same integrands after permutation, and to combine them into
\textquotedblleft reference" classes. Graphically speaking, it allows the
combination of trees sharing the same reference enumeration but different
structure. However, the wild moves are not compatible with the upper echelon
classes coming from the original Klainerman-Machedon board game. We have to
restart from the very beginning at the level of the $2^{k}k!$ summands.

Before applying the wild moves, in \S \ref{Sec:Wild1}-\ref{Sec:Wild2}, we
turn the $2^{k}k!$ summands in the initial Duhamel-Born expansion, into
their tamed forms, which would be invariant under the wild moves, via a
reworked signed Klainerman-Machedon acceptable moves. We then sort the tamed
forms into tamed classes via the wild moves in \S \ref{Sec:Wild4}. Finally,
in \S \ref{Sec:Wild5}, we use the algorithm we just developed in \S \ref%
{Sec:Find Integration Limits} to calculate the time integration domain for
each tamed class. In fact, we prove that, given a tamed class, there is a
reference form representing the tamed class and the time integration domain
for the whole tamed class can be directly read out from the reference form.

Using this extended Klainerman-Machedon board game coming from scratch, we
found that, the previously thought unrepresentable or even disconnected time
integration domain specific for each integrand, the time integration domain
which got expanded into $[0,T]^{k}$ in all previous work as there was no
other options to use it, that time integration domain can be
\textquotedblleft miraculously" written as one single iterated integral in
the integration order ready to apply the quantum de Finette theorem.
Moreover, once these integration limits are put together with the integrand,
each distinct tamed class becomes an exact fit to apply the $U $-$V$
trilinear estimates we proved in \S \ref{Sec:UVandTrilinear}. This
combinatorial analysis, which is compatible with space-time norms and the
method to explicitly compute the time integration domain in the general
recombined Duhamel-Born expansion (which includes more than the GP
hierarchies) is the main technical achievements in this paper.

With everything set ready by the extended Klainerman-Machedon board game we
concluded in \S \ref{Sec:ElaboratedKMBoardGame}, the quantum de Finette
theorem from \cite{CHPS}, the $U$-$V$ space techniques from \cite{KTV}, the
trilinear estimates proved using the scale invariant Stichartz estimates / $%
l^{2}$-decoupling theorem in \cite{BD,KV}, and the HUFL properties from \cite%
{CJInvent}, then all work together seamlessly in \S \ref{Sec:GP Uniqueness-2}
to establish Theorem \ref{Thm:GP uniqueness} and provide a unified proof of
the large solution uniqueness for the $\mathbb{R}^{3}$/$\mathbb{T}^{3}$
quintic and the $\mathbb{R}^{4}$/$\mathbb{T}^{4}$ cubic energy-critical GP
hierarchies and hence the corresponding NLS. The discovery of such an
unexpected close and effective collaboration of these previously independent
deep theorems is the main novelty of this paper. We now expect to be able to
bring the full strength of the dispersive estimate technology to bear on
various type of hierarchies of equations and related problems, and this is
our first example of it.

\section{Trilinear Estimates in the $U$-$V$ Spaces\label{Sec:UVandTrilinear}}

As mentioned in the introduction, our proof of Theorem \ref%
{THM:MainUniquenessTHM} requires the $U$-$V$ space while the $\mathbb{R}^{3}$%
/$\mathbb{R}^{4}/\mathbb{T}^{3}$ cases do not. Referring to the now standard
text \cite{KTV} for the definition of $U_{t}^{p}$ and $V_{t}^{p}$, we define 
\begin{equation*}
\Vert u\Vert _{X^{s}\left( \left[ 0,T\right) \right) }=\left( \sum_{\xi \in 
\mathbb{Z}^{4}}\left\langle \xi \right\rangle ^{2s}\Vert \widehat{%
e^{-it\Delta }u(t,\cdot )}\left( \xi \right) \Vert _{U_{t}^{2}}^{2}\right) ^{%
\frac{1}{2}}
\end{equation*}%
and 
\begin{equation*}
\Vert u\Vert _{Y^{s}\left( \left[ 0,T\right) \right) }=\left( \sum_{\xi \in 
\mathbb{Z}^{4}}\left\langle \xi \right\rangle ^{2s}\Vert \widehat{%
e^{-it\Delta }u(t,\cdot )}\left( \xi \right) \Vert _{V_{t}^{2}}^{2}\right) ^{%
\frac{1}{2}}
\end{equation*}%
as in \cite{HTT,HTT1,IP,KV}. In particular, we have the usual properties, 
\begin{equation}
\Vert u\Vert _{L_{t}^{\infty }H_{x}^{s}}\lesssim \Vert u\Vert _{X^{s}},
\label{E:H^1/2-embedding}
\end{equation}%
\begin{equation}
\left\Vert e^{it\Delta }f\right\Vert _{Y^{s}}\lesssim \left\Vert e^{it\Delta
}f\right\Vert _{X^{s}}\lesssim \Vert f\Vert _{H^{s}}
\label{eqn:v norm applied to free solution}
\end{equation}%
\begin{eqnarray}
&&\left\Vert \int_{a}^{t}e^{i(t-s)\Delta }f(s,\cdot )\,ds\right\Vert
_{X^{s}\left( \left[ 0,T\right) \right) }
\label{eqn:dual definition of Duhamel X^s norm} \\
&\leqslant &\sup_{v\in Y^{-s}\left( \left[ 0,T\right) \right) :\left\Vert
v\right\Vert _{Y^{-s}}=1}\int_{0}^{T}\int_{\mathbb{T}^{4}}f(t,x)\overline{%
v(t,x)}dtdx\text{, }\forall a\in \left[ 0,T\right)  \notag
\end{eqnarray}%
which were proved on \cite[p.46]{KTV} and in \cite[Propositions 2.8-2.11]%
{HTT}. With the above definitions of $X^{s}$ and $Y^{s}$, we have the
following trilinear estimates.

\begin{lemma}
\label{Lem:T4TrilinearWithUV}On $\mathbb{T}^{4}$, we have the high frequency
estimate 
\begin{align}
\hspace{0.3in}& \hspace{-0.3in}
\iint_{x,t}u_{1}(t,x)u_{2}(t,x)u_{3}(t,x)g(t,x)dxdt
\label{eqn:T4UVTrilinearEstimateHigh} \\
& \lesssim \Vert u_{1}\Vert _{Y^{-1}}\Vert u_{2}\Vert _{Y^{1}}\Vert
u_{3}\Vert _{Y^{1}}\Vert g\Vert _{Y^{1}},  \notag
\end{align}
and the low frequency estimate 
\begin{align}
\hspace{0.3in}& \hspace{-0.3in}\iint_{x,t}u_{1}(t,x)\left( P_{\leqslant
M_{0}}u_{2}\right) (t,x)u_{3}(t,x)g(t,x)dxdt
\label{eqn:T4UVTrilinearEstimateLow} \\
& \lesssim T^{\frac{1}{7}}M_{0}^{\frac{3}{5}}\Vert u_{1}\Vert _{Y^{-1}}\Vert
P_{\leqslant M_{0}}u_{2}\Vert _{Y^{1}}\Vert u_{3}\Vert _{Y^{1}}\Vert g\Vert
_{Y^{1}},  \notag
\end{align}
for all $T\leqslant 1$ and all frequencies $M_{0}\geqslant 1$. Or 
\begin{eqnarray}
&&\left\Vert \int_{a}^{t}e^{i(t-s)\Delta }\left( u_{1}u_{2}u_{3}\right)
\,ds\right\Vert _{X^{-1}\left( \left[ 0,T\right) \right) }
\label{eqn:H-1-T4TrilinearEstimateWithGain} \\
&\lesssim &\Vert u_{1}\Vert _{Y^{-1}}\left( T^{\frac{1}{7}}M_{0}^{\frac{3}{5}
}\Vert P_{\leqslant M_{0}}u_{2}\Vert _{Y^{1}}+\Vert P_{>M_{0}}u_{2}\Vert
_{Y^{1}}\right) \Vert u_{3}\Vert _{Y^{1}}  \notag
\end{eqnarray}
and 
\begin{equation}
\left\Vert \int_{a}^{t}e^{i(t-s)\Delta }\left( u_{1}u_{2}u_{3}\right)
\,ds\right\Vert _{X^{-1}\left( \left[ 0,T\right) \right) }\lesssim \Vert
u_{1}\Vert _{Y^{-1}}\Vert u_{2}\Vert _{Y^{1}}\Vert u_{3}\Vert _{Y^{1}}
\label{eqn:H-1-T4TrilinearEstimateWithoutGain}
\end{equation}
Moreover, if $u_{j}=e^{it\Delta }f_{j}$ for some $j$, then the $Y^{s}$ norm
of $u_{j}$ in (\ref{eqn:H-1-T4TrilinearEstimateWithGain}) or (\ref%
{eqn:H-1-T4TrilinearEstimateWithoutGain}) can be replaced by the $H^{s}$
norm of $f_{j}$.
\end{lemma}

Similarly, we have the $X^{1}$ estimates.

\begin{lemma}
\label{Lem:T4TrilinearWithUVH1}On $\mathbb{T}^{4}$, we have the high
frequency estimate 
\begin{align}
\hspace{0.3in}& \hspace{-0.3in}
\iint_{x,t}u_{1}(t,x)u_{2}(t,x)u_{3}(t,x)g(t,x)dxdt
\label{eqn:T4UVTrilinearEstimateHighH1} \\
& \lesssim \Vert u_{1}\Vert _{Y^{1}}\Vert u_{2}\Vert _{Y^{1}}\Vert
u_{3}\Vert _{Y^{1}}\Vert g\Vert _{Y^{-1}}.  \notag
\end{align}
and the low frequency estimate 
\begin{align}
\hspace{0.3in}& \hspace{-0.3in}\iint_{x,t}u_{1}(t,x)\left( P_{\leqslant
M_{0}}u_{2}\right) (t,x)u_{3}(t,x)g(t,x)dxdt
\label{eqn:T4UVTrilinearEstimateLowH1} \\
& \lesssim T^{\frac{1}{7}}M_{0}^{\frac{3}{5}}\Vert u_{1}\Vert _{Y^{1}}\Vert
P_{\leqslant M_{0}}u_{2}\Vert _{Y^{1}}\Vert u_{3}\Vert _{Y^{1}}\Vert g\Vert
_{Y^{-1}}.  \notag
\end{align}
In other words, 
\begin{eqnarray}
&&\left\Vert \int_{a}^{t}e^{i(t-s)\Delta }\left( u_{1}u_{2}u_{3}\right)
\,ds\right\Vert _{X^{1}\left( \left[ 0,T\right) \right) }
\label{eqn:H1-T4TrilinearEstimateWithGain} \\
&\lesssim &\Vert u_{1}\Vert _{Y^{1}}\left( T^{\frac{1}{7}}M_{0}^{\frac{3}{5}
}\Vert P_{\leqslant M_{0}}u_{2}\Vert _{Y^{1}}+\Vert P_{>M_{0}}u_{2}\Vert
_{Y^{1}}\right) \Vert u_{3}\Vert _{Y^{1}}  \notag
\end{eqnarray}
and 
\begin{equation}
\left\Vert \int_{a}^{t}e^{i(t-s)\Delta }\left( u_{1}u_{2}u_{3}\right)
\,ds\right\Vert _{X^{1}\left( \left[ 0,T\right) \right) }\lesssim \Vert
u_{1}\Vert _{Y^{1}}\Vert u_{2}\Vert _{Y^{1}}\Vert u_{3}\Vert _{Y^{1}}
\label{eqn:H1-T4TrilinearEstimateWithoutGain}
\end{equation}
Moreover, if $u_{j}=e^{it\Delta }f_{j}$ for some $j$, then the $Y^{s}$ norm
of $u_{j}$ in (\ref{eqn:H1-T4TrilinearEstimateWithGain}) or (\ref%
{eqn:H1-T4TrilinearEstimateWithoutGain}) can be replaced by the $H^{s}$ norm
of $f_{j}$.
\end{lemma}

We prove only Lemma \ref{Lem:T4TrilinearWithUV}. On the one hand, Lemma \ref%
{Lem:T4TrilinearWithUVH1} follows from the proof of Lemma \ref%
{Lem:T4TrilinearWithUV} with little modifications. On the other hand, (\ref%
{eqn:T4UVTrilinearEstimateHighH1}) has already been proved as \cite[
Proposition 2.12]{HTT1} \cite[(4.4)]{KV} and the non-scale-invariant
estimate (\ref{eqn:T4UVTrilinearEstimateLowH1}) is easy. Thence, we omit the
proof of Lemma \ref{Lem:T4TrilinearWithUVH1}.

\begin{remark}
At least one of the $L_{t}^{1}H_{x}^{s}$ versions of (\ref%
{eqn:H-1-T4TrilinearEstimateWithoutGain}) and (\ref%
{eqn:H1-T4TrilinearEstimateWithoutGain}) fail (see Appendix \ref%
{Sec:OldEstimatesFail}). While the $L_{t}^{1}H_{x}^{s}$ version of (\ref%
{eqn:H-1-T4TrilinearEstimateWithoutGain}) and (\ref%
{eqn:H1-T4TrilinearEstimateWithoutGain}) hold for $\mathbb{T}^{3}$ (see \cite%
{CJInvent}) and $\mathbb{R}^{4}$ (see Appendix \ref{sec:R4TrilinearEstimate}%
), we see that the relation between $\mathbb{T}^{3}$ and $\mathbb{T}^{4}$ is
very different from the relation between $\mathbb{R}^{3}$ and $\mathbb{R}%
^{4} $. This is the reason why $\mathbb{T}^{4}$ stands out and we have to
use $U$-$V$ spaces here.
\end{remark}

The following tools will be used to prove Lemma \ref{Lem:T4TrilinearWithUV}.

\begin{lemma}[Strichartz estimate on $\mathbb{T}^{4}$ \protect\cite{BD,KV}]
\label{estimate:T^3 strichartz}For $p>3$, 
\begin{equation}
\Vert P_{\leqslant M}u\Vert _{L_{t,x}^{p}}\lesssim M^{2-\frac{6}{p}}\Vert
u\Vert _{Y^{0}}  \label{eqn:T^4 p>3 u-v str}
\end{equation}
\end{lemma}

\begin{corollary}[Strichartz estimates on $\mathbb{T}^{4}$ with noncentered
frequency localization]
\label{cor:str with noncentered cube}Let $M$ be a dyadic value and let $Q$
be a (possibly) noncentered $M$-cube in Fourier space 
\begin{equation*}
Q=\left\{ \xi _{0}+\eta :\left\vert \eta \right\vert \leq M\right\} \,.
\end{equation*}
Let $P_{Q}$ be the corresponding Littlewood-Paley projection, then by the
Galilean invariance, we have 
\begin{equation}
\Vert P_{Q}u\Vert _{L_{t,x}^{p}}\lesssim M^{2-\frac{6}{p}}\Vert P_{Q}u\Vert
_{Y^{0}}\text{, }p>3.  \label{eqn:T^4 p>3 u-v str with Galinean}
\end{equation}
The net effect of this observation is that we pay a factor of only $M^{2- 
\frac{6}{p}}$, when applying (\ref{eqn:T^4 p>3 u-v str}).
\end{corollary}

\begin{proof}
Such a fact is well-known and widely used. Readers interested in a version
of the proof can see \cite[Corollary 5.18]{CJInvent}.
\end{proof}

\subsubsection{Proof of Lemma \protect\ref{Lem:T4TrilinearWithUV}}

We first present the proof of the sharp estimate (\ref%
{eqn:T4UVTrilinearEstimateHigh}), then (\ref{eqn:T4UVTrilinearEstimateLow})

\paragraph{Proof of (\protect\ref{eqn:T4UVTrilinearEstimateHigh})}

Let $I$ denote the integral in (\ref{eqn:T4UVTrilinearEstimateHigh}).
Decompose the 4 factors into Littlewood-Paley pieces so that 
\begin{equation*}
I=\sum_{M_{1},M_{2},M_{3},M}I_{M_{1},M_{2},M_{3},M}
\end{equation*}%
where 
\begin{equation*}
I_{M_{1},M_{2},M_{3},M}=\iint_{x,t}u_{1,M_{1}}u_{2,M_{2}}u_{3,M_{3}}g_{M}dxdt
\end{equation*}%
with $u_{j,M_{j}}=P_{M_{j}}u_{j}$ and $g_{M}=P_{M}g$. As $M_{2}$, $M_{3}$
and $M$ are symmetric, it suffices to take care of the $M_{1}\sim M_{2}\geq
M_{3}\geq M$ case. Decompose the $M_{1}$ and $M_{2}$ dyadic spaces into $%
M_{3}$ size cubes, then 
\begin{eqnarray*}
I_{1A} &\lesssim &\sum_{\substack{ M_{1},M_{2},M_{3},M  \\ M_{1}\sim
M_{2}\geq M_{3}\geq M}}\sum_{Q}\left\Vert
P_{Q}u_{1,M_{1}}P_{Q_{c}}u_{2,M_{2}}u_{3,M_{3}}g_{M}\right\Vert
_{L_{t,x}^{1}} \\
&\lesssim &\sum_{\substack{ M_{1},M_{2},M_{3},M  \\ M_{1}\sim M_{2}\geq
M_{3}\geq M}}\sum_{Q}\left\Vert P_{Q}u_{1,M_{1}}\right\Vert _{L_{t,x}^{\frac{%
10}{3}}}\left\Vert P_{Q_{c}}u_{2,M_{2}}\right\Vert _{L_{t,x}^{\frac{10}{3}%
}}\left\Vert u_{3,M_{3}}\right\Vert _{L_{t,x}^{\frac{10}{3}}}\left\Vert
g_{M}\right\Vert _{L_{t,x}^{10}}
\end{eqnarray*}%
Use (\ref{eqn:T^4 p>3 u-v str}) and (\ref{eqn:T^4 p>3 u-v str with Galinean}%
) 
\begin{eqnarray*}
&\lesssim &\sum_{\substack{ M_{1},M_{2},M_{3},M  \\ M_{1}\sim M_{2}\geq
M_{3}\geq M}}\sum_{Q}M_{3}^{\frac{2}{5}}\left\Vert
P_{Q}u_{1,M_{1}}\right\Vert _{Y^{0}}\left\Vert u_{3,M_{3}}\right\Vert
_{Y^{0}}M_{3}^{\frac{1}{5}}\left\Vert P_{Q_{c}}u_{2,M_{2}}\right\Vert
_{Y^{0}}M^{\frac{7}{5}}\left\Vert g_{M}\right\Vert _{Y^{0}}) \\
&\lesssim &\sum_{\substack{ M_{1},M_{2},M_{3},M  \\ M_{1}\sim M_{2}\geq
M_{3}\geq M}}M_{3}^{\frac{3}{5}}M^{\frac{7}{5}}\left\Vert g_{M}\right\Vert
_{Y^{0}}\left\Vert u_{3,M_{3}}\right\Vert _{Y^{0}}\sum_{Q}\left\Vert
P_{Q}u_{1,M_{1}}\right\Vert _{Y^{0}}\left\Vert
P_{Q_{c}}u_{2,M_{2}}\right\Vert _{Y^{0}}
\end{eqnarray*}%
Applying Cauchy-Schwarz to sum in $Q$, 
\begin{eqnarray*}
&\lesssim &\sum_{\substack{ M_{1},M_{2},M_{3},M  \\ M_{1}\sim M_{2}\geq
M_{3}\geq M}}M_{3}^{\frac{3}{5}}M^{\frac{7}{5}}\left\Vert
u_{1,M_{1}}\right\Vert _{Y^{0}}\left\Vert u_{2,M_{2}}\right\Vert
_{Y^{0}}\left\Vert u_{3,M_{3}}\right\Vert _{Y^{0}}\left\Vert
g_{M}\right\Vert _{Y^{0}} \\
&\lesssim &\sum_{\substack{ M_{1},M_{2}  \\ M_{1}\sim M_{2}}}%
M_{2}M_{1}^{-1}\left\Vert u_{1,M_{1}}\right\Vert _{Y^{-1}}\left\Vert
u_{2,M_{2}}\right\Vert _{Y^{1}}\sum_{\substack{ M_{3},M  \\ M_{1}\sim
M_{2}\geq M_{3}\geq M}}M_{3}^{-\frac{2}{5}}M^{\frac{2}{5}}\left\Vert
u_{3,M_{3}}\right\Vert _{Y^{1}}\left\Vert g_{M}\right\Vert _{Y^{1}}
\end{eqnarray*}%
We are done by Schur's test.

\bigskip

\paragraph{Proof of (\protect\ref{eqn:T4UVTrilinearEstimateLow})}

We reuse the set up in the proof of (\ref{eqn:T4UVTrilinearEstimateHigh}).
However, due to the symmetry assumption $M_{1}\geqslant M_{2}\geqslant M_{3}$
on the frequencies in the proof of (\ref{eqn:T4UVTrilinearEstimateHigh}), we
cannot simply assume $P_{\leqslant M_{0}}$ lands on $u_{2}$. The worst /
least gain case here would be $u_{1}$ is still put in $Y^{-1}$ while $%
P_{\leqslant M_{0}}$ is applied to $u_{3}$. Thus we will prove estimate (\ref%
{eqn:T4UVTrilinearEstimateLow}) subject to the extra localization that $%
P_{\leqslant M_{0}}$ is applied on $u_{3}$. By symmetry in $M_{2}$ and $M$,
it suffices to take care of Case A: $M_{1}\sim M_{2}\geqslant M_{3}\geqslant
M$ and Case B: $M_{1}\sim M_{2}\geqslant M\geqslant M_{3}$. We will get a $%
T^{\frac{1}{4}}M_{0}^{\frac{3}{5}}$ in Case A and a $T^{\frac{1}{7}}M_{0}^{%
\frac{3}{7}}$ in Case B. Since (\ref{eqn:T4UVTrilinearEstimateLow}) is
nowhere near optimal and we just need it to hold with some powers of $T$ and 
$M_{0}$, there is no need to match these powers or pursue the best power in
these cases.

\subparagraph{Case A of (\protect\ref{eqn:T4UVTrilinearEstimateLow}): $%
M_{1}\sim M_{2}\geqslant M_{3}\geqslant M$}

Decompose the $M_{1}$ and $M_{2}$ dyadic spaces into $M_{3}$ size cubes, 
\begin{eqnarray*}
I_{M_{1},M_{2},M_{3},M} &\leqslant &\sum_{Q}\left\Vert
P_{Q}u_{1,M_{1}}P_{Q_{c}}u_{2,M_{2}}\left( P_{\leqslant
M_{0}}u_{3,M_{3}}\right) g_{M}\right\Vert _{L_{t,x}^{1}} \\
&\leqslant &\sum_{Q}\left\Vert P_{Q}u_{1,M_{1}}g_{M}\right\Vert
_{L_{t,x}^{2}}\left\Vert P_{Q_{C}}u_{2,M_{2}}\right\Vert
_{L_{t,x}^{4}}\left\Vert P_{\leqslant M_{0}}u_{3,M_{3}}\right\Vert
_{L_{t,x}^{4}}
\end{eqnarray*}%
where 
\begin{equation*}
\left\Vert P_{\leqslant M_{0}}u_{3,M_{3}}\right\Vert _{L_{t,x}^{4}}\leqslant
T^{\frac{1}{4}}M_{0}^{\frac{3}{5}}M_{3}^{\frac{2}{5}}\left\Vert P_{\leqslant
M_{0}}u_{3,M_{3}}\right\Vert _{L_{t}^{\infty }L_{x}^{2}}\lesssim T^{\frac{1}{%
4}}M_{0}^{\frac{3}{5}}M_{3}^{\frac{2}{5}}\left\Vert P_{\leqslant
M_{0}}u_{3,M_{3}}\right\Vert _{Y^{0}}
\end{equation*}%
Use (\ref{eqn:T^4 p>3 u-v str}) and (\ref{eqn:T^4 p>3 u-v str with Galinean}%
), 
\begin{equation*}
I_{M_{1},M_{2},M_{3},M}\lesssim T^{\frac{1}{4}}M_{0}^{\frac{3}{5}%
}\sum_{Q}\left( M\left\Vert P_{Q}u_{1,M_{1}}\right\Vert _{Y^{0}}\left\Vert
g_{M}\right\Vert _{Y^{0}}\right) M_{3}^{\frac{1}{2}}\left\Vert
P_{Q_{C}}u_{2,M_{2}}\right\Vert _{Y^{0}}M_{3}^{\frac{2}{5}}\left\Vert
P_{\leqslant M_{0}}u_{3,M_{3}}\right\Vert _{Y^{0}}
\end{equation*}%
Note that, in the above, we actually used a bilinear estimate for the 1st
factor but did not record or use the bilinear gain factor. Cauchy-Schwarz to
sum in $Q$, 
\begin{equation*}
I_{M_{1},M_{2},M_{3},M}\lesssim T^{\frac{1}{4}}M_{0}^{\frac{3}{5}}\left\Vert
u_{1,M_{1}}\right\Vert _{Y^{0}}\left\Vert g_{M}\right\Vert
_{Y^{1}}\left\Vert u_{2,M_{2}}\right\Vert _{Y^{0}}M_{3}^{\frac{9}{10}%
}\left\Vert P_{\leqslant M_{0}}u_{3,M_{3}}\right\Vert _{Y^{0}}
\end{equation*}%
Thus, summing in $M$ non-optimally gives 
\begin{eqnarray*}
I_{1A} &\lesssim &T^{\frac{1}{4}}M_{0}^{\frac{3}{5}}\left\Vert g\right\Vert
_{Y^{1}}\sum_{\substack{ M_{1},M_{2},M_{3}  \\ M_{1}\sim M_{2}\geqslant
M_{3} }}\left\Vert u_{1,M_{1}}\right\Vert _{Y^{0}}\left\Vert
u_{2,M_{2}}\right\Vert _{Y^{0}}M_{3}^{\frac{9}{10}}\log M_{3}\left\Vert
P_{\leqslant M_{0}}u_{3,M_{3}}\right\Vert _{Y^{0}} \\
&\lesssim &T^{\frac{1}{4}}M_{0}^{\frac{3}{5}}\left\Vert g\right\Vert
_{Y^{1}}\sum_{\substack{ M_{1},M_{2},M_{3}  \\ M_{1}\sim M_{2}\geqslant
M_{3} }}\left\Vert u_{1,M_{1}}\right\Vert _{Y^{0}}\left\Vert
u_{2,M_{2}}\right\Vert _{Y^{0}}\frac{M_{3}^{\frac{9}{10}}\log M_{3}}{M_{3}}%
\left\Vert P_{\leqslant M_{0}}u_{3,M_{3}}\right\Vert _{Y^{1}}
\end{eqnarray*}%
Again, summing in $M_{3}$ non-optimally and swapping a derivative between $%
u_{1}$ and $u_{2}$ give 
\begin{eqnarray*}
I_{1A} &\lesssim &T^{\frac{1}{4}}M_{0}^{\frac{3}{5}}\left\Vert g\right\Vert
_{Y^{1}}\left\Vert P_{\leqslant M_{0}}u_{3}\right\Vert _{Y^{1}}\sum 
_{\substack{ M_{1},M_{2}  \\ M_{1}\sim M_{2}}}\left\Vert
u_{1,M_{1}}\right\Vert _{Y^{-1}}\left\Vert u_{2,M_{2}}\right\Vert _{Y^{1}} \\
&\lesssim &T^{\frac{1}{4}}M_{0}^{\frac{3}{5}}\left\Vert u_{1}\right\Vert
_{Y^{-1}}\left\Vert u_{2}\right\Vert _{Y^{1}}\left\Vert P_{\leqslant
M_{0}}u_{3}\right\Vert _{Y^{1}}\left\Vert g\right\Vert _{Y^{1}}
\end{eqnarray*}

\subparagraph{Case B of (\protect\ref{eqn:T4UVTrilinearEstimateLow}): $%
M_{1}\sim M_{2}\geqslant M\geqslant M_{3}$}

Sum $M_{3}$ up first, we then consider 
\begin{equation*}
I_{M_{1},M_{2},M}=\iint_{x,t}u_{1,M_{1}}u_{2,M_{2}}\left( P_{\leqslant
M}P_{\leqslant M_{0}}u_{3}\right) g_{M}dxdt\text{.}
\end{equation*}
Decompose the $M_{1}$ and $M_{2}$ dyadic spaces into $M$ size cubes, 
\begin{eqnarray*}
I_{M_{1},M_{2},M} &\leqslant &\sum_{Q}\left\Vert
P_{Q}u_{1,M_{1}}P_{Q_{c}}u_{2,M_{2}}\left( P_{\leqslant M}P_{\leqslant
M_{0}}u_{3}\right) g_{M}\right\Vert _{L_{t,x}^{1}} \\
&\leqslant &\sum_{Q}\left\Vert P_{Q}u_{1,M_{1}}\right\Vert _{L_{t,x}^{\frac{%
7 }{2}}}\left\Vert P_{Q_{C}}u_{2,M_{2}}\right\Vert _{L_{t,x}^{\frac{7}{2}
}}\left\Vert P_{\leqslant M}P_{\leqslant M_{0}}u_{3}\right\Vert
_{L_{t,x}^{7}}\left\Vert g_{M}\right\Vert _{L_{t,x}^{\frac{7}{2}}}
\end{eqnarray*}
where 
\begin{eqnarray*}
\left\Vert P_{\leqslant M}P_{\leqslant M_{0}}u_{3}\right\Vert _{L_{t,x}^{7}}
&\leqslant &T^{\frac{1}{7}}\left\Vert P_{\leqslant M}P_{\leqslant
M_{0}}u_{3}\right\Vert _{L_{x}^{7}} \\
&\lesssim &T^{\frac{1}{7}}M_{0}^{\frac{3}{7}}\left\Vert P_{\leqslant
M}P_{\leqslant M_{0}}u_{3}\right\Vert _{L_{t}^{\infty }H_{x}^{1}} \\
&\lesssim &T^{\frac{1}{7}}M_{0}^{\frac{3}{7}}\left\Vert P_{\leqslant
M}P_{\leqslant M_{0}}u_{3}\right\Vert _{Y^{1}}\text{.}
\end{eqnarray*}
Apply (\ref{eqn:T^4 p>3 u-v str}) and (\ref{eqn:T^4 p>3 u-v str with
Galinean}), 
\begin{equation*}
I_{M_{1},M_{2},M}\lesssim T^{\frac{1}{7}}M_{0}^{\frac{3}{7}}\left\Vert
P_{\leqslant M}P_{\leqslant M_{0}}u_{3}\right\Vert _{Y^{1}}\sum_{Q}M^{\frac{%
2 }{7}}\left\Vert P_{Q}u_{1,M_{1}}\right\Vert _{Y^{0}}M^{\frac{2}{7}
}\left\Vert P_{Q_{C}}u_{2,M_{2}}\right\Vert _{Y^{0}}M^{\frac{2}{7}
}\left\Vert g_{M}\right\Vert _{Y^{0}}
\end{equation*}
Apply Cauchy-Schwarz to sum in $Q,$ 
\begin{equation*}
I_{M_{1},M_{2},M}\lesssim T^{\frac{1}{7}}M_{0}^{\frac{3}{7}}\left\Vert
P_{\leqslant M}P_{\leqslant M_{0}}u_{3}\right\Vert _{Y^{1}}M^{-\frac{1}{7}
}\left\Vert u_{1,M_{1}}\right\Vert _{Y^{0}}\left\Vert u_{2,M_{2}}\right\Vert
_{Y^{0}}\left\Vert g_{M}\right\Vert _{Y^{1}}
\end{equation*}
Thus, swapping a derivative between $u_{1}$ and $u_{2}$ gives 
\begin{equation*}
I_{1B}\lesssim T^{\frac{1}{7}}M_{0}^{\frac{3}{7}}\left\Vert P_{\leqslant
M_{0}}u_{3}\right\Vert _{Y^{1}}\sum_{\substack{ M_{1},M_{2},M  \\ M_{1}\sim
M_{2}\geqslant M}}M^{-\frac{1}{7}}\left\Vert u_{1,M_{1}}\right\Vert
_{Y^{-1}}\left\Vert u_{2,M_{2}}\right\Vert _{Y^{1}}\left\Vert
g_{M}\right\Vert _{Y^{1}}
\end{equation*}
Burning that $\frac{1}{7}$-derivative to sum in $M$ and then applying
Cauchy-Schwarz in $M_{1}$, we have 
\begin{equation*}
I_{1B}\lesssim T^{\frac{1}{7}}M_{0}^{\frac{3}{7}}\left\Vert P_{\leqslant
M_{0}}u_{3}\right\Vert _{Y^{1}}\left\Vert u_{1}\right\Vert
_{Y^{-1}}\left\Vert u_{2}\right\Vert _{Y^{1}}\left\Vert g\right\Vert _{Y^{1}}
\end{equation*}
as needed.

\section{Uniqueness for GP Hierarchy (\protect\ref{Hierarchy:T^4 cubic GP})
and the Proof of Theorem \protect\ref{THM:MainUniquenessTHM} - Set Up\label%
{Sec:GP Uniqueness}}

\begin{theorem}
\label{Thm:GP uniqueness}Let $\Gamma =\left\{ \gamma ^{(k)}\right\} \in
\oplus _{k\geqslant 1}C\left( \left[ 0,T_{0}\right] ,\mathcal{L}%
_{k}^{1}\right) $ be a solution to (\ref{Hierarchy:T^4 cubic GP}) in $\left[
0,T_{0}\right] $ in the sense that

(a) $\Gamma $ is admissible in the sense of Definition \ref{Def:Admissible}.

(b) $\Gamma $ satisfies the kinetic energy condition that $\exists C_{0}>0$
s.t. 
\begin{equation*}
\sup_{t\in \left[ 0,T_{0}\right] }\left( \prod\limits_{j=1}^{k}\left\langle
\nabla _{x_{j}}\right\rangle \right) \gamma ^{(k)}\left( t\right) \left(
\prod\limits_{j=1}^{k}\left\langle \nabla _{x_{j}}\right\rangle \right)
\leqslant C_{0}^{2k}
\end{equation*}
Then there is a threshold $\eta (C_{0})>0$ such that the solution is unique
in $[0,T_{0}]$ provided 
\begin{equation*}
\sup_{t\in \lbrack 0,T_{0}]}\limfunc{Tr}\left(
\prod\limits_{j=1}^{k}P_{>M}^{j}\left\langle \nabla _{x_{j}}\right\rangle
\right) \gamma ^{(k)}(t)\left( \prod\limits_{j=1}^{k}P_{>M}^{j}\left\langle
\nabla _{x_{j}}\right\rangle \right) \leqslant \eta ^{2k},
\end{equation*}%
for some frequency $M$. Our proof shows that $\eta(C_0)$ can be $\left(
100CC_{0}\right) ^{-2}$ with $C$ being a universal constant depending on the 
$U$-$V$ estimate constants and the Sobolev constants. The frequency
threshold $M$ is allowed to depend on $\gamma ^{(k)}$ (the particular
solution under consideration) but must apply uniformly on $[0,T_{0}]$.
\end{theorem}

Here, we have intentionally stated Theorem \ref{Thm:GP uniqueness} before
writing out the definition of admissible / Definition \ref{Def:Admissible}
to bring up readers' attention. For the purpose of only proving Theorem \ref%
{THM:MainUniquenessTHM}, Definition \ref{Def:Admissible} and its companion,
the quantum de Finette theorem / Theorem \ref{Thm:qdF} are, in fact, not
necessary. One could just apply the proof of Theorem \ref{Thm:GP uniqueness}
to the special case 
\begin{eqnarray}
\gamma ^{(k)}\left( t\right) &\equiv &\int_{L^{2}(\mathbb{T}^{4})}\left\vert
\phi \right\rangle \left\langle \phi \right\vert ^{\otimes k}d\mu _{t}(\phi )
\label{eqn:difference of NLS solution} \\
&\equiv &\dprod\limits_{j=1}^{k}u_{1}(t,x_{j})\bar{u}_{1}(t,x_{j}^{\prime
})-\dprod\limits_{j=1}^{k}u_{2}(t,x_{j})\bar{u}_{2}(t,x_{j}^{\prime }), 
\notag
\end{eqnarray}%
where $u_{1}$ and $u_{2}$ are two solutions to (\ref{NLS:T^4 cubic NLS}) and 
$\mu _{t}$ is the signed measure $\delta _{u_{1}}-$ $\delta _{u_{2}}$ on $%
L^{2}(\mathbb{T}^{4})$, to get that the difference is zero for all $k$ and
obtain a uniqueness theorem which is solely about solutions to (\ref{NLS:T^4
cubic NLS}) and is enough to conclude Theorems \ref{THM:MainUniquenessTHM}.
Readers unfamiliar with Theorem \ref{Thm:qdF} could first skip Definition %
\ref{Def:Admissible} and Theorem \ref{Thm:qdF}, put (\ref{eqn:difference of
NLS solution}) in the place of (\ref{eqn:qdF Rep for zero initial solution}%
), get to know how the GP hierarchy is involved and then come back to
Definition \ref{Def:Admissible} and Theorem \ref{Thm:qdF}. Once one
understands the role of the GP hierarchy in the proof, it is easy to see
that, due to Theorem \ref{Thm:qdF}, the more general theorem / Theorem \ref%
{Thm:GP uniqueness} costs nothing more and the origin of the current scheme
of proving NLS uniqueness using GP hierarchies is indeed Theorem \ref%
{Thm:qdF} as mentioned in the introduction. In fact, Theorem \ref{Thm:GP
uniqueness}, implies the following corollary.

\begin{corollary}
\label{Cor:weak uniqueness regarding NLS}Given an intial datum $u_{0}\in
H^{1}(\mathbb{T}^{4})$, there is at most one $C\left( \left[ 0,T_{0}\right]
,H_{x,\mathnormal{weak}}^{1}\right) $ solution $u$ to (\ref{NLS:T^4 cubic
NLS}) on $\mathbb{T}^{4}$ satisfying the following two properties:

(1) There is a $C_{0}>0$ such that 
\begin{equation*}
\sup_{t\in \lbrack 0,T_{0}]}\left\Vert u(t)\right\Vert _{H^{1}}\leqslant
C_{0}.
\end{equation*}

(2) There is some frequency $M$ such that 
\begin{equation}
\sup_{t\in \lbrack 0,T_{0}]}\Vert \nabla P_{\geqslant M}u(t)\Vert
_{L_{x}^{2}}\leqslant \eta ,  \label{eqn:utfl weak}
\end{equation}%
for the threshold $\eta (C_{0})>0$ concluded in Theorem \ref{Thm:GP
uniqueness}.
\end{corollary}

Corollary \ref{Cor:weak uniqueness regarding NLS}, an unclassified
uniqueness theorem, seems to be stronger than the unconditional uniqueness
theorem, Theorem \ref{THM:MainUniquenessTHM}, as it concludes uniqueness in
a larger class of solutions. We wonder if there could be a more detailed
classification regarding the word \textquotedblleft unconditional
uniqueness" at the critical regularity.

Theorem \ref{THM:MainUniquenessTHM}, follows from Theorem \ref{Thm:GP
uniqueness} and the following lemma.

\begin{lemma}
\label{Cor:UTFLforNLS}\footnote{%
The proof of Lemma \ref{Cor:UTFLforNLS} uses only compactness and is much
simpler than \cite[Theorem A.2]{CJInvent}.}$u$ is a $%
C_{[0,T_{0}]}^{0}H_{x}^{1}\cap \dot{C}_{[0,T_{0}]}^{1}H_{x}^{-1}$ solution
of (\ref{NLS:T^4 cubic NLS}) if and only if $u\ $is a $C_{[0,T_{0}]}^{0}H_{x,%
\mathnormal{weak}}^{1}\cap \dot{C}_{[0,T_{0}]}^{1}H_{x,\mathnormal{weak}%
}^{-1}$ solution and satisfies \emph{uniform in time frequency localization}
(UTFL), that is, for each $\varepsilon >0$ there exists $M(\varepsilon )$
such that 
\begin{equation}
\Vert \nabla P_{\geqslant M(\varepsilon )}u\Vert _{L_{[0,T_{0}]}^{\infty
}L_{x}^{2}}\leq \varepsilon  \label{E:UTFL01}
\end{equation}
\end{lemma}

\begin{proof}
Postponed to $\S \ref{Sec:NLS UTFL}$. We remark that (\ref{E:UTFL01})
implies (\ref{eqn:utfl weak}), but the converse is not true. That is,
Corollary \ref{Cor:weak uniqueness regarding NLS} implies Theorem \ref%
{THM:MainUniquenessTHM}, the unconditional uniqueness theorem, but the type
of uniqueness concluded in Corollary \ref{Cor:weak uniqueness regarding NLS}
and also Theorem \ref{Thm:GP uniqueness} is unclassified.
\end{proof}

We can start the proof of Theorem \ref{Thm:GP uniqueness} now. We set up
some notations first. We rewrite (\ref{Hierarchy:T^4 cubic GP}) in Duhamel
form 
\begin{equation}
\gamma ^{(k)}(t_{k})=U^{(k)}(t_{k})\gamma _{0}^{(k)}\mp
i\int_{0}^{t_{k}}U^{(k)}(t_{k}-t_{k+1})B^{(k+1)}\left( \gamma
^{(k+1)}(t_{k+1})\right) dt_{k+1}  \label{hierarchy:GP in integral form}
\end{equation}%
where $U^{(k)}(t)=\prod\limits_{j=1}^{k}e^{it\left( \Delta _{x_{j}}-\Delta
_{x_{j}^{\prime }}\right) }$ and 
\begin{eqnarray*}
B^{(k+1)}\left( \gamma ^{(k+1)}\right) &\equiv
&\sum_{j=1}^{k}B_{j,k+1}\left( \gamma ^{(k+1)}\right) \\
&\equiv &\sum_{j=1}^{k}(B_{j,k+1}^{+}-B_{j,k+1}^{-})\left( \gamma
^{(k+1)}\right) \\
&\equiv &\sum_{j=1}^{k}\limfunc{Tr}\nolimits_{k+1}\delta
(x_{j}-x_{k+1})\gamma ^{(k+1)}-\gamma ^{(k+1)}\delta (x_{j}-x_{k+1}).
\end{eqnarray*}%
\newline
In the above, products are interpreted as the compositions of operators. For
example, in kernels, 
\begin{eqnarray*}
&&\left( \limfunc{Tr}\nolimits_{k+1}\delta (x_{1}-x_{k+1})\gamma
^{(k+1)}\right) \left( \mathbf{x}_{k},\mathbf{x}_{k}^{\prime }\right) \\
&=&\int \delta (x_{1}-x_{k+1})\gamma ^{(k+1)}(\mathbf{x}_{k},x_{k+1};\mathbf{%
\ x}_{k}^{\prime },x_{k+1})dx_{k+1}.
\end{eqnarray*}

We will prove that if $\Gamma _{1}=\left\{ \gamma _{1}^{(k)}\right\} $ and $%
\Gamma _{2}=\left\{ \gamma _{2}^{(k)}\right\} $ are two solutions to (\ref%
{hierarchy:GP in integral form}), subject to the same initial datum and (a)
- (b) in Theorem \ref{Thm:GP uniqueness}, then $\Gamma =\left\{ \gamma
^{(k)}=\gamma _{1}^{(k)}-\gamma _{2}^{(k)}\right\} $ is identically zero.
Note that, because (\ref{hierarchy:GP in integral form}) is linear, $\Gamma $
is a solution to (\ref{hierarchy:GP in integral form}). We will start using
a representation of $\Gamma $ given by the quantum de Finette theorem /
Theorem \ref{Thm:qdF}. To this end, we hereby define admissibility.

\begin{definition}[\protect\cite{CHPS}]
\label{Def:Admissible}A nonnegative trace class symmetric operators sequence 
$\Gamma =\left\{ \gamma ^{(k)}\right\} \in \oplus _{k\geqslant 1}C\left( %
\left[ 0,T\right] ,\mathcal{L}_{k}^{1}\right) $, is called admissible if for
all $k$, one has 
\begin{equation*}
\limfunc{Tr}\gamma ^{(k)}=1\text{, }\gamma ^{(k)}=\limfunc{Tr}%
\nolimits_{k+1}\gamma ^{(k+1)}.
\end{equation*}%
Here, a trace class operator is called symmetry, if, written in kernel form, 
\begin{eqnarray*}
\gamma ^{(k)}(\mathbf{x}_{k};\mathbf{x}_{k}^{\prime }) &=&\overline{\gamma
^{(k)}(\mathbf{x}_{k}^{\prime };\mathbf{x}_{k}),} \\
\gamma ^{(k)}(x_{1},...,x_{k};x_{1}^{\prime },...,x_{k}^{\prime }) &=&\gamma
^{(k)}(x_{\sigma \left( 1\right) },...,x_{\sigma \left( k\right) };x_{\sigma
\left( 1\right) }^{\prime },...,x_{\sigma \left( k\right) }^{\prime }),
\end{eqnarray*}%
for all $\sigma \in S_{k}$, the permutation group of $k$ elements.
\end{definition}

\begin{theorem}[quantum de Finette Theorem \protect\cite{CHPS,LMR}]
\label{Thm:qdF}Under assumption (a), there exists a probability measure $%
d\mu _{t}(\phi )$ supported on the unit sphere of $L^{2}(\mathbb{T}^{4})$
such that 
\begin{equation*}
\gamma ^{(k)}(t)=\int \left\vert \phi \right\rangle \left\langle \phi
\right\vert ^{\otimes k}d\mu _{t}(\phi )\text{.}
\end{equation*}
\end{theorem}

By Theorem \ref{Thm:qdF}, $\exists d\mu _{1,t}$ and $d\mu _{2,t}$
representing the two solutions $\Gamma _{1}$ and $\Gamma _{2}$. The same
Chebyshev argument as in \cite[Lemma 4.5]{CHPS} turns the assumptions in
Theorem \ref{Thm:GP uniqueness} to the property that $d\mu _{j,t}$ is
supported in the set 
\begin{equation}
S=\{\phi \in \mathbb{S}\left( L^{2}(\mathbb{T}^{4})\right) :\left\Vert
P_{>M}\left\langle \nabla \right\rangle \phi \right\Vert _{L^{2}}\leqslant
\varepsilon \}\cap \{\phi \in \mathbb{S}\left( L^{2}(\mathbb{T}^{4})\right)
:\left\Vert \phi \right\Vert _{H^{1}}\leqslant C_{0}\}.
\label{set:spt of qdF Measure}
\end{equation}%
That is, let the signed measure $d\mu _{t}=d\mu _{1,t}-d\mu _{2,t}$, we have 
\begin{equation}
\gamma ^{(k)}(t_{k})=\left( \gamma _{1}^{(k)}-\gamma _{2}^{(k)}\right)
(t_{k})=\int \left\vert \phi \right\rangle \left\langle \phi \right\vert
^{\otimes k}d\mu _{t_{k}}(\phi )
\label{eqn:qdF Rep for zero initial solution}
\end{equation}%
and $d\mu _{t_{k}}$ is supported in the set $S$ defined in (\ref{set:spt of
qdF Measure}).

So our task of establishing Theorem \ref{Thm:GP uniqueness} is now
transformed into proving the solution is zero if the solution takes the form
(\ref{eqn:qdF Rep for zero initial solution}) and is subject to zero initial
datum. It suffices to prove $\gamma ^{(1)}=0$ as the proof is the same for
the general $k$ case. The proof involves coupling (\ref{hierarchy:GP in
integral form}) multiple times. To this end, we plug in zero initial datum,
set the ``$\mp i$" in (\ref{hierarchy:GP in integral form}) to be $1$ so
that we do not need track its power and rewrite (\ref{hierarchy:GP in
integral form}) as 
\begin{equation}
\gamma ^{(k)}(t_{k})=\int_{0}^{t_{k}}U^{(k)}(t_{k}-t_{k+1})B^{(k+1)}\left(
\gamma ^{(k+1)}(t_{k+1})\right) dt_{k+1}
\label{hierarchy:GP in integral form (used)}
\end{equation}

Define 
\begin{eqnarray*}
&&J^{(k+1)}(f^{(k+1)})(t_{1},\underline{t}_{k+1}) \\
&=&U^{(1)}(t_{1}-t_{2})B^{(2)}U^{(2)}(t_{2}-t_{3})B^{(3)}...U^{(k)}(t_{k}-t_{k+1})B^{(k+1)}f^{(k+1)}(t_{k+1})
\end{eqnarray*}%
with $\underline{t}_{k+1}=\left( t_{2},t_{3},...,t_{k},t_{k+1}\right) $. We
can then write 
\begin{equation*}
\gamma
^{(1)}(t_{1})=\int_{0}^{t_{1}}\int_{0}^{t_{2}}...\int_{0}^{t_{k}}J^{(k+1)}(%
\gamma ^{(k+1)})(t_{1},\underline{t}_{k+1})d\underline{t}_{k+1}\text{,}
\end{equation*}%
after iterating (\ref{hierarchy:GP in integral form (used)}) $k$ times. To
estimate $\gamma ^{(1)}$, we first use the Klainerman-Machedon board game%
\footnote{%
As mentioned before, we actually need an extended Klainerman-Machedon board
game, we do so in \S \ref{Sec:ElaboratedKMBoardGame}.} to reduce the number
of summands inside $\gamma ^{(1)},$ which is $k!2^{k}$ at the moment, by
combining them.

\begin{lemma}[Klainerman-Machedon board game \protect\cite{KM}]
\label{lemma:Klainerman-MachedonBoardGame}One can express 
\begin{equation*}
\int_{0}^{t_{1}}\int_{0}^{t_{2}}...%
\int_{0}^{t_{k+1}}J^{(k+1)}(f^{(k+1)})(t_{1},\underline{t}_{k+1})d\underline{%
t}_{k+1}
\end{equation*}%
as a sum of at most $4^{k}$ terms of the form 
\begin{equation*}
\int_{D_{m}}J_{\mu _{m}}^{(k+1)}(f^{(k+1)})(t_{1},\underline{t}_{k+1})d%
\underline{t}_{k+1},
\end{equation*}%
or in other words, 
\begin{equation}
\int_{0}^{t_{1}}\int_{0}^{t_{2}}...%
\int_{0}^{t_{k+1}}J^{(k+1)}(f^{(k+1)})(t_{1},\underline{t}_{k+1})d\underline{%
t}_{k+1}=\sum_{m}\int_{D_{m}}J_{\mu _{m}}^{(k+1)}(f^{(k+1)})(t_{1},%
\underline{t}_{k+1})d\underline{t}_{k+1}.  \label{eqn:KMBoardGameSum}
\end{equation}%
Here, $D_{m}$ is a subset of $[0,t_{1}]^{k}$, depending on $\mu _{m}$; $%
\left\{ \mu _{m}\right\} $ are a set of maps from $\{2,\ldots ,k+1\}$ to $%
\{1,\ldots ,k\}$ satisfying $\mu _{m}(2)=1$ and $\mu _{m}(l)<l$ for all $l,$
and 
\begin{eqnarray*}
J_{\mu _{m}}^{(k+1)}(f^{(k+1)})(t_{1},\underline{t}_{k+1})
&=&U^{(1)}(t_{1}-t_{2})B_{1,2}U^{(2)}(t_{2}-t_{3})B_{\mu _{m}(3),3}\cdots \\
&&\cdots U^{(k)}(t_{k}-t_{k+1})B_{\mu _{m}(k+1),k+1}(f^{(k+1)})(t_{1}).
\end{eqnarray*}
\end{lemma}

Using Lemma \ref{lemma:Klainerman-MachedonBoardGame}, to estimate $\gamma
^{(1)}$, it suffices to deal with a summand in the right hand side of (\ref%
{eqn:KMBoardGameSum}) 
\begin{equation*}
\int_{D_{m}}J_{\mu _{m}}^{(k+1)}(\gamma ^{(k+1)})(t_{1},\underline{t}_{k+1})d%
\underline{t}_{k+1}
\end{equation*}%
at the expense of a $4^{k}$. Since $B_{j,k+1}=B_{j,k+1}^{+}-B_{j,k+1}^{-}$, $%
J_{\mu _{m}}^{(k+1)}(\gamma ^{(k+1)})$ is but another sum. Thus, by paying
an extra $2^{k}$, we can just estimate a typical term 
\begin{equation}
\int_{D_{m}}J_{\mu _{m},sgn}^{(k+1)}(\gamma ^{(k+1)})(t_{1},\underline{t}%
_{k+1})d\underline{t}_{k+1}  \label{eqn:a typical term in uniqueness}
\end{equation}%
where 
\begin{eqnarray}
J_{\mu _{m},sgn}^{(k+1)}(f^{(k+1)})(t_{1},\underline{t}_{k+1})
&=&U^{(1)}(t_{1}-t_{2})B_{1,2}^{sgn(2)}U^{(2)}(t_{2}-t_{3})B_{\mu
_{m}(3),3}^{sgn(3)}\cdots  \label{E:JSgnDef} \\
&&\cdots U^{(k)}(t_{k}-t_{k+1})B_{\mu
_{m}(k+1),k+1}^{sgn(k+1)}(f^{(k+1)})(t_{k+1}).  \notag
\end{eqnarray}%
with $sgn$ meaning the signature array $(sgn(2),...,sgn(k+1))$ and $%
B_{j,k+1}^{sgn(k+1)}$ stands for $B_{j,k+1}^{+}$ or $B_{j,k+1}^{-}$
depending on the sign of the $(k+1)$-th signature element. The estimate of ( %
\ref{eqn:a typical term in uniqueness}) is given by the following
proposition.

\begin{proposition}
\label{Prop:KeyUniquenessEstimate} 
\begin{eqnarray*}
&&\left\Vert \left\langle \nabla _{x_{1}}\right\rangle ^{-1}\left\langle
\nabla _{x_{1}^{\prime }}\right\rangle ^{-1}\int_{D_{m}}J_{\mu
_{m},sgn}^{(k+1)}(\gamma ^{(k+1)})(t_{1},\underline{t}_{k+1})d\underline{t}
_{k+1}\right\Vert _{L_{t_{1}}^{\infty }L_{x_{1},x_{1}^{\prime }}^{2}} \\
&\leqslant &2TC_{0}^{2}\left( CC_{0}^{3}T^{\frac{1}{7}}M_{0}^{\frac{3}{5}
}+CC_{0}^{2}\varepsilon \right) ^{\frac{2}{3}k}
\end{eqnarray*}
\end{proposition}

\begin{proof}
See \S \ref{Sec:GP Uniqueness-2}.
\end{proof}

Once Proposition \ref{Prop:KeyUniquenessEstimate} is proved, Theorem \ref%
{Thm:GP uniqueness} then follows. In fact, 
\begin{eqnarray*}
&&\left\Vert \left\langle \nabla _{x_{1}}\right\rangle ^{-1}\left\langle
\nabla _{x_{1}^{\prime }}\right\rangle ^{-1}\gamma ^{(1)}(t_{1})\right\Vert
_{L_{t_{1}}^{\infty }L_{x_{1},x_{1}^{\prime }}^{2}} \\
&\leqslant &4^{k}\left\Vert \left\langle \nabla _{x_{1}}\right\rangle
^{-1}\left\langle \nabla _{x_{1}^{\prime }}\right\rangle
^{-1}\int_{D_{m}}J_{\mu _{m}}^{(k+1)}(\gamma ^{(k+1)})(t_{1},\underline{t}%
_{k+1})d\underline{t}_{k+1}\right\Vert _{L_{t_{1}}^{\infty
}L_{x_{1},x_{1}^{\prime }}^{2}} \\
&\leqslant &8^{k}\left\Vert \left\langle \nabla _{x_{1}}\right\rangle
^{-1}\left\langle \nabla _{x_{1}^{\prime }}\right\rangle
^{-1}\int_{D_{m}}J_{\mu _{m},sgn}^{(k+1)}(\gamma ^{(k+1)})(t_{1},\underline{t%
}_{k+1})d\underline{t}_{k+1}\right\Vert _{L_{t_{1}}^{\infty
}L_{x_{1},x_{1}^{\prime }}^{2}} \\
&\leqslant &2TC_{0}^{2}\left( CC_{0}^{3}T^{\frac{1}{7}}M_{0}^{\frac{3}{5}%
}+CC_{0}^{2}\varepsilon \right) ^{\frac{2}{3}k}
\end{eqnarray*}%
Select $\varepsilon $ small enough (The threshold $\eta $ is also determined
here.) so that $CC_{0}^{2}\varepsilon <\frac{1}{4}$ and then select $T$
small enough so that $CC_{0}^{3}T^{\frac{1}{7}}M_{0}^{\frac{3}{5}}$ $<\frac{1%
}{4}$, we then have 
\begin{equation*}
\left\Vert \left\langle \nabla _{x_{1}}\right\rangle ^{-1}\left\langle
\nabla _{x_{1}^{\prime }}\right\rangle ^{-1}\gamma ^{(1)}(t_{1})\right\Vert
_{L_{t_{1}}^{\infty }L_{x_{1},x_{1}^{\prime }}^{2}}\leqslant \left( \frac{1}{%
2}\right) ^{k}\rightarrow 0\text{ as }k\rightarrow \infty \text{.}
\end{equation*}%
We can then bootstrap to fill the whole $\left[ 0,T_{0}\right] $ interval as 
$M$ applies uniformly on $[0,T_{0}].$

Before moving into the proof of Proposition \ref{Prop:KeyUniquenessEstimate}%
, we remark that the extra $2T$ does not imply the estimate is critical or
subcritical, this $T$ actually appears only once. Such a $T$ is due to the
GP hierarchy method instead of scaling because the $dt_{k+1}$ time integral
is not used for any Strichartz type estimates. This one factor of $T$
appeared in the other energy-critical $\mathbb{T}^{3}$ quintic case \cite%
{CJInvent} as well.

\subsection{Proof of Lemma \protect\ref{Cor:UTFLforNLS} / UTFL for NLS\label%
{Sec:NLS UTFL}}

By substituting the equation, we compute 
\begin{align*}
\left\vert \partial _{t}\Vert \nabla P_{\leq M}u\Vert
_{L_{x}^{2}}^{2}\right\vert & =2\left\vert \Im \int P_{\leq M}\nabla u\cdot
P_{\leq M}\nabla (|u|^{2}u)\,dx\right\vert \\
& \leqslant 2\Vert P_{\leq M}\nabla u\Vert _{L^{4}}\Vert P_{\leq M}\nabla
(|u|^{2}u)\Vert _{L^{4/3}} \\
& =2M^{2}\Vert \tilde{P}_{\leq M}u\Vert _{L^{4}}\Vert \tilde{P}_{\leq
M}(|u|^{2}u)\Vert _{L^{4/3}}
\end{align*}%
where, if the symbol associated to $P_{\leq M}$ is $\chi (\xi /M)$, then the
symbol associated to $\tilde{P}_{\leq M}$ is $\tilde{\chi}(\xi /M)$, with $%
\tilde{\chi}(\xi )=\xi \chi (\xi )$. By the $L^{p}\rightarrow L^{p}$
boundedness of the Littlewood-Paley projections (see for example, \cite[
Appendix]{HS2}), 
\begin{equation*}
\left\vert \partial _{t}\Vert \nabla P_{\leq M}u\Vert
_{L_{x}^{2}}^{2}\right\vert \lesssim M^{2}\Vert u\Vert _{L^{4}}^{4}
\end{equation*}%
By Sobolev embedding 
\begin{equation*}
\left\vert \partial _{t}\Vert \nabla P_{\leq M}u\Vert
_{L_{x}^{2}}^{2}\right\vert \lesssim M^{2}\Vert u\Vert _{H^{1}}^{4}
\end{equation*}%
Hence there exists $\delta ^{\prime }>0$ (depending on $M$, $\Vert u\Vert
_{L_{[0,T]}^{\infty }H^{1}}$, and $\varepsilon $) such that for any $%
t_{0}\in \lbrack 0,T]$, it holds that for any $t\in (t_{0}-\delta ^{\prime
},t_{0}+\delta ^{\prime })\cap \lbrack 0,T]$, 
\begin{equation}
\left\vert \Vert \nabla P_{\leq M}u(t)\Vert _{L_{x}^{2}}^{2}-\Vert \nabla
P_{\leq M}u(t_{0})\Vert _{L_{x}^{2}}^{2}\right\vert \leq \tfrac{1}{16}%
\varepsilon ^{2}  \label{E:A}
\end{equation}%
On the other hand, since $u\in C_{[0,T]}^{0}H_{x}^{1}$, for each $t_{0}$,
there exists $\delta ^{\prime \prime }>0$ such that for any $t\in
(t_{0}-\delta ^{\prime \prime },t_{0}+\delta ^{\prime \prime })\cap \lbrack
0,T]$, 
\begin{equation}
\left\vert \Vert \nabla u(t)\Vert _{L_{x}^{2}}^{2}-\Vert \nabla
u(t_{0})\Vert _{L_{x}^{2}}^{2}\right\vert \leq \tfrac{1}{16}\varepsilon ^{2}
\label{E:B}
\end{equation}%
Note that $\delta ^{\prime \prime }$ depends on $u$ itself (or the
\textquotedblleft modulus of continuity\textquotedblright\ of $u$), unlike $%
\delta ^{\prime }$ that depends only on $M$, $\Vert u\Vert
_{L_{[0,T]}^{\infty }H^{1}}$, and $\varepsilon $. Now let $\delta =\min
(\delta ^{\prime },\delta ^{\prime \prime })$. Then by \eqref{E:A} and %
\eqref{E:B} we have that for any $t\in (t_{0}-\delta ,t_{0}+\delta )\cap
\lbrack 0,T]$, 
\begin{equation*}
\left\vert \Vert \nabla P_{>M}u(t)\Vert _{L_{x}^{2}}^{2}-\Vert \nabla
P_{>M}u(t_{0})\Vert _{L_{x}^{2}}^{2}\right\vert \leq \tfrac{1}{4}\varepsilon
^{2}
\end{equation*}%
For each $t\in \lbrack 0,T]$, there exists $M_{t}$ such that 
\begin{equation*}
\Vert \nabla P_{>M_{t}}u(t)\Vert _{L_{x}^{2}}\leq \tfrac{1}{2}\varepsilon
\end{equation*}%
By the above, there exists $\delta _{t}>0$ such that on $(t-\delta
_{t},t+\delta _{t})$, we have 
\begin{equation*}
\Vert \nabla P_{>M_{t}}u\Vert _{L_{(t-\delta _{t},t+\delta _{t})}^{\infty
}L_{x}^{2}}\leq \varepsilon
\end{equation*}%
Here, $\delta _{t}>0$ depends on $u$ and $M_{t}$. The collection of
intervals $(t-\delta _{t},t+\delta _{t})$, as $t$ ranges over $[0,T]$ is an
open cover of $[0,T]$. Let 
\begin{equation*}
(t_{1}-\delta _{t_{1}},t_{1}+\delta _{t_{1}})\,,\ldots (t_{J}-\delta
_{t_{J}},t_{J}+\delta _{t_{J}})
\end{equation*}%
be an open cover of $[0,T]$. Letting 
\begin{equation*}
M=\max (M_{t_{1}},\ldots ,M_{t_{J}}).
\end{equation*}
we have established \eqref{E:UTFL01}.

Now conversely suppose that $u\in C_{[0,T]}^0H_{x,\text{weak}}^1\cap
C_{[0,T]}^1H_{x,\text{weak}}^{-1}$ and $u$ satisfies \eqref{E:UTFL01}. Then
we claim that $u\in C_{[0,T]}^0H_x^1 \cap C_{[0,T]}^1 H_x^{-1}$. Let $t_0\in
[0,T]$ be arbitrary. If $u$ is not strongly continuous at $t_0$, then there
exist $\epsilon>0$ and a sequence $t_k \to t_0$ such that $\|u(t_k) - u(t_0)
\|_{H_x^1} > 2\epsilon$. Then for each $k$, there exists $\phi_k \in
H_x^{-1} $ with $\|\phi_k \|_{H_x^{-1}} \leq 1$ and 
\begin{equation}  \label{E:UTFL02}
|\langle u(t_k) - u(t_0) , \phi_k \rangle| > 2\epsilon
\end{equation}
Get $M$ as in \eqref{E:UTFL01}. Then 
\begin{equation}  \label{E:UTFL03}
|\langle u(t_k)-u(t_0) , P_{>M} \phi_k \rangle | \leq \epsilon
\end{equation}
On the other hand, by the Rellich-Kondrachov compactness theorem, there
exists a subsequence such that $P_{\leq M} \phi_k \to \phi$ in $H_x^{-1}$.
This combined with the assumption that $u$ is weakly continuous implies that 
\begin{equation}  \label{E:UTFL04}
\langle u(t_k) - u(t_0) , P_{\leq M} \phi_k \rangle \to 0
\end{equation}
But \eqref{E:UTFL03} and \eqref{E:UTFL04} contradict \eqref{E:UTFL02}. The
proof that $\partial_t u$ is strongly continuous is similarly
straightforward.

\section{An Extended Klainerman-Machedon Board Game\label%
{Sec:ElaboratedKMBoardGame}}

This section is divided into two main parts. We first provide as a warm up,
in \S \ref{Sec:Find Integration Limits}, a more elaborated proof of the
original Klainerman-Machedon Board game (Lemma \ref%
{lemma:Klainerman-MachedonBoardGame}) which yields the previously unknown
time integration limits in (\ref{eqn:KMBoardGameSum}). We then prove, in \S %
\ref{Sec:Wild1}-\ref{Sec:Wild5}, an extension of Lemma \ref%
{lemma:Klainerman-MachedonBoardGame} which further combines the summands
inside $J^{(k+1)}(f^{(k+1)})(t_{1},\underline{t}_{k+1})$ to enable the
application of $U$-$V$ spaces techniques.

\subsection{A More Elaborated Proof of Lemma \protect\ref%
{lemma:Klainerman-MachedonBoardGame}\label{Sec:Find Integration Limits}}

Let us first give a brief review of the original Klainerman-Machedon (KM)
board game which, since invented, has been used in every paper in which the
analysis of Gross-Pitaevskii hiearchy is involved. Recall the notation of $%
\mu$ in Lemma \ref{lemma:Klainerman-MachedonBoardGame}: $\left\{ \mu
\right\} $ is a set of maps from $\{2,\ldots ,k+1\}$ to $\{1,\ldots ,k\}$
satisfying $\mu(2)=1$ and $\mu(l)<l$ for all $l,$ and 
\begin{eqnarray*}
J_{\mu}^{(k+1)}(f^{(k+1)})(t_{1},\underline{t}_{k+1})
&=&U^{(1)}(t_{1}-t_{2})B_{1,2}U^{(2)}(t_{2}-t_{3})B_{\mu(3),3}\cdots \\
&&\cdots U^{(k)}(t_{k}-t_{k+1})B_{\mu(k+1),k+1}(f^{(k+1)}(t_{k+1}))
\end{eqnarray*}

\begin{example}
An example of $\mu$ when $k=5$ is 
\begin{equation*}
\begin{tabular}{c|ccccc}
$j$ & $2$ & $3$ & $4$ & $5$ & $6$ \\ \hline
$\mu $ & $1$ & $1$ & $3$ & $2$ & $1$%
\end{tabular}
.
\end{equation*}
\end{example}

If $\mu $ satisfies $\mu (j)\leq \mu (j+1)$ for $2\leq j\leq k$ in addition
to $\mu (j)<j$ for all $2\leq j\leq k+1$, then it is in \emph{upper-echelon
form}\footnote{%
This word makes more sense when one uses the matrix / board game
representation of $J_{\mu}^{(k+1)}(f^{(k+1)})$ in \cite{KM}.} as they are
called in \cite{KM}.

Let $\mu $ be a collapsing map as defined above and $\sigma $ a permutation
of $\{2,\ldots ,k+1\}$. A \emph{Klainerman-Machedon acceptable move}, which
we denote $\text{KM}(j,j+1)$, is allowed when $\mu (j)\neq \mu (j+1)$ and $%
\mu(j+1)<j$, and is the following action: $(\mu ^{\prime },\sigma ^{\prime
})=\text{KM}(j,j+1)(\mu ,\sigma )$: 
\begin{align*}
\mu ^{\prime }& =(j,j+1)\circ \mu \circ (j,j+1) \\
\sigma ^{\prime }& =(j,j+1)\circ \sigma
\end{align*}%
A key observation in Klainerman-Machedon \cite{KM} is that if $(\mu ^{\prime
},\sigma ^{\prime })=\text{KM}(j,j+1)(\mu ,\sigma )$ and $f^{(k+1)}$ is a
symmetric density, then 
\begin{equation}
J_{\mu ^{\prime }}^{(k+1)}(f^{(k+1)})(t_{1},{\sigma ^{\prime }}^{-1}(%
\underline{t}_{k+1}))=J_{\mu }^{(k+1)}(f^{(k+1)})(t_{1},{\sigma }^{-1}(%
\underline{t}_{k+1}))  \label{E:KM-den1}
\end{equation}%
where, 
\begin{equation*}
\text{for }\underline{t}_{k+1}=(t_{2},\ldots ,t_{k+1})\text{ we define }%
\sigma ^{-1}(\underline{t}_{k+1})=(t_{\sigma ^{-1}(2)},\ldots ,t_{\sigma
^{-1}(k+1)})
\end{equation*}%
Associated to each $\mu $ and $\sigma $, we define the Duhamel integrals 
\begin{equation}
I(\mu ,\sigma ,f^{(k+1)})(t_{1})=\int_{t_{1}\geq t_{\sigma (2)}\geq \cdots
\geq t_{\sigma (k+1)}}J_{\mu }^{(k+1)}(f^{(k+1)})(t_{1},\underline{t}%
_{k+1})\,d\underline{t}_{k+1}  \label{E:KM-den2}
\end{equation}%
It follows from \eqref{E:KM-den1} that 
\begin{equation*}
I(\mu ^{\prime },\sigma ^{\prime (k+1)})=I(\mu ,\sigma ,f^{(k+1)})
\end{equation*}%
It is clear that we can combine Klainerman-Machedon acceptable moves as
follows: if $\rho $ is a permutation of $\{2,\ldots ,k+1\}$ such that it is
possible to write $\rho $ as a composition of transpositions 
\begin{equation*}
\rho =\tau _{1}\circ \cdots \circ \tau _{r}
\end{equation*}%
for which each operator $\text{KM}(\tau _{j})$ on the right side of the
following is an acceptable action 
\begin{equation*}
\text{KM}(\rho )\overset{\mathrm{def}}{=}\text{KM}(\tau _{1})\circ \cdots
\circ \text{KM}(\tau _{r})
\end{equation*}%
then $\text{KM}(\rho )$, defined by this composition, is acceptable as well.
In this case $(\mu^{\prime },\sigma ^{\prime })=\text{KM}(\rho )(\mu ,\sigma
)$ and 
\begin{align*}
& \mu ^{\prime } = \rho \circ \mu \circ \rho^{-1} \\
&\sigma ^{\prime } =\rho \circ \sigma
\end{align*}%
\eqref{E:KM-den1} and \eqref{E:KM-den2} hold as well. If $\mu $ and $\mu
^{\prime }$ are such that there exists $\rho $ as above for which $(\mu
^{\prime },\sigma ^{\prime })=\text{KM}(\rho )(\mu ,\sigma )$ then we say
that $\mu ^{\prime }$ and $\mu $ are \emph{KM-relatable}. This is an
equivalence relation that partitions the set of collapsing maps into
equivalence classes.

In short, one can describe the KM board game in \cite{KM} which combines the 
$k!$ many terms in $J^{(k+1)}(f^{(k+1)})$ as the following.

\begin{algorithm}[\protect\cite{KM}]
\label{alg:KMOriginal} \leavevmode

\begin{enumerate}
\item Convert each of the $k!$ many $\mu _{\text{in}}^{\prime }s$ in $%
J^{(k+1)}(f^{(k+1)})$, into one of the $\leqslant 4^{k}$ many upper echelon
form $\mu _{\text{out}}$ via acceptable moves, defined in the board game
argument, and at the same time produce an array $\sigma $ which changes the
time integration domain from the simplex 
\begin{equation*}
t_{1}\geqslant t_{2}\geqslant t_{3}\geqslant ...\geqslant t_{k+1},
\end{equation*}%
into the simplex 
\begin{equation*}
t_{1}\geqslant t_{\sigma (2)}\geqslant t_{\sigma (3)}...\geqslant t_{\sigma
\left( k+1\right) }.
\end{equation*}%
Hence, there are $\leqslant 4^{k}$ classes on the right hand side of (\ref%
{eqn:KMBoardGameSum}).

\item For each upper echelon form $\mu _{\text{out}}$, take a union of the
time integration domains of its $\mu _{\text{in}}$'s after the acceptable
moves and use it as the time integration domain for the whole class. Thus,
the integration domain $D_{m}$ on the right hand side of (\ref%
{eqn:KMBoardGameSum}) depends on $\mu _{m}$ and we have successfully
combined $k!$ summands into $\leqslant 4^{k}$ summands.
\end{enumerate}
\end{algorithm}

The key take away in Algorithm \ref{alg:KMOriginal} is that, though very
unobvious, quite a few of the summands in $J^{(k+1)}(f^{(k+1)})$ actually
have the same integrand if one switches the variable labelings in a clever
way. Algorithm \ref{alg:KMOriginal} leaves only one ambiguity, that is, the
time integration domain $D_{m}$, which is, obviously very complicated for
large $k,$ as it is a union of a very large number of simplexes in high
dimension under the action of a proper subset of the permutation group $%
S_{k} $ depending on the integrand. So far, for the analysis of GP
hierarchies on $\mathbb{R}^{d}$/$\mathbb{T}^{d}$, $d\leqslant 3$, knowing $%
D_{m}\subset \left[ 0,1\right] ^{k}$ has been enough as the related $%
L_{t}^{1}H^{s}$ estimates are true. $\mathbb{T}^{4}$ appears to be the first
domain on which one has to know what $D_{m}$ is so that one can, at least,
have a chance to use space-time norms like $X_{s,b}$ and $U$-$V$, as the
related $L_{t}^{1}H^{s}$ estimates fails.

It turns out that $D_{m}$ is, in fact, simple, as we will see. We now
present a more elaborated proof of Lemma \ref%
{lemma:Klainerman-MachedonBoardGame}, in which, $D_{m}$ is computed in a
clear way. Given a $\mu $, and hence a summand inside $J^{(k+1)}(f^{(k+1)})$
, we construct a binary tree with the following algorithm.

\begin{algorithm}
\label{alg:u to tree} \leavevmode

\begin{enumerate}
\item Set counter $j=2$

\item Given $j$, find the next pair of indices $a$ and $b$ so that $a>j$, $%
b>j$ and 
\begin{equation*}
\mu (a)=\mu (j)\text{ and }\mu (b)=j
\end{equation*}
and moreover $a$ and $b$ are the minimal indices for which the above
equalities hold. It is possible that there is no such $a$ and/or no such $b$.

\item At the node $j$, put $a$ as the left child and $b$ as the right child
(if there is no $a$, then the $j$ node will be missing a left child, and if
there is no $b$, then the $j$ node will be missing a right child.)

\item If $j=k+1,$ then stop, otherwise set $j=j+1$ and go to step 2.
\end{enumerate}
\end{algorithm}

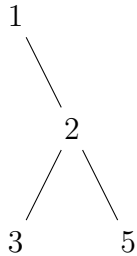
\begin{example}
\label{example:Tree-1}\footnote{%
This simple example is in fact one of the two largest $k=5$ upper echelon
classes in which there are eight $\mu ^{\prime }$s equilvalent to the upper
echelon form.}

\begin{minipage}{0.30\linewidth}
\begin{center}
\begin{tikzpicture} 
\node{$1$}
	child[missing]
	child{node{$2$}
		child{node{$3$}} 
		child{node{$5$}}
	};
\end{tikzpicture}
\end{center}
\end{minipage}
\begin{minipage}{0.70\linewidth}
Let us work with the following example 

\medskip

\begin{center}
\begin{tabular}{c|ccccc}
$j$ & $2$ & $3$ & $4$ & $5$ & $6$ \\ \hline
$\mu _{\text{out}}$ & $1$ & $1$ & $1$ & $2$ & $3$
\end{tabular}
\end{center}

\medskip

We start with $j=2$, and note that $\mu _{\text{out}}(2)=1$ so need to find
minimal $a>2$, $b>2$ such that $\mu (a)=1$ and $\mu (b)=2$. In this case, it
is $a=3$ and $b=5$, so we put those as left and right children of $2$,
respectively, in the tree (shown at left)
\end{minipage}

\bigskip

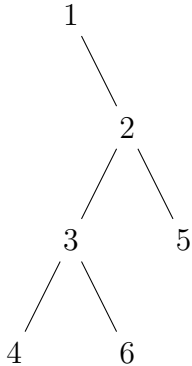
\begin{minipage}{0.30\linewidth}
\begin{center}
\begin{tikzpicture} 
\node{$1$}
	child[missing]
	child{node{$2$}
		child{node{$3$} 
			child{node{$4$}} 
			child{node{$6$}} 
		}
		child{node{$5$}}
	}; 
\end{tikzpicture}
\end{center}
\end{minipage}
\begin{minipage}{0.70\linewidth}
Now we move to $j=3$. Since $\mu _{\text{out}}(3)=1$, we find minimal $a$
and $b$ so that $a>3$, $b>3$ and $\mu (a)=1$ and $\mu (b)=3$. We find that $
a=4$ and $b=6$, so we put these as left and right children of $3$,
respectively, in the tree (shown at left).  Since all indices appear in the tree, it is complete.
\end{minipage}
\end{example}

\begin{definition}
A binary tree is called an \emph{admissible} tree if every child node's
label is strictly larger than its parent node's label.\footnote{%
This is certainly a natural requirement coming from the hierarchy.} For an
admissible tree, we call, the graph of the tree without any labels in its
nodes, the \emph{skeleton} of the tree.
\end{definition}

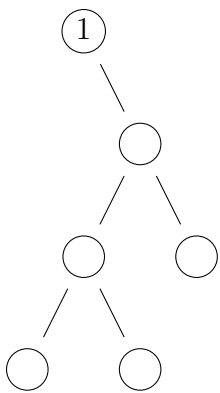
\begin{minipage}{0.30\linewidth}
\begin{center}
\begin{tikzpicture} 
\node{$\encircle{1}$}
	child[missing]
	child{node{$\encircle{\;}$}
		child{node{$\encircle{\;}$} 
			child{node{$\encircle{\;}$}} 
			child{node{$\encircle{\;}$}} 
		}
		child{node{$\encircle{\;}$}}
	}; 
\end{tikzpicture}
\end{center}
\end{minipage}
\begin{minipage}{0.70\linewidth}
For example, the skeleton of the tree in Example \ref{example:Tree-1} is shown at left.

\bigskip

By the hierarchy structure, Algorithm \ref{alg:u to tree}, which produces a
tree from a $\mu $, only produces admissible trees. As we have made a
distinction between left and right children in the algorithm, the procedure
is reversible -- given an admissible binary tree, we can uniquely
reconstruct the $\mu $ that generated it.
\end{minipage}

\begin{algorithm}
\leavevmode
\label{alg:tree to u}

\begin{enumerate}
\item For every right child, $\mu $ maps the child value to the parent value
(i.e. if $f$ is a right child of $d$, then $\mu (f)=d$). Start by filling
these into the $\mu $ table.

\item Fill in the table using that for every left child, $\mu $ maps the
child value to $\mu (\text{parent value})$.
\end{enumerate}
\end{algorithm}

\begin{example}
Suppose we are given the tree

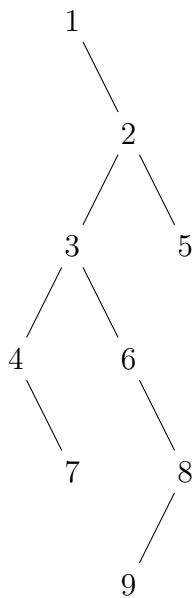
\begin{minipage}{0.375\linewidth}
\begin{center}
\begin{tikzpicture} 
\node{$1$}
child[missing]
child{node{$2$}
	child{node{$3$} 
		child{node{$4$} 
			child[missing]
			child{node{$7$}} 
		}
		child{node{$6$} 
			child[missing] 
			child{node{$8$} 
				child{node{$9$}}
				child[missing] 
			} 
		} 
	} 
	child{node{$5$}}
}; 
\end{tikzpicture}
\end{center}
\end{minipage}
\begin{minipage}{0.575\linewidth}

Using that for every right child, $\mu$ maps the child value to the parent
value, we fill in the following values in the $\mu $ table:

\bigskip

\begin{center}
\begin{tabular}{c|cccccccc}
$j$ & $2$ & $3$ & $4$ & $5$ & $6$ & $7$ & $8$ & $9$ \\ \hline
$\mu $ & $1$ &  &  & $2$ & $3$ & $4$ & $6$ & 
\end{tabular}
\end{center}

\bigskip

Now we employ the left child rule, and note that since $3$ is a left child
of $2$ and $\mu (2)=1$, we must have $\mu (3)=1$, and etc. to recover

\bigskip

\begin{center}
\begin{tabular}{c|cccccccc}
$j$ & $2$ & $3$ & $4$ & $5$ & $6$ & $7$ & $8$ & $9$ \\ \hline
$\mu $ & $1$ & $1$ & $1$ & $2$ & $3$ & $4$ & $6$ & $6$
\end{tabular}
\end{center}
\end{minipage}
\end{example}

One can show that, in the tree representation of $\mu $, an acceptable move
defined in \cite{KM}, is the operation which switches the labels of two
nodes with consecutive labels on an admissible tree provided that the
outcome is still an admissible tree by writing out the related trees on \cite%
[p.180-182]{KM}. For example, interchanging the labeling of 5 and 6 in the
tree in Example \ref{example:Tree-1} is an acceptable move. That is,
acceptable moves in \cite{KM} preserve the tree structures but permute the
labeling under the admissibility requirement. Two collapsing maps $\mu$ and $%
\mu^{\prime }$ are KM-relatable if and only the trees corresponding to $\mu$
and $\mu^{\prime }$ have the same skeleton.

Given $k$, we would like to have the number of different binary tree
structures of $k$ nodes. This number is exactly defined as the Catalan
number and is controlled by $4^{k}$. Hence, we just provided a proof of the
original Klainerman-Machedon board game, neglecting the trees showing the
acceptable moves' effect on a tree.

But let us get to the main \textquotedblleft elaborate\textquotedblright\
part, namely, how to compute $D_{m}$ for a given upper echelon class now. To
this end, we need to define what is an upper echelon form. Though the
requirement $\mu (j)\leq \mu (j+1)$ for $2\leq j\leq k$ is good enough, we
give an algorithm which produces the upper echelon tree given the tree
structure, as the tree representation of an upper echelon form is in fact
labeled in sequential order. See, for example, the tree in Example \ref%
{example:Tree-1}.

\begin{algorithm}
\label{alg:TreeToUpperEchelon}\footnote{%
The difference between the definition of left and right children in
Algorithm \ref{alg:u to tree} makes all the enumeration algorithms in this
paper address left branches first. See also \S \ref{Sec:Wild2} for the
enumeration of the tamed form.} \leavevmode

\begin{enumerate}
\item Given a tree structure with $k$ nodes, label the top node with $2$ and
set counter $j=2.$

\item If $j=k+1$, then stop, otherwise continue.

\item If the node labeled $j$ has a left child, then label that left child
node with $j+1$, set counter $j=j+1$ and go to step (2). If not, continue.

\item In the already labeled nodes which has an empty right child, search
for the node with the smallest label. If such a node can be found, label
that node's empty right child as $j+1$, set counter $j=j+1,$ and go to step
(2). If none of the labeled nodes has an empty right child, then stop.
\end{enumerate}
\end{algorithm}

\begin{definition}
We say $\mu $ is in upper echelon form if $\mu (j)\leq \mu (j+1)$ for $2\leq
j\leq k$ or its corresponding tree given by Algorithm \ref{alg:u to tree}
agrees with the tree with the same skeleton given by Algorithm \ref%
{alg:TreeToUpperEchelon}.
\end{definition}

We define a map $T_{D}$ which maps an upper echelon tree to a time
integration domain / a set of inequality relations by 
\begin{equation}  \label{E:TD-def}
\begin{aligned} T_{D}(\alpha ) = \{\; t_{j}\geqslant t_{k} \; : \; &
j,k\text{ are labels on nodes of }\alpha \\ &\text{such that the }k\text{
node is a child of the }j\text{ node} \; \} \end{aligned}
\end{equation}
where $\alpha $ is an upper echelon tree. We then have the integration
domain as follows.

\begin{proposition}
\label{Prop:KMIntegrationLimits}Given a $\mu _{m}$ in upper echelon form, we
have 
\begin{equation*}
\sum_{\mu \sim \mu _{m}}\int_{t_{1}\geqslant t_{2}\geqslant t_{3}\geqslant
...\geqslant t_{k+1}}J_{\mu_m }^{(k+1)}(f^{(k+1)})(t_{1},\underline{t}%
_{k+1})d \underline{t}_{k+1}=\int_{T_{D}(\mu _{m})}J_{\mu
_{m}}^{(k+1)}(f^{(k+1)})(t_{1},\underline{t}_{k+1})d\underline{t}_{k+1}.
\end{equation*}
Here, $\mu \sim \mu _{m}$ means that $\mu $ is equivalent to $\mu _{m}$
under acceptable moves / the trees representing $\mu $ and $\mu _{m}$ have
the same structure and $T_{D}(\mu _{m})$ is the domain defined in %
\eqref{E:TD-def}.
\end{proposition}

\begin{proof}

We prove by an example as the notation is already heavy. For the general
case, one merely needs to rewrite $\Sigma _{1}$ and $\Sigma _{2}$, to be
defined in this proof. The key is the admissible condition or the simple
requirement that the child must carry a larger lable than the parent.

Recall the upper echelon tree in Example \ref{example:Tree-1}, and denote it
with $\alpha $. Here are all the admissible trees equivalent to $\alpha .$


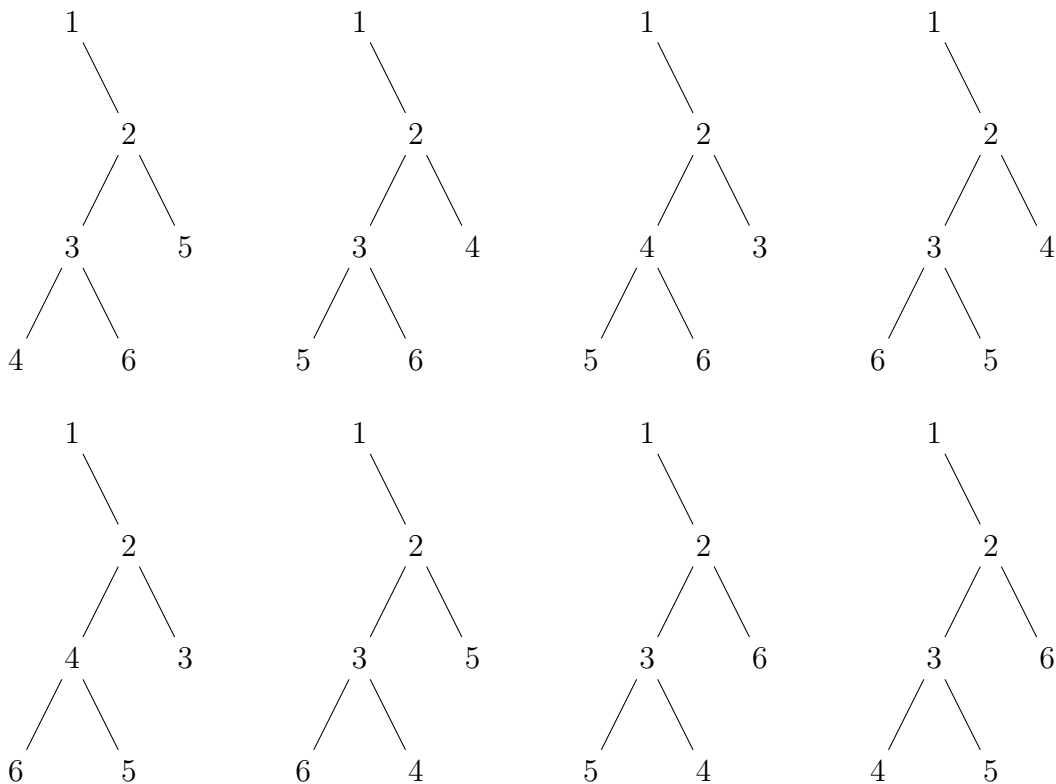
\begin{minipage}{1.45in}
\begin{tikzpicture}
\node{$1$} 
	child[missing] 
	child{node{$2$}
		child{node{$3$}
			child{node{$4$}}
			child{node{$6$}}
		}
		child{node{$5$}}
	};
\end{tikzpicture}
\end{minipage}
\begin{minipage}{1.45in}
\begin{tikzpicture}
\node{$1$} 
	child[missing] 
	child{node{$2$}
		child{node{$3$}
			child{node{$5$}}
			child{node{$6$}}
		}
		child{node{$4$}}
	};
\end{tikzpicture}
\end{minipage}
\begin{minipage}{1.45in}
\begin{tikzpicture}
\node{$1$} 
	child[missing] 
	child{node{$2$}
		child{node{$4$}
			child{node{$5$}}
			child{node{$6$}}
		}
		child{node{$3$}}
	};
\end{tikzpicture}
\end{minipage}
\begin{minipage}{1.45in}
\begin{tikzpicture}
\node{$1$} 
	child[missing] 
	child{node{$2$}
		child{node{$3$}
			child{node{$6$}}
			child{node{$5$}}
		}
		child{node{$4$}}
	};
\end{tikzpicture}
\end{minipage}

\bigskip

\begin{minipage}{1.45in}
\begin{tikzpicture}
\node{$1$} 
	child[missing] 
	child{node{$2$}
		child{node{$4$}
			child{node{$6$}}
			child{node{$5$}}
		}
		child{node{$3$}}
	};
\end{tikzpicture}
\end{minipage}
\begin{minipage}{1.45in}
\begin{tikzpicture}
\node{$1$} 
	child[missing] 
	child{node{$2$}
		child{node{$3$}
			child{node{$6$}}
			child{node{$4$}}
		}
		child{node{$5$}}
	};
\end{tikzpicture}
\end{minipage}
\begin{minipage}{1.45in}
\begin{tikzpicture}
\node{$1$} 
	child[missing] 
	child{node{$2$}
		child{node{$3$}
			child{node{$5$}}
			child{node{$4$}}
		}
		child{node{$6$}}
	};
\end{tikzpicture}
\end{minipage}
\begin{minipage}{1.45in}
\begin{tikzpicture}
\node{$1$} 
	child[missing] 
	child{node{$2$}
		child{node{$3$}
			child{node{$4$}}
			child{node{$5$}}
		}
		child{node{$6$}}
	};
\end{tikzpicture}
\end{minipage}

\bigskip


We first read by definition that 
\begin{equation*}
T_{D}(\alpha )=\{t_{1}\geqslant t_{2},t_{2}\geqslant t_{3},t_{3}\geqslant
t_{4},t_{3}\geqslant t_{6},t_{2}\geqslant t_{5}\}\text{.}
\end{equation*}
Let $\sigma $ denote some composition of acceptable moves, we then notice
the equivalence of the two sets 
\begin{eqnarray*}
\Sigma _{1} &=&\left\{ \sigma :\sigma ^{-1}(1)<\sigma ^{-1}(2)<\sigma
^{-1}(3)<\sigma ^{-1}(4),\sigma ^{-1}(2)<\sigma ^{-1}(5),\sigma
^{-1}(3)<\sigma ^{-1}(6)\right\} , \\
\Sigma _{2} &=&\left\{ \sigma :\sigma \text{ takes input tree to }\alpha 
\text{ where the input tree is admissibile}\right\} ,
\end{eqnarray*}
both generated by the requirement that the child must carry a larger label
than the parent. That is, both $\Sigma _{1}$ and $\Sigma _{2}$ classifies
the whole upper echelon class represented by $\alpha $.

Hence, 
\begin{equation*}
\bigcup _{\sigma \in \Sigma _{1}}\left\{ t_{1}\geqslant
t_{\sigma(2)}\geqslant t_{\sigma (3)}...\geqslant t_{\sigma \left( 6\right)
}\right\} =\{t_{1}\geqslant t_{2}\geqslant t_{3}\geqslant
t_{4},t_{2}\geqslant t_{5},t_{3}\geqslant t_{6}\}=T_{D}(\alpha )
\end{equation*}
and we are done.
\end{proof}

\subsection{Signed KM Acceptable Moves\label{Sec:Wild1}}

While Proposition \ref{Prop:KMIntegrationLimits} shows that summing over an
entire KM upper echelon class yields a time integration domain with clean
structure, it is not sufficient for our purposes. We prove an extended KM
board game in \S \ref{Sec:Wild1}-\ref{Sec:Wild5}. Recall the key observation
of the KM board game is that, many summands in $J^{(k+1)}(f^{(k+1)})$
actually have the same integrand if one switches the variable labelings,
thus one can take the acceptable moves to combine them. In fact, one can
combine them even more after the acceptable moves to get the integration
domain $D_{m}$ larger.\footnote{%
We do not know if one could combine even more than what we are going to do
in \S \ref{Sec:Wild1}-\ref{Sec:Wild5}.} Instead of aiming to reduce the
number of summands in $J^{(k+1)}(f^{(k+1)})$ even more, our goal this time
is to enlarge the integration domain when estimating $J_{\mu
,sgn}^{(k+1)}(f^{(k+1)})(t_{1},\underline{t}_{k+1})$ so that $U$-$V$
techniques can actually apply. Depending on the sign combination in $J_{\mu
_{m},sgn}^{(k+1)}(f^{(k+1)})(t_{1},\underline{t}_{k+1})$, one could run into
the problem that one needs to estimate the $x$ part and the $x^{\prime }$
part using the same time integral. This problem is another obstacle stopping 
$U$-$V$ space techniques from being used in the analysis of $GP$
hierarchies, away from the other obstacle that $D_{m}$ was previously
unknown.

From here on out, we denote the already unioned/combined integrals in one
echelon class a upper echelon class integral and we use Proposition \ref%
{Prop:KMIntegrationLimits} for its integration limits. We also put a $+$ or $%
-$ sign at the corresponding node of a tree as we are dealing with $J_{\mu
,sgn}^{(k+1)}(f^{(k+1)})(t_{1},\underline{t}_{k+1})$ in which there are $%
B^{+}$ and $B^{-}$ at each coupling. We start with the following example.

\begin{example}
\label{example:one could indeed combine more}The following two upper echelon
class intergrals 
\begin{eqnarray*}
I_{1}
&=&\int_{t_{4}=0}^{t_{1}}\int_{t_{2}=t_{4}}^{t_{1}}%
\int_{t_{3}=0}^{t_{2}}U^{(1)}(t_{1}-t_{2})B_{1,2}^{-}U^{(2)}(t_{2}-t_{3})B_{1,3}^{+}U^{(3)}(t_{3}-t_{4})B_{2,4}^{+}(f^{(4)})(t_{1},%
\underline{t}_{4})d\underline{t}_{4} \\
I_{2}
&=&\int_{t_{4}=0}^{t_{1}}\int_{t_{2}=0}^{t_{1}}%
\int_{t_{3}=t_{4}}^{t_{2}}U^{(1)}(t_{1}-t_{2})B_{1,2}^{+}U^{(2)}(t_{2}-t_{3})B_{1,3}^{-}U^{(3)}(t_{3}-t_{4})B_{3,4}^{+}(f^{(4)})(t_{1},%
\underline{t}_{4})d\underline{t}_{4}
\end{eqnarray*}%
actually have the same integrand if one does a $t_{2}\leftrightarrow t_{3}$
swap in $I_{1}$, despite that the trees corresponding to $I_{1}$ and $I_{2}$
have different skeltons. On the one hand, shorten $e^{i\left(
t_{i}-t_{j}\right) \triangle }$ as $U_{i,j}$, 
\begin{eqnarray*}
I_{1}
&=&\int_{t_{4}=0}^{t_{1}}\int_{t_{2}=t_{4}}^{t_{1}}%
\int_{t_{3}=0}^{t_{2}}U_{1,3}(\left\vert U_{3,4}\phi \right\vert
^{2}U_{3,4}\phi )(x_{1})U_{1,2}(\overline{U_{2,4}\phi }\overline{U_{2,4}\phi 
}U_{2,4}\left( \left\vert \phi \right\vert ^{2}\phi \right) )(x_{1}^{\prime
})d\underline{t}_{4} \\
&=&\int_{t_{4}=0}^{t_{1}}\int_{t_{2}=0}^{t_{1}}\int_{t_{3}=\max
(t_{2},t_{4})}^{t_{1}}U_{1,2}(\left\vert U_{2,4}\phi \right\vert
^{2}U_{2,4}\phi )(x_{1})U_{1,3}(\overline{U_{3,4}\phi }\overline{U_{3,4}\phi 
}U_{3,4}\left( \left\vert \phi \right\vert ^{2}\phi \right) )(x_{1}^{\prime
})d\underline{t}_{4} \\
&=&\int_{t_{4}=0}^{t_{1}}\int_{t_{2}=0}^{t_{1}}\int_{t_{3}=\max
(t_{2},t_{4})}^{t_{1}}U^{(1)}(t_{1}-t_{2})B_{1,2}^{+}U^{(2)}(t_{2}-t_{3})B_{1,3}^{-}U^{(3)}(t_{3}-t_{4})B_{3,4}^{+}(f^{(4)})(t_{1},%
\underline{t}_{4})d\underline{t}_{4}
\end{eqnarray*}%
where we have put in $f^{(4)}=\left( \left\vert \phi \right\rangle
\left\langle \phi \right\vert \right) ^{\otimes 4}$ for simplicity.\footnote{%
One could put a general symmetric $f^{(4)}$ here and get the same result.}%
Hence, 
\begin{equation*}
I_{1}+I_{2}=\int_{t_{4}=0}^{t_{1}}\int_{t_{2}=0}^{t_{1}}%
\int_{t_{3}=t_{4}}^{t_{1}}U^{(1)}(t_{1}-t_{2})B_{1,2}^{+}U^{(2)}(t_{2}-t_{3})B_{1,3}^{-}U^{(3)}(t_{2}-t_{3})B_{3,4}^{+}(f^{(4)})(t_{1},%
\underline{t}_{4})d\underline{t}_{4}.
\end{equation*}%
On the other hand, if one puts the $I_{1}$ tree on the left and the $I_{2}$
tree on the right, the trees read

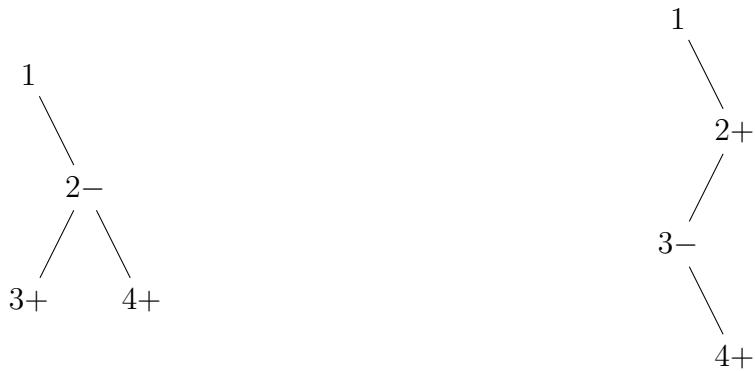
\begin{minipage}{0.50\linewidth}
\begin{center}
\begin{tikzpicture}
\node{$1$} 
	child[missing] 
	child{node{$2-$}
		child{node{$3+$}}
		child{node{$4+$}}
	};
\end{tikzpicture}
\end{center}
\end{minipage}%
\begin{minipage}{0.50\linewidth}
\begin{center}
\begin{tikzpicture}
\node{$1$}
	child[missing]
	child{node{$2+$}
		child{node{$3-$}
			child[missing]
			child{node{$4+$}}
		}
		child[missing]
	};
\end{tikzpicture}
\end{center}
\end{minipage}
\end{example}

Example \ref{example:one could indeed combine more} shows that one could
indeed further combine the summands in $J^{(k+1)}(f^{(k+1)})$ after the
original KM board game has been performed. We will explain why our $U$-$V$
techniques apply to $I_{1}+I_{2}$ but not $I_{1}$, $I_{2}$ individually in 
\S \ref{sec:key example}. Despite Example \ref{example:one could indeed
combine more} uses the already combined upper echelon integrals, our
extended KM board game actually starts from scratch, that is, it starts from 
$\gamma ^{(1)}(t_{1})$ instead of already combined upper echelon integrals.
However, it is still a multi-step process. We will first switch the terms in 
$\gamma ^{(1)}(t_{1})$ into their tamed form via signed KM acceptable moves
in \S \ref{Sec:Wild1}-\ref{Sec:Wild2}, we then categorize the tamed forms
into tamed classes via the wild moves in \S \ref{Sec:Wild4}-\ref{Sec:Wild5}.

We now explain the program as follows: as before, start by expanding $\gamma
^{(1)}(t_{1})$ to coupling level $k$, which generates a sum expansion of $k!$
terms. But now for each of these $k!$ terms, expand the collapsing operators 
$B_{\mu (j),j}^{(j)}$ into $+$ and $-$ components, which introduces $2^{k}$
terms. Thus, in all, we have $2^{k}k!$ terms, each of which has
sign-dependent collapsing operators.

\begin{equation}
\gamma ^{(1)}=\sum_{\mu ,\func{sgn}}I(\mu ,\func{id},\func{sgn},\gamma
^{(k+1)})  \label{E:gamma-expansion}
\end{equation}%
where $\func{id}$ is the identity permutation on $\{2,\ldots ,k+1\},$ 
\begin{equation*}
I(\mu ,\sigma ,\func{sgn})=\int_{t_{1}\geq t_{\sigma (2)}\geq \cdots \geq
t_{\sigma (k+1)}}J_{\mu ,\func{sgn}}^{(k+1)}(\gamma ^{(k+1)})(t_{1},%
\underline{t}_{k+1})\,d\underline{t}_{k+1}
\end{equation*}%
and $J_{\mu ,\func{sgn}}^{(k+1)}$ is defined as in (\ref{E:JSgnDef}). The
equation \eqref{E:gamma-expansion} is a sum over all \emph{admissible} $\mu $%
, that is, collapsing maps that satisfy $\mu (j)<j$, of which there are $k!$%
. It is also a sum over all $\func{sgn}$ maps, of which there are $2^{k}$.

We define a signed version of the KM acceptable moves, still denoted $\text{%
KM}(j,j+1)$, defined provided $\mu (j)\neq \mu (j+1)$ and $\mu(j+1)<j$. It
is defined as the following action on a triple $(\mu ,\sigma ,\func{sgn})$: 
\begin{equation*}
(\mu ^{\prime },\sigma ^{\prime },\func{sgn}^{\prime })=\text{KM}(j,j+1)(\mu
,\sigma ,\func{sgn})
\end{equation*}%
where 
\begin{equation*}
\begin{aligned} &\mu' = (j,j+1) \circ \mu \circ (j,j+1) \\ &\sigma' =
(j,j+1) \circ \sigma\\ &\operatorname{sgn}' = \operatorname{sgn} \circ
(j,j+1) \\ \end{aligned}
\end{equation*}

Graphically, this means that nodes $j$ and $j+1$ belong to different left
branches and corresponds to switching nodes $j$ and $j+1$, \emph{leaving the
signs in place on the tree} -- in other words, the node previously labeled $%
j $ is relabeled $j+1$, and the node previously labeled $j+1$ is relabeled $%
j $, but the signs are left in place.

A slight modification of the arguments in \cite{KM} shows that, analogous to %
\eqref{E:KM-den1}, if $(\mu ^{\prime },\sigma ^{\prime },\func{sgn}^{\prime
})=\text{KM}(j,j+1)(\mu ,\sigma ,\func{sgn})$ and $f^{(k+1)}$ is a symmetric
density, then 
\begin{equation}
J_{\mu ^{\prime },\func{sgn}^{\prime }}^{(k+1)}(f^{(k+1)})(t_{1},{\sigma
^{\prime }}^{-1}(\underline{t}_{k+1}))=J_{\mu ,\func{sgn}%
}^{(k+1)}(f^{(k+1)})(t_{1},{\sigma }^{-1}(\underline{t}_{k+1}))
\label{E:KM-den1b}
\end{equation}%
It follows from \eqref{E:KM-den1b} that 
\begin{equation}
I(\mu ^{\prime },\sigma ^{\prime },\func{sgn}^{\prime (k+1)})=I(\mu ,\sigma ,%
\func{sgn},f^{(k+1)})  \label{E:KM-den2b}
\end{equation}%
As in the sign independent case (or more accurately, the combined sign case)
we can combine KM acceptable moves as follows: if $\rho $ is a permutation
of $\{2,\ldots ,k+1\}$ such that it is possible to write $\rho $ as a
composition of transpositions 
\begin{equation*}
\rho =\tau _{1}\circ \cdots \circ \tau _{r}
\end{equation*}%
for which each operator $\text{KM}(\tau _{j})$ on the right side of the
following is an acceptable action 
\begin{equation*}
\text{KM}(\rho )\overset{\mathrm{def}}{=}\text{KM}(\tau _{1})\circ \cdots
\circ \text{KM}(\tau _{r})
\end{equation*}%
then $\text{KM}(\rho )$, defined by this composition, is acceptable as well.
In this case $(\mu ^{\prime },\sigma ^{\prime },\func{sgn}^{\prime })=\text{%
KM}(\rho )(\mu ,\sigma ,\func{sgn})$ and 
\begin{align*}
& \mu ^{\prime} = \rho \circ \mu \circ \rho^{-1} \\
& \sigma ^{\prime }=\rho \circ \sigma \\
& \func{sgn}^{\prime }=\func{sgn}\circ \rho ^{-1}
\end{align*}%
Of course, \eqref{E:KM-den1b} and \eqref{E:KM-den2b} hold as well. If $(\mu ,%
\func{sgn})$ and $(\mu ^{\prime },\func{sgn}^{\prime })$ are such that there
exists $\rho $ as above for which $(\mu ^{\prime },\sigma ^{\prime },\func{%
sgn}^{\prime })=\text{KM}(\rho )(\mu ,\sigma ,\func{sgn}^{\prime })$ then we
say that $(\mu ^{\prime },\func{sgn}^{\prime })$ and $(\mu ,\func{sgn})$ are 
\emph{KM-relatable}. This is an equivalence relation that partitions the set
of collapsing map/sign map pairs into equivalence classes. In the graphical
representation two collapsing map/sign map pairs are KM-relatable if and
only if they have the same signed skeleton tree.

Whereas we could use the signed KM acceptable moves to convert an arbitrary
admissible $\mu $ to an upper echelon $\mu ^{\prime }$, this will no longer
suit our purpose. Instead, our program will be to convert each pair $(\mu ,%
\func{sgn})$ to a \emph{tamed} form, which we define in the next section.
The reason for our preference of tamed form over upper echelon form is that
it is invariant under \emph{wild moves}, to be introduced in \S \ref%
{Sec:Wild4}

\pagebreak

\subsection{Tamed Form\label{Sec:Wild2}}

In this section, we define what it means for a pair $(\mu ,\func{sgn})$, and
its corresponding tree representation, to be \emph{tamed, }in Definition \ref%
{def:tamedF}. Then first through an example, we present an algorithm for
producing the tamed enumeration of a signed skeleton. The general algorithm
is then stated in Algorithm \ref{alg:tamed enumeration}. Notice that it
produces a different enumeration from Algorithm \ref{alg:TreeToUpperEchelon}%
. Compared with Algorithm \ref{alg:TreeToUpperEchelon}, not only the tamed
form enumeration deals with left branches first, it also deals with "$+$"
first.\footnote{%
By symmetry, one could deal with "$-$" first here to get a very similar
tamed form. But left and right branches are not symmetric as they are
defined differently.} In \S \ref{Sec:Wild3}, we exhibit how to reduce a
signed tree with same skeleton but different enumeration into the tamed form
using signed KM acceptable moves.

We will now give a nongraphical set of conditions on $\mu$ and $\func{sgn}$
that determine whether or not $(\mu,\func{sgn})$ is tamed. First, we define
the concept of \emph{tier}. We say that $j\geq 2$ is tier $q$ if 
\begin{equation*}
\mu^q(j) = 1 \quad \text{but} \quad \mu^{q-1}(j) >1
\end{equation*}
where $\mu^q = \mu \circ \cdots \circ \mu$, the composition taken $q$ times.
We write $t(j)$ for the tier value of $j$.

\begin{definition}
\label{def:tamedF}A pair $(\mu ,\func{sgn})$ is \emph{tamed} if it meets the
following four requirements:

\begin{enumerate}
\item If $t(\ell )<t(r)$, then $\ell <r.$

\item If $t(\ell )=t(r)$, $\mu^2(\ell) = \mu^2(r)$, $\func{sgn}(\mu (\ell ))=%
\func{sgn}(\mu (r))$, and $\mu (\ell )<\mu (r)$, then $\ell <r.$

\item If $t(\ell )=t(r)$, $\mu^2(\ell) = \mu^2(r)$, $\func{sgn}(\mu (\ell
))=+$, and $\func{sgn}(\mu (r))=-$, then $\ell <r.$

\item If $t(\ell) = t(r)$, $\mu^2(\ell) \neq \mu^2(r)$, $\mu(\ell)<\mu(r)$,
then $\ell<r$.
\end{enumerate}
\end{definition}

Note that the statement $\mu^2(\ell) = \mu^2(r)$ means graphically that the
parents of $\ell$ and $r$ belong to the same left branch. Conditions (2),
(3), and (4) specify the ordering for $\ell$ and $r$ belonging to the same
tier, and the rule depends upon whether or not the \emph{parents} of $\ell$
and $r$ belong to the same left branch. If they do, rule (3) says that a
positive parent dominates over a negative parent, but rule (2) says that if
the parents are of the same sign, then the ordering follows the parental
ordering. Finally, if the parents do not belong to the same left branch,
rule (4) says that the ordering follows the parental ordering regardless of
the signs of the parents.

\begin{example}
The $(\mu ,\func{sgn})$ pair with tier properties indicated in the following
chart%
\begin{equation*}
\begin{tabular}{c|ccccccccccccc}
$j$ & 2 & $3$ & $4$ & $5$ & $6$ & $7$ & $8$ & $9$ & $10$ & $11$ & $12$ & $13$
& $14$ \\ \hline
$\mu (j)$ & $1$ & $1$ & $1$ & $1$ & $5$ & $5$ & $2$ & $2$ & $7$ & $7$ & $9$
& $9$ & $8$ \\ 
$\func{sgn}(j)$ & $-$ & $-$ & $+$ & $+$ & $-$ & $+$ & $-$ & $+$ & $-$ & $+$
& $+$ & $-$ & $+$ \\ 
$t(j)$ & $1$ & $1$ & $1$ & $1$ & $2$ & $2$ & $2$ & $2$ & $3$ & $3$ & $3$ & $%
3 $ & $3$%
\end{tabular}%
\end{equation*}%
is tamed. All four conditions in Definition \ref{def:tamedF} can be checked
from the above chart. This is in fact the $(\mu ,\func{sgn})$ pair that
appears in the example that follows.
\end{example}

In the example below, we illustrate an algorithm for determining the unique
tamed enumeration of a signed skeleton tree. After the example is completed,
we will give the general form of the algorithm.

\newpage

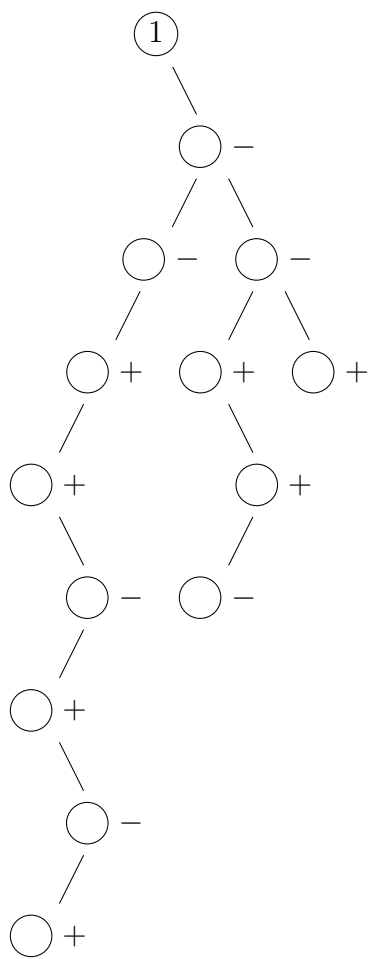
\begin{minipage}[t]{0.45\linewidth}
\vspace{0pt}
\begin{center}
\begin{tikzpicture}
\node{$\encircle{1}$}
	child[missing]
	child{node{$\encircle{\;}-$}
		child{node{$\encircle{\;}-$}
			child{node{$\encircle{\;}+$}
				child{node{$\encircle{\;}+$}
					child[missing]
					child{node{$\encircle{\;}-$}
						child{node{$\encircle{\;}+$}
							child[missing]
							child{node{$\encircle{\;}-$}
								child{node{$\encircle{\;}+$}}
								child[missing]
							}
						}
						child[missing]
					}
				}
				child[missing]
			}
			child[missing]
		}
		child{node{$\encircle{\;}-$}
			child{node{$\encircle{\;}+$}
				child[missing]
				child{node{$\encircle{\;}+$}
					child{node{$\encircle{\;}-$}}
					child[missing]
				}
			}
			child{node{$\encircle{\;}+$}}
		}
	};
\end{tikzpicture}
\end{center}
\end{minipage}%
\begin{minipage}[t]{0.55\linewidth}
We illustrate the tamed form enumeration with the following example.  Let's start with the following skeleton (on the left) with only the signs indicated.  (Recall that KM acceptable moves will leave the signs in place in the tree and just change the numbering of the nodes).
Start by considering all nodes mapping to $1$ (the universal ancestor) - this is the left branch attached to $1$ that is four nodes long in the order $--++$, and we enumerate it in order as $2,3,4,5$.
\end{minipage}

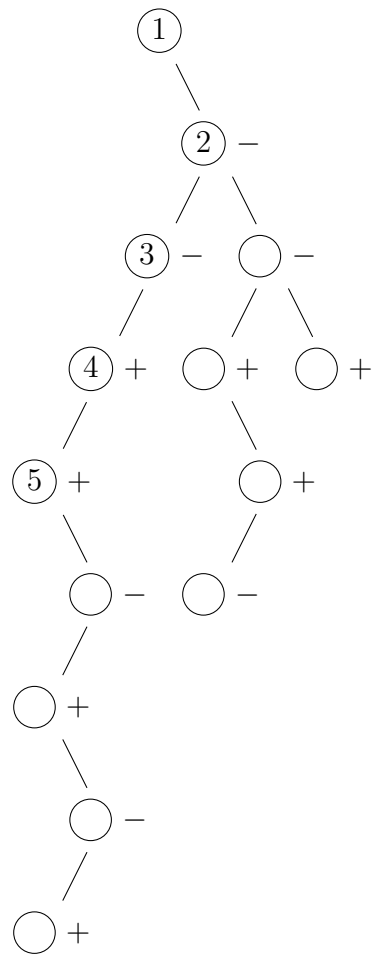
\vspace{-2in}

\begin{minipage}[b]{0.55\linewidth}
We then put this full left branch in the (empty) queue, but list the $+$ nodes first and then the $-$ nodes
$$\text{Queue:  } 4+, 5+, 2-, 3-$$
Then we start working along the queue from left to right.  Since $4+$ has no right child, we skip it and move to $5+$.  Since $5+$ does have a right child, we label it with the next available number $6$, and completely enumerate the entire left branch that starts with this $6$ node (that means, in this case, labeling $6-$ and $7+$ as shown on the next graph).   
\end{minipage}%
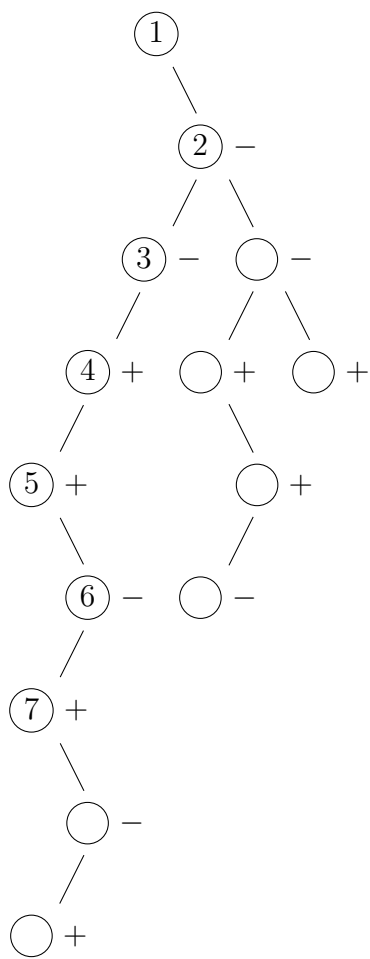
\begin{minipage}[b]{0.45\linewidth}
\begin{center}
\begin{tikzpicture}
\node{$\encircle{1}$}
	child[missing]
	child{node{$\encircle{2}-$}
		child{node{$\encircle{3}-$}
			child{node{$\encircle{4}+$}
				child{node{$\encircle{5}+$}
					child[missing]
					child{node{$\encircle{\;}-$}
						child{node{$\encircle{\;}+$}
							child[missing]
							child{node{$\encircle{\;}-$}
								child{node{$\encircle{\;}+$}}
								child[missing]
							}
						}
						child[missing]
					}
				}
				child[missing]
			}
			child[missing]
		}
		child{node{$\encircle{\;}-$}
			child{node{$\encircle{\;}+$}
				child[missing]
				child{node{$\encircle{\;}+$}
					child{node{$\encircle{\;}-$}}
					child[missing]
				}
			}
			child{node{$\encircle{\;}+$}}
		}
	};
\end{tikzpicture}
\end{center}
\end{minipage}

\newpage

\begin{minipage}[t]{0.45\linewidth}
\vspace{0pt}
\begin{center}
\begin{tikzpicture}
\node{$\encircle{1}$}
	child[missing]
	child{node{$\encircle{2}-$}
		child{node{$\encircle{3}-$}
			child{node{$\encircle{4}+$}
				child{node{$\encircle{5}+$}
					child[missing]
					child{node{$\encircle{6}-$}
						child{node{$\encircle{7}+$}
							child[missing]
							child{node{$\encircle{\;}-$}
								child{node{$\encircle{\;}+$}}
								child[missing]
							}
						}
						child[missing]
					}
				}
				child[missing]
			}
			child[missing]
		}
		child{node{$\encircle{\;}-$}
			child{node{$\encircle{\;}+$}
				child[missing]
				child{node{$\encircle{\;}+$}
					child{node{$\encircle{\;}-$}}
					child[missing]
				}
			}
			child{node{$\encircle{\;}+$}}
		}
	};
\end{tikzpicture}
\end{center}
\end{minipage}
\begin{minipage}[t]{0.55\linewidth}
\vspace{0pt}
Then we add this entire left branch to the queue, putting the $+$ nodes before the $-$ nodes.  We also pop $4+$ and $5+$ from the (left of the) queue, since we have already dealt with them.  The queue now reads
$$\text{Queue:  }  2-, 3-, 7+, 6-$$
Now we come to the next node in the queue (reading from the left) which is $2-$.  The node $2$ does have a right child.  We label it as $8$ (the next available number) and completely enumerate the left branch that starts with $8$, which means labeling $8-$, $9+$ as shown.  
\end{minipage}

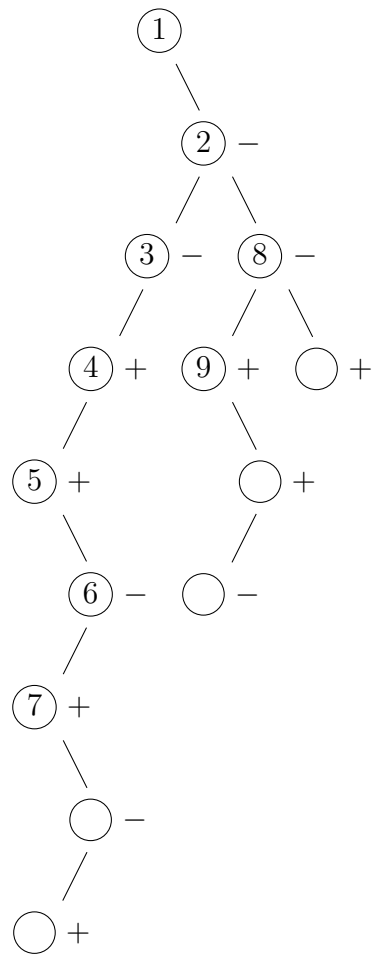
\vspace{-2in}

\begin{minipage}[b]{0.55\linewidth}
From the queue, we pop $2$ and add the $8-$, $9+$ left branch, but first all $+$ nodes and then all $-$ nodes:
$$\text{Queue:  }  3-, 7+, 6-, 9+, 8-$$
Since $3-$ does not have a right child, we pop it and proceed to $7+$, which does have a right child, which is labeled with $10$, and the left branch starting at $10$ is enumerated as $10-, 11+$, as shown
\end{minipage}%
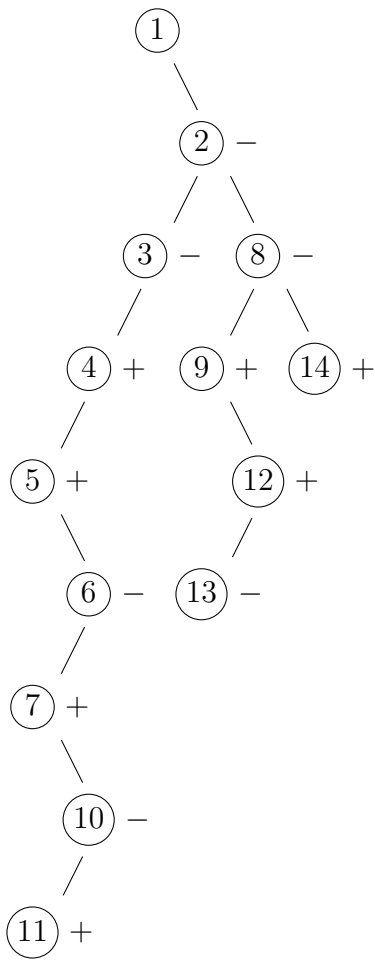
\begin{minipage}[b]{0.45\linewidth}
\begin{center}
\begin{tikzpicture}
\node{$\encircle{1}$}
	child[missing]
	child{node{$\encircle{2}-$}
		child{node{$\encircle{3}-$}
			child{node{$\encircle{4}+$}
				child{node{$\encircle{5}+$}
					child[missing]
					child{node{$\encircle{6}-$}
						child{node{$\encircle{7}+$}
							child[missing]
							child{node{$\encircle{\;}-$}
								child{node{$\encircle{\;}+$}}
								child[missing]
							}
						}
						child[missing]
					}
				}
				child[missing]
			}
			child[missing]
		}
		child{node{$\encircle{8}-$}
			child{node{$\encircle{9}+$}
				child[missing]
				child{node{$\encircle{\;}+$}
					child{node{$\encircle{\;}-$}}
					child[missing]
				}
			}
			child{node{$\encircle{\;}+$}}
		}
	};
\end{tikzpicture}
\end{center}
\end{minipage}

\newpage

\begin{minipage}{0.40\linewidth}
\begin{tikzpicture}
\node{$\encircle{1}$}
	child[missing]
	child{node{$\encircle{2}-$}
		child{node{$\encircle{3}-$}
			child{node{$\encircle{4}+$}
				child{node{$\encircle{5}+$}
					child[missing]
					child{node{$\encircle{6}-$}
						child{node{$\encircle{7}+$}
							child[missing]
							child{node{$\encircle{10}-$}
								child{node{$\encircle{11}+$}}
								child[missing]
							}
						}
						child[missing]
					}
				}
				child[missing]
			}
			child[missing]
		}
		child{node{$\encircle{8}-$}
			child{node{$\encircle{9}+$}
				child[missing]
				child{node{$\encircle{12}+$}
					child{node{$\encircle{13}-$}}
					child[missing]
				}
			}
			child{node{$\encircle{14}+$}}
		}
	};
\end{tikzpicture}
\end{minipage}
\begin{minipage}{0.60\linewidth}
The queue is updated:
$$\text{Queue:  }   6-, 9+, 8-, 11+, 10-$$
By now the procedure is probably clear, so we will jump to the fully enumerated tree.

\end{minipage}

\bigskip

Here is the general algorithm. Recall that a queue is a data structure where
elements are added on the right and removed (dequeued) on the left.

\begin{algorithm}
\label{alg:tamed enumeration}Start with a \emph{queue} that at first
contains only $1$, and start with a \emph{next available label} $j=2$.

\begin{enumerate}
\item Dequeue the leftmost entry $\ell$ of the queue. (If the queue is
empty, stop). On the tree, pass to the right child of $\ell$, and enumerate
its left branch starting with the next available label $j, j+1, \ldots, j+q$%
. If there is no right child of $\ell$, return to the beginning of step (1).

\item Take the left branch enumerated in (1) and first list all $+$ nodes in
order from $j, \ldots, j+q$ and add them to the right side of the queue, and
then list in order all $-$ nodes from $j, \ldots, j+q$ and add them to the
right side of the queue

\item Set the next available label to be $j+q+1$, and return to step (1).
\end{enumerate}
\end{algorithm}

\newpage

\subsubsection{Reduce to Tamed Forms via Signed KM Board Game\label%
{Sec:Wild3}}

We will now explain how to execute a sequence of signed KM acceptable moves
that will bring the example tree from the previous section, with some other
enumeration, into the tamed form.

\bigskip

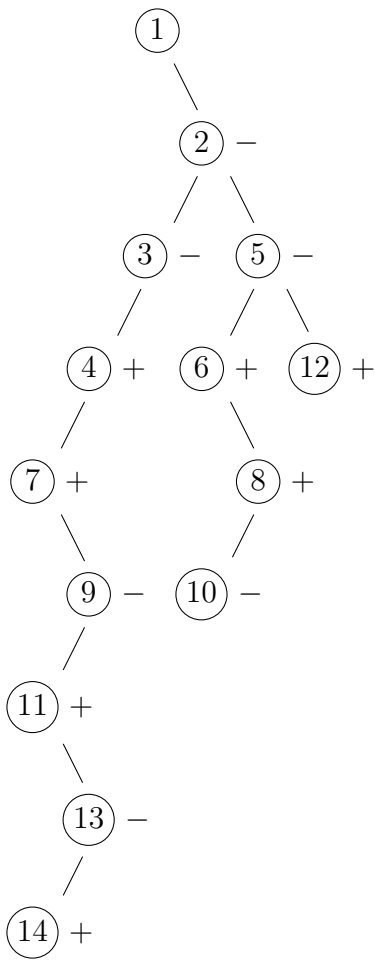
\begin{minipage}{0.35\linewidth}
\begin{tikzpicture}
\node{$\encircle{1}$}
	child[missing]
	child{node{$\encircle{2}-$}
		child{node{$\encircle{3}-$}
			child{node{$\encircle{4}+$}
				child{node{$\encircle{7}+$}
					child[missing]
					child{node{$\encircle{9}-$}
						child{node{$\encircle{11}+$}
							child[missing]
							child{node{$\encircle{13}-$}
								child{node{$\encircle{14}+$}}
								child[missing]
							}
						}
						child[missing]
					}
				}
				child[missing]
			}
			child[missing]
		}
		child{node{$\encircle{5}-$}
			child{node{$\encircle{6}+$}
				child[missing]
				child{node{$\encircle{8}+$}
					child{node{$\encircle{10}-$}}
					child[missing]
				}
			}
			child{node{$\encircle{12}+$}}
		}
	};
\end{tikzpicture}
\end{minipage}%
\begin{minipage}{0.65\linewidth}
We are going to start with the enumeration at left, which is not tamed, and explain how to execute KM acceptable moves in order to convert this tree into tamed form.  Of course, this is quite similar to what Klainerman-Machedon described, with just a modification to prioritize plusses over minuses.

\bigskip

This tree corresponds to the following $\mu$ and $\sgn$ functions

\begin{center}
\begin{tabular}{ c | c c c c c c c c c c c c c c}
$j$ & $2$ & $3$ & $4$ & $5$ & $6$ & $7$ & $8$ & $9$ & $10$ & $11$ & $12$ & $13$ & $14$\\
\hline
$\sgn(j)$ & $-$ & $-$ & $+$ & $-$ & $+$ & $+$ & $+$ & $-$ & $-$ & $+$ & $+$ & $-$  & $+$\\
$\mu(j)$ & 1 & 1 & 1 & 2 & 2 & 1 & 6 & 7 & 6 & 7 & 5 & 11 & 11
\end{tabular}
\medskip
\end{center}

We will keep a queue that right now includes only the node $1$
$$\text{Queue: }1$$
Following the queue, we move all nodes (all $j$) for which $\mu(j)=1$ all the way to left using KM moves.  Since $\mu(7)=1$ although $\mu(5)=2$ and $\mu(6)=2$, we apply the following KM moves $\text{KM}(6,7)$ and then $\text{KM}(5,6)$.

\end{minipage}

\bigskip

\noindent The $\text{KM}(6,7)$ move is 
\begin{align*}
&\mu \mapsto (6,7) \circ \mu \circ (6,7) \\
&\func{sgn} \mapsto \func{sgn} \circ (6,7)
\end{align*}
The $\text{KM}(5,6)$ move is 
\begin{align*}
&\mu \mapsto (5,6) \circ \mu \circ (5,6) \\
&\func{sgn} \mapsto \func{sgn} \circ (5,6)
\end{align*}
and together these result in the following:

\newpage

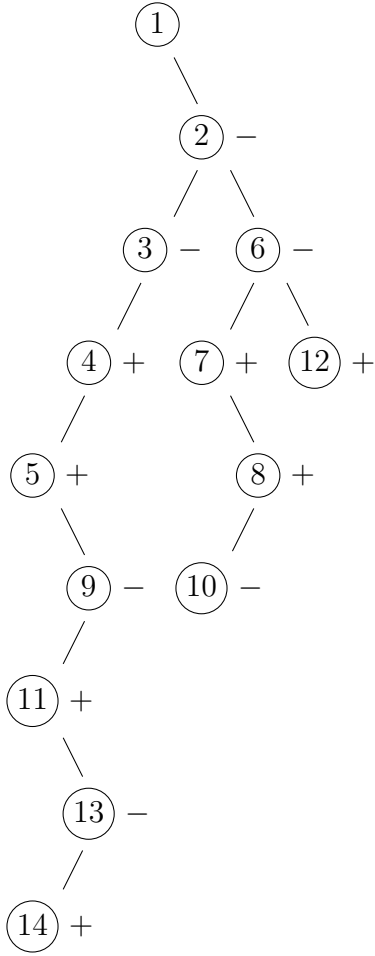
\begin{minipage}{0.35\linewidth}
\begin{tikzpicture}
\node{$\encircle{1}$}
	child[missing]
	child{node{$\encircle{2}-$}
		child{node{$\encircle{3}-$}
			child{node{$\encircle{4}+$}
				child{node{$\encircle{5}+$}
					child[missing]
					child{node{$\encircle{9}-$}
						child{node{$\encircle{11}+$}
							child[missing]
							child{node{$\encircle{13}-$}
								child{node{$\encircle{14}+$}}
								child[missing]
							}
						}
						child[missing]
					}
				}
				child[missing]
			}
			child[missing]
		}
		child{node{$\encircle{6}-$}
			child{node{$\encircle{7}+$}
				child[missing]
				child{node{$\encircle{8}+$}
					child{node{$\encircle{10}-$}}
					child[missing]
				}
			}
			child{node{$\encircle{12}+$}}
		}
	};
\end{tikzpicture}
\end{minipage}
\begin{minipage}{0.65\linewidth}

\begin{center}
\begin{tabular}{ c | c c c c c c c c c c c c c c}
$j$ & $2$ & $3$ & $4$ & $5$ & $6$ & $7$ & $8$ & $9$ & $10$ & $11$ & $12$ & $13$ & $14$\\
\hline
$\sgn(j)$ & $-$ & $-$ & $+$ & $+$ & $-$ & $+$ & $+$ & $-$ & $-$ & $+$ & $+$ & $-$  & $+$\\
$\mu(j)$ & 1 & 1 & 1 & 1 & 2 & 2 & 7 & 5 & 7 & 5 & 6 & 11 & 11
\end{tabular}
\end{center}

\bigskip

These two moves have been implemented on the revised graph at left.    \\

Inspecting the $\mu$ chart above, we see all output $1$'s have been moved to the left, and the complete list of $j$ for which $\mu(j)=1$ is $2-, 3-, 4+, 5+$.  We add these numbers to our queue, but first add all plusses and then all minuses:
$$\text{Queue: }\text{\sout{1}, 4, 5, 2, 3}$$
Since we have completed $1$ on the queue, we next move to $4$, but there are no $j$ for which $\mu(j)=4$, so we proceed to $5$.  
As we can see from the $\mu$ table or from the tree, $\mu(9)=5$ and $\mu(11)=5$, so we execute KM moves to bring these all the way to the left (but to the right of the $1$'s):

\end{minipage}

\bigskip

The next step is therefore to implement moves $\text{KM}(8,9)$, $\text{KM}%
(7,8)$, $\text{KM}(6,7)$, which brings the $\mu$ table to

\bigskip

\begin{center}
\begin{tabular}{c|cccccccccccccc}
$j$ & $2$ & $3$ & $4$ & $5$ & $6$ & $7$ & $8$ & $9$ & $10$ & $11$ & $12$ & $%
13$ & $14$ &  \\ \hline
$\func{sgn}(j)$ & $-$ & $-$ & $+$ & $+$ & $-$ & $-$ & $+$ & $+$ & $-$ & $+$
& $+$ & $-$ & $+$ &  \\ 
$\mu(j)$ & 1 & 1 & 1 & 1 & 5 & 2 & 2 & 8 & 8 & 5 & 7 & 11 & 11 & 
\end{tabular}
\end{center}

\bigskip

This is followed by the moves $\text{KM}(11,10)$, $\text{KM}(10,9)$, $\text{%
KM}(9,8)$, $\text{KM}(8,7)$, which bring the $\mu$ table to

\bigskip

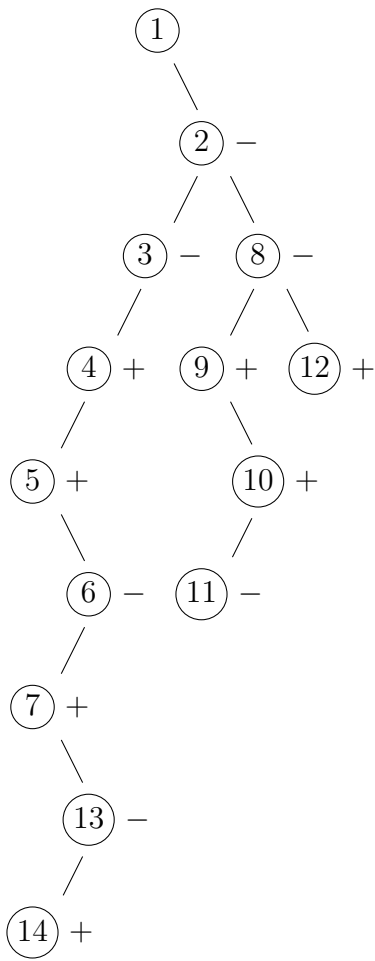
\begin{center}
\begin{tabular}{c|cccccccccccccc}
$j$ & $2$ & $3$ & $4$ & $5$ & $6$ & $7$ & $8$ & $9$ & $10$ & $11$ & $12$ & $%
13$ & $14$ &  \\ \hline
$\func{sgn}(j)$ & $-$ & $-$ & $+$ & $+$ & $-$ & $+$ & $-$ & $+$ & $+$ & $-$
& $+$ & $-$ & $+$ &  \\ 
$\mu(j)$ & 1 & 1 & 1 & 1 & 5 & 5 & 2 & 2 & 9 & 9 & 8 & 7 & 7 & 
\end{tabular}
\end{center}

\newpage

\begin{minipage}{0.35\linewidth}
\begin{tikzpicture}
\node{$\encircle{1}$}
	child[missing]
	child{node{$\encircle{2}-$}
		child{node{$\encircle{3}-$}
			child{node{$\encircle{4}+$}
				child{node{$\encircle{5}+$}
					child[missing]
					child{node{$\encircle{6}-$}
						child{node{$\encircle{7}+$}
							child[missing]
							child{node{$\encircle{13}-$}
								child{node{$\encircle{14}+$}}
								child[missing]
							}
						}
						child[missing]
					}
				}
				child[missing]
			}
			child[missing]
		}
		child{node{$\encircle{8}-$}
			child{node{$\encircle{9}+$}
				child[missing]
				child{node{$\encircle{10}+$}
					child{node{$\encircle{11}-$}}
					child[missing]
				}
			}
			child{node{$\encircle{12}+$}}
		}
	};
\end{tikzpicture}
\end{minipage}
\begin{minipage}{0.65\linewidth}
At this point, the tree takes the form as pictured to the left.  All $5$'s have been moved to their proper position in the $\mu$ table.   The complete list of $j$ for which $\mu(j)=5$ is $6-, 7-$, so  we add these numbers to the queue but first add the plusses, then the minuses:
$$\text{Queue: }\text{\sout{1, 4, 5}, 2, 3, 7, 6}$$
Since we have addressed $5$ on the queue, we move to the next one, which is $2$.  This means we have to move all $j$ for which $\mu(j)=2$ all the way to the left (just to the right of $5$).  Examining the $\mu$ table, we see that these $j$ are already in place, at positions $8-, 9+$.  So no KM moves are needed, and we add to the queue
$$\text{Queue: }\text{\sout{1, 4, 5, 2}, 3, 7, 6, 9, 8}$$
Next on the queue is $3$, but there are no $j$ for which $\mu(j)=3$, so we proceed to $7$ on the queue.  From the $\mu$ table or the tree, we see there are two $j$ for which $\mu(j)=7$, namely $13$ and $14$.    We therefore execute KM moves to bring these to the left in the $\mu$ table, just to the right of $2$.  

\end{minipage}

\bigskip

Specifically, we do $\text{KM}(12,13)$, $\text{KM}(11,12)$ and $\text{KM}%
(10,11)$, which brings us to this $\mu$ table

\begin{center}
\begin{tabular}{c|cccccccccccccc}
$j$ & $2$ & $3$ & $4$ & $5$ & $6$ & $7$ & $8$ & $9$ & $10$ & $11$ & $12$ & $%
13$ & $14$ &  \\ \hline
$\func{sgn}(j)$ & $-$ & $-$ & $+$ & $+$ & $-$ & $+$ & $-$ & $+$ & $-$ & $+$
& $-$ & $+$ & $+$ &  \\ 
$\mu(j)$ & 1 & 1 & 1 & 1 & 5 & 5 & 2 & 2 & 7 & 9 & 9 & 8 & 7 & 
\end{tabular}
\end{center}

After that, we do $\text{KM}(13,14)$, $\text{KM}(12,13)$ and $\text{KM}%
(11,12)$, which brings us to this $\mu$ table

\begin{center}
\begin{tabular}{c|cccccccccccccc}
$j$ & $2$ & $3$ & $4$ & $5$ & $6$ & $7$ & $8$ & $9$ & $10$ & $11$ & $12$ & $%
13$ & $14$ &  \\ \hline
$\func{sgn}(j)$ & $-$ & $-$ & $+$ & $+$ & $-$ & $+$ & $-$ & $+$ & $-$ & $+$
& $+$ & $-$ & $+$ &  \\ 
$\mu(j)$ & 1 & 1 & 1 & 1 & 5 & 5 & 2 & 2 & 7 & 7 & 9 & 9 & 8 & 
\end{tabular}
\end{center}

Now that the $7$ outputs are in place, we take the set of $j$ for which $%
\mu(j) = 7$, which is $10-, 11+$, and put them in the queue with plusses
first followed by minuses: 
\begin{equation*}
\text{Queue: }\text{\sout{1, 4, 5, 2, 3, 7}, 6, 9, 8, 11, 10}
\end{equation*}
There are no $j$ for which $\mu(j)=6$, so we proceed on the queue to $9$.
However, the two $9$'s are already in place, and the next on the list is $8$%
, and the one $8$ is already in place. So this completes the example.

\bigskip

We now describe the above algorithm in general.

\begin{algorithm}
Given $(\mu,\func{sgn})$, start with a queue $Q$ that initially contains $1$%
, and a \emph{marker} $j$, which is initially set to $j=2$. Repeat the
following steps:

\begin{enumerate}
\item Dequeue the leftmost entry $\ell$ of the queue. If the queue is empty,
then stop. Clear the temporary ordered list $L$.

\item If $\mu(j)=\ell$, add $j$ to the right of $L$, then increment the
marker $j$ by $1$ (so now $j$ is the old $j+1$). If (the new marker) $j$ is
out of range, jump to Step (4). If $\mu(j)\neq \ell$, then proceed to Step
(3); otherwise repeat Step (2).

\item Find the smallest $r\geq j+1$ such that $\mu(r) = \ell$ (if there is
no such $r$, jump to Step (4)). Execute signed KM moves $\text{KM}(r-1,r)$,
followed by $\text{KM}(r-2,r-1)$, $\ldots$, until $\text{KM}(j+1,j)$. Now $%
\mu(j)=\ell$. Return to Step (2).

\item Take all elements of the temporary ordered list $L$, read all $+$
entries in order (from left to right) and add them to the (right end of the)
queue $Q$, then read all $-$ entries in order (from left to right) and add
them to the (right end of the) queue $Q$. Return to Step (1).
\end{enumerate}
\end{algorithm}

\bigskip

We have the following adaptation of Proposition \ref%
{Prop:KMIntegrationLimits}, revised to include sign maps and to reference
tamed form in place of upper echelon form.

\begin{proposition}
Within a signed KM-relatable equivalence class of collapsing map/sign map
pairs $(\mu ,\func{sgn})$, there is a unique tamed $(\mu _{\ast },\func{sgn}%
_{\ast })$. Moreover, 
\begin{equation}
\sum_{(\mu ,\func{sgn})\sim (\mu _{\ast },\func{sgn}_{\ast })}I(\mu ,\func{id%
},\func{sgn},\gamma ^{(k+1)})=\int_{T_{D}(\mu _{\ast })}J_{\mu _{\ast },%
\func{sgn}_{\ast }}(\gamma ^{(k+1)})(t_{1},\underline{t}_{k+1})\,d\underline{%
t}_{k+1}  \label{E:equiv-reduce-tamed}
\end{equation}%
where $T_{D}(\mu _{\ast })$ is defined in \eqref{E:TD-def}.
\end{proposition}

To proceed with our program, we divide the expansion %
\eqref{E:gamma-expansion} into sums over signed KM-relatable equivalence
class, and apply \eqref{E:equiv-reduce-tamed} for the sum over each
equivalence class. Thus we obtain 
\begin{equation}
\gamma ^{(1)}(t_{1})=\sum_{(\mu _{\ast },\func{sgn}_{\ast })\text{ tamed}%
}\int_{T_{D}(\mu _{\ast })}J_{\mu _{\ast },\func{sgn}_{\ast }}(\gamma
^{(k+1)})(t_{1},\underline{t}_{k+1})\,d\underline{t}_{k+1}
\label{E:gamma-expansion2}
\end{equation}

The next step will be to round up the tamed pairs $(\mu _{\ast },\func{sgn}%
_{\ast })$ via \emph{wild moves}, as defined and discussed in the next
section. This will produce a further reduction of \eqref{E:gamma-expansion2}.

\subsection{Wild Moves\label{Sec:Wild4}}

\begin{definition}
\label{D:wild} A \emph{wild move} $\text{W}(\rho)$ is defined as follows.
Suppose $(\mu,\func{sgn})$ is a collapsing operator/sign map pair in tamed
form, and $\{\ell, \ldots, r\}$ is a full left branch, i.e. 
\begin{equation*}
z \overset{\mathrm{def}}{=} \mu(\ell)=\mu(\ell+1)=\cdots = \mu(r)
\end{equation*}
but $\mu(\ell-1)\neq z$ (or is undefined) and $\mu(r+1) \neq z$ (or is
undefined).

Let $\rho$ be a permutation of $\{\ell,\ell+1, \ldots, r \}$ that satisfies
the following condition: if $\ell\leq q<s\leq r$ and $\func{sgn}(q)=\func{sgn%
}(s)$, then $q$ appears before $s$ in the list $(\rho^{-1}(\ell), \ldots,
\rho^{-1}(r))$ (equivalently, $\rho(q)<\rho(s)$)

Then the wild move $\text{W}(\rho)$ is defined as an action on a triple $%
(\mu,\sigma, \func{sgn})$, where 
\begin{equation*}
(\mu^{\prime },\sigma^{\prime },\func{sgn}^{\prime }) = W(\rho)(\mu,\sigma,%
\func{sgn})
\end{equation*}
provided 
\begin{equation*}
\begin{aligned} &\mu' = \rho \circ \mu = \rho \circ \mu \circ \rho^{-1}\\
&\sigma' = \rho \circ \sigma \\ &\operatorname{sgn}' = \operatorname{sgn}
\circ \rho^{-1} \\ \end{aligned}
\end{equation*}
We note that $W$ is an action 
\begin{equation*}
W(\rho_1)W(\rho_2) = W(\rho_1 \circ\rho_2)
\end{equation*}
\end{definition}

It is fairly straightforward to show, using the definition of tamed form,
the following. It is important to note that the analogous statement for
upper echelon forms does not hold, and it is the purpose of introducing the
tamed class.

\begin{proposition}
\label{P:tamed-to-tamed} Suppose $(\mu,\func{sgn})$ is a collapsing
operator/sign map pair in tamed form, and $W(\rho)$ is a wild move as
defined above. Letting $(\mu^{\prime }, \func{sgn}^{\prime })$ be the
output, i.e. 
\begin{equation*}
(\mu^{\prime },\sigma^{\prime }, \func{sgn}^{\prime }) = W(\rho)
(\mu,\sigma, \func{sgn})
\end{equation*}
then $(\mu^{\prime },\func{sgn}^{\prime })$ is also tamed.
\end{proposition}

Thus wild moves preserve the tamed class, and we can say that two tamed
forms $(\mu,\func{sgn})$ and $(\mu^{\prime },\func{sgn}^{\prime })$ are 
\emph{wildly relatable} if there exists $\rho$ as in Definition \ref{D:wild}
such that 
\begin{equation*}
(\mu^{\prime },\sigma^{\prime }, \func{sgn}^{\prime }) = W(\rho)(\mu,\sigma, 
\func{sgn})
\end{equation*}
This is an equivalence relation, and in the sum \eqref{E:gamma-expansion2},
we can partition the class of tamed pairs $(\mu,\func{sgn})$ into
equivalence classes of wildly relatable forms (we pursue this in the next
section).

The main result of this section is

\begin{proposition}
\label{P:wild-preserves} Suppose that $\rho$ is as in Definition \ref{D:wild}
and 
\begin{equation*}
(\mu^{\prime }, \sigma^{\prime }, \func{sgn}^{\prime }) = W(\rho)(\mu,
\sigma, \func{sgn})
\end{equation*}
Then for any symmetric density $f^{(k+1)}$, 
\begin{equation*}
J_{\mu^{\prime },\func{sgn}^{\prime }}( f^{(k+1)})(t_1, {\sigma^{\prime }}%
^{-1}(\underline{t}_{k+1})) = J_{\mu,\func{sgn}}( f^{(k+1)})(t_1,
\sigma^{-1}(\underline{t}_{k+1}))
\end{equation*}
Consequently, the Duhamel integrals are preserved, after adjusting for the
time permutations 
\begin{equation*}
\int_{\sigma^{\prime }[T_D(\mu^{\prime })]} J_{\mu^{\prime },\func{sgn}%
^{\prime }}( \gamma^{(k+1)})(t_1, \underline{t}_{k+1}) \, d\underline{t}%
_{k+1} = \int_{\sigma[T_D(\mu)]} J_{\mu,\func{sgn}}( \gamma^{(k+1)})(t_1, 
\underline{t}_{k+1}) \, d\underline{t}_{k+1}
\end{equation*}
where $\sigma[T_D(\mu)]$ is defined by modifying \eqref{E:TD-def} so that
nodes labels are pushed forward by $\sigma$: 
\begin{equation*}
\begin{aligned} \sigma[T_D(\mu)] = \{\; t_{\sigma(j)}\geqslant t_{\sigma(k)}
\; : \; & j,k\text{ are labels on nodes of }\alpha \text{ such}\\
&\text{that the }k\text{ node is a child of the }j\text{ node} \; \}
\end{aligned}
\end{equation*}
\end{proposition}

\begin{proof}
A permutation $\rho$ of the type described in Definition \ref{D:wild} can be
written as a composition of permutations 
\begin{equation*}
\rho = \tau_1 \circ \cdots \circ \tau_s
\end{equation*}
with the property that each $\tau = (i,i+1)$ for some $i\in \{\ell, \ldots,
\ell+r\}$ and $\func{sgn}(i) \neq \func{sgn}(i+1)$. Thus it suffices to
prove 
\begin{align*}
\hspace{0.3in}&\hspace{-0.3in} U^{(i-1)}(-t_i) B_{\mu(i),i}^- U^{(i)} (t_i
-t_{i+1}) B_{\mu(i+1),i+1}^+ U^{(i+1)}(t_{i+1}) \\
&= U^{(i-1)}(-t_{i+1}) B_{\mu(i),i}^+ U^{(i)} (t_{i+1} -t_i)
B_{\mu(i+1),i+1}^- U^{(i+1)}(t_i)
\end{align*}
when the two sides act on a symmetric density. Recall that $%
z=\mu(i)=\mu(i+1) $. Without loss, we might as well take $z=1$ and $i=2$ so
that this becomes 
\begin{equation}  \label{E:main-wild-id}
U^{(1)}(-t_2) B_{1,2}^- U^{(2)} (t_2 -t_3) B_{1,3}^+ U^{(3)}(t_3)=
U^{(1)}(-t_3) B_{1,2}^+ U^{(2)} (t_3 -t_2) B_{1,3}^- U^{(3)}(t_2)
\end{equation}
To prove \eqref{E:main-wild-id}, on the left side, we proceed as follows.
First, we'll plug in: 
\begin{equation*}
U^{(1)}(-t_2) = U^1_{-2} U^{1^{\prime }}_{2}
\end{equation*}
\begin{equation*}
U^{(2)}(t_2-t_3) = U_2^1 U_{-3}^1 U_{-2}^{1^{\prime }} U_{3}^{1^{\prime }}
U_2^2 U_{-3}^2 U_{-2}^{2^{\prime }} U_3^{2^{\prime }}
\end{equation*}
\begin{equation*}
U^{(3)}(t_3) = U_3^1 U_{-3}^{1^{\prime }} U_3^2 U_{-3}^{2^{\prime }} U_3^3
U_{-3}^{3^{\prime }}
\end{equation*}
where the subscript indicates the time variable, and the superscript
indicates the spatial variable. Then we note that for the two collapsing
operators on the left side of \eqref{E:main-wild-id}

\begin{itemize}
\item $B_{1,2}^-$ acts only on the $2$, $2^{\prime }$, and $1^{\prime }$
coordinates, so we can move all $U^1$ operators in the middle to the left

\item $B_{1,3}^+$ acts only on the $3$, $3^{\prime }$, and $1$ coordinates,
we can move all $U^2$, $U^{2^{\prime }}$, and $U^{1^{\prime }}$ operators in
the middle to the right.
\end{itemize}

This results in 
\begin{equation}  \label{E:main-wild-id-left}
\text{left side of }\eqref{E:main-wild-id} = U_{-3}^1 U_2^{1^{\prime }}
B_{1,2}^- B_{1,3}^+ U_3^1 U_{-2}^{1^{\prime }} U_2^2 U_{-2}^{2^{\prime }}
U_3^3 U_{-3}^3
\end{equation}

Similarly, on the right side of \eqref{E:main-wild-id}, plug in: 
\begin{equation*}
U^{(1)}(-t_3) = U_{-3}^1 U_3^{1^{\prime }}
\end{equation*}
\begin{equation*}
U^{(2)}(t_3-t_2) = U_3^1 U_{-2}^1 U_{-3}^{1^{\prime }} U_{2}^{1^{\prime }}
U_3^2 U_{-2}^2 U_{-3}^{2^{\prime }} U_2^{2^{\prime }}
\end{equation*}
\begin{equation*}
U^{(3)}(t_2) = U_2^1 U_{-2}^{1^{\prime }} U_2^2 U_{-2}^{2^{\prime }} U_2^3
U_{-2}^{3^{\prime }}
\end{equation*}
Then we note that for the two collapsing operators on the right side of %
\eqref{E:main-wild-id},

\begin{itemize}
\item $B_{1,2}^+$ acts only on the $2$, $2^{\prime }$, and $1$ coordinates,
so we can move all $U^{1^{\prime }}$ operators in the middle to the left

\item $B_{1,3}^-$ acts only on the $3$, $3^{\prime }$, and $1^{\prime }$
coordinates, we can move all $U^2$, $U^{2^{\prime }}$, and $U^{1}$ operators
in the middle to the right.
\end{itemize}

This results in 
\begin{equation}  \label{E:main-wild-id-right}
\text{right side of }\eqref{E:main-wild-id} = U_{-3}^1 U_2^{1^{\prime }}
B_{1,2}^+ B_{1,3}^- U_3^1 U_{-2}^{1^{\prime }} U_3^2 U_{-3}^{2^{\prime }}
U_2^3 U_{-2}^{3^{\prime }}
\end{equation}
Since \eqref{E:main-wild-id-left} and \eqref{E:main-wild-id-right} are equal
when applied to a symmetric density, this proves equality %
\eqref{E:main-wild-id}. In particular, one just needs that to permute 
\begin{equation*}
(x_2,x_2^{\prime },x_3,x_3^{\prime }) \leftrightarrow (x_3,x_3^{\prime
},x_2,x_2^{\prime })
\end{equation*}
\end{proof}

\begin{example}
\label{EX:wild} Starting from the pair $(\mu_1, \func{sgn}_1)$, defined as
follows

\bigskip

\begin{equation*}
\begin{tabular}{c|cccccc}
& $2$ & $3$ & $4$ & $5$ & $6$ & $7$ \\ \hline
$\mu _{1}$ & $1$ & $1$ & $1$ & $2$ & $4$ & $4$ \\ 
$\func{sgn}_{1}$ & $+$ & $+$ & $-$ & $-$ & $+$ & $-$%
\end{tabular}%
\end{equation*}

\bigskip

there are five nontrivial wild moves for $j = 2, \ldots, 6$:

\begin{equation*}
(\mu_j,\sigma_j, \func{sgn}_j) = W(\rho_j)(\mu_1, \text{id}, \func{sgn}_1)
\end{equation*}

as indicated in the table below

\bigskip

\renewcommand{\arraystretch}{1.2} 
\begin{equation*}
\begin{tabular}{c|ccc|cc||c|ccc|cc}
& $2$ & $3$ & $4$ & $6$ & $7$ &  & $2$ & $3$ & $4$ & $6$ & $7$ \\ \hline
$\rho _{1}$ & $2$ & $3$ & $4$ & $6$ & $7$ & $\rho _{1}^{-1}$ & $2$ & $3$ & $%
4 $ & $6$ & $7$ \\ 
$\rho _{2}$ & $2$ & $4$ & $3$ & $6$ & $7$ & $\rho _{2}^{-1}$ & $2$ & $4$ & $%
3 $ & $6$ & $7$ \\ 
$\rho _{3}$ & $3$ & $4$ & $2$ & $6$ & $7$ & $\rho _{3}^{-1}$ & $4$ & $2$ & $%
3 $ & $6$ & $7$ \\ 
$\rho _{4}$ & $2$ & $3$ & $4$ & $7$ & $6$ & $\rho _{4}^{-1}$ & $2$ & $3$ & $%
4 $ & $7$ & $6$ \\ 
$\rho _{5}$ & $2$ & $4$ & $3$ & $7$ & $6$ & $\rho _{5}^{-1}$ & $2$ & $4$ & $%
3 $ & $7$ & $6$ \\ 
$\rho _{6}$ & $3$ & $4$ & $2$ & $7$ & $6$ & $\rho _{6}^{-1}$ & $4$ & $2$ & $%
3 $ & $7$ & $6$%
\end{tabular}%
\end{equation*}

\bigskip

Notice that each $\rho_j^{-1}$ preserves the order of $2, 3$, as in
Definition \ref{D:wild} (meaning that $2$ appears before $3$ in the list $%
(\rho_j^{-1}(2), \rho_j^{-1}(3), \rho_j^{-1}(4))$), equivalently $%
\rho(2)<\rho(3)$. Thus the action of $\rho^{-1}$ on $\{2,3,4\}$ is
completely determined by where $4$ appears in the list $(\rho_j^{-1}(2),
\rho_j^{-1}(3), \rho_j^{-1}(4))$.

The corresponding trees and explicit mappings $(\mu_j,\func{sgn}_j)$ are
indicated below. We notice that all $(\mu_j,\func{sgn}_j)$ are tamed (in
accordance with Proposition \ref{P:tamed-to-tamed}) and also that the wild
moves, unlike the KM moves, do change the tree skeleton, but this change is
restricted to shuffling nodes along a left branch, subject to the
restrictions (indicated in Definition \ref{D:wild}) that the ordering of the
plus nodes and ordering of the minus nodes remain in tact.
\end{example}

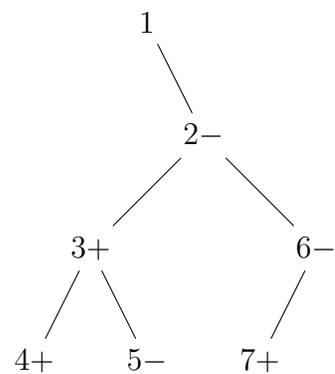
\begin{minipage}[t]{0.33\linewidth}
\begin{center}
Tree for $(\mu_1,\sgn_1)$

\bigskip

\begin{tabular}{ c |  c c c c c c}
& $2$ & $3$ & $4$ & $5$ & $6$ & $7$ \\
\hline
$\mu_1$ & $1$ & $1$ & $1$ & $2$ & $4$ & $4$ \\
$\sgn_1$ & $+$ & $+$ & $-$ & $-$ & $+$ & $-$
\end{tabular}

\begin{tikzpicture}
\node{$1$}
	child[missing]
	child{node{$2+$}
		child{node{$3+$}
			child{node{$4-$}
				child[missing]
				child{node{$6+$}
					child{node{$7-$}}
					child[missing]
				}
			}
			child[missing]
		}
		child{node{$5-$}	}	
	};
\end{tikzpicture}
\end{center}
\end{minipage}
\begin{minipage}[t]{0.33\linewidth}

\begin{center}
Tree for $(\mu_2,\sgn_2)$

\bigskip

\begin{tabular}{ c |  c c c c c c}
& $2$ & $3$ & $4$ & $5$ & $6$ & $7$ \\
\hline
$\mu_2$ & $1$ & $1$ & $1$ & $2$ & $3$ & $3$ \\
$\sgn_2$ & $+$ & $-$ & $+$ & $-$ & $+$ & $-$
\end{tabular}

\bigskip

\begin{tikzpicture}
\node{$1$}
	child[missing]
	child{node{$2+$}
		child{node{$3-$}
			child{node{$4+$}}
			child{node{$6+$}
				child{node{$7-$}}
				child[missing]
			}
		}
		child{node{$5-$}}
	};
\end{tikzpicture}
\end{center}
\end{minipage}
\begin{minipage}[t]{0.33\linewidth}

\begin{center}
Tree for $(\mu_3,\sgn_3)$

\bigskip

\begin{tabular}{ c |  c c c c c c}
& $2$ & $3$ & $4$ & $5$ & $6$ & $7$ \\
\hline
$\mu_3$ & $1$ & $1$ & $1$ & $3$ & $2$ & $2$ \\
$\sgn_3$ & $-$ & $+$ & $+$ & $-$ & $+$ & $-$
\end{tabular}

\bigskip

\begin{tikzpicture}
\node{$1$}
	child[missing]
	child{node{$2-$}
		child{node{$3+$}
			child{node{$4+$}}
			child{node{$5-$}}
		}
		child[missing]
		child{node{$6+$}
			child{node{$7-$}}
			child[missing]
		}
	};
\end{tikzpicture}
\end{center}
\end{minipage}

\bigskip

\begin{minipage}[t]{0.33\linewidth}
\begin{center}
Tree for $(\mu_4,\sgn_4)$

\bigskip

\begin{tabular}{ c |  c c c c c c}
& $2$ & $3$ & $4$ & $5$ & $6$ & $7$ \\
\hline
$\mu_4$ & $1$ & $1$ & $1$ & $2$ & $4$ & $4$ \\
$\sgn_4$ & $+$ & $+$ & $-$ & $-$ & $-$ & $+$
\end{tabular}

\begin{tikzpicture}
\node{$1$}
	child[missing]
	child{node{$2+$}
		child{node{$3+$}
			child{node{$4-$}
				child[missing]
				child{node{$6-$}
					child{node{$7+$}}
					child[missing]
				}
			}
			child[missing]
		}
		child{node{$5-$}	}	
	};
\end{tikzpicture}
\end{center}
\end{minipage}
\begin{minipage}[t]{0.33\linewidth}

\begin{center}
Tree for $(\mu_5,\sgn_5)$

\bigskip

\begin{tabular}{ c |  c c c c c c}
& $2$ & $3$ & $4$ & $5$ & $6$ & $7$ \\
\hline
$\mu_5$ & $1$ & $1$ & $1$ & $2$ & $3$ & $3$ \\
$\sgn_5$ & $+$ & $-$ & $+$ & $-$ & $-$ & $+$
\end{tabular}

\bigskip

\begin{tikzpicture}
\node{$1$}
	child[missing]
	child{node{$2+$}
		child{node{$3-$}
			child{node{$4+$}}
			child{node{$6-$}
				child{node{$7+$}}
				child[missing]
			}
		}
		child{node{$5-$}}
	};
\end{tikzpicture}
\end{center}
\end{minipage}
\begin{minipage}[t]{0.33\linewidth}

\begin{center}
Tree for $(\mu_6,\sgn_6)$

\bigskip

\begin{tabular}{ c |  c c c c c c}
& $2$ & $3$ & $4$ & $5$ & $6$ & $7$ \\
\hline
$\mu_6$ & $1$ & $1$ & $1$ & $3$ & $2$ & $2$ \\
$\sgn_6$ & $-$ & $+$ & $+$ & $-$ & $-$ & $+$
\end{tabular}

\bigskip

\begin{tikzpicture}
\node{$1$}
	child[missing]
	child{node{$2-$}
		child{node{$3+$}
			child{node{$4+$}}
			child{node{$5-$}}
		}
		child[missing]
		child{node{$6-$}
			child{node{$7+$}}
			child[missing]
		}
	};
\end{tikzpicture}
\end{center}
\end{minipage}

\newpage 

\subsection{Reference Forms and the Tamed Integration Domains\label%
{Sec:Wild5}}

\begin{definition}
A tamed pair $(\hat \mu,\hat{\func{sgn}})$ will be called a \emph{reference}
pair provided that in every left-branch, all the $+$ nodes come before all
the $-$ nodes.
\end{definition}

\begin{definition}
Given a reference pair $(\hat\mu, \hat{\func{sgn}})$, we will call a
permutation $\rho$ of $\{2, \ldots, k+1\}$ \emph{allowable} if it meets the
conditions in Definition \ref{D:wild}, i.e. it leaves all left branches
invariant and moreover, for each left branch $(\ell, \ldots, r)$, all plus
nodes appear in their original order and all minus nodes appear in their
original order within the list $(\rho^{-1}(\ell), \ldots, \rho^{-1}(r))$.
\end{definition}

For example, Tree $(\mu_1,\func{sgn}_1)$ in the Example \ref{EX:wild} is a
reference pair. If $(\ell, \ldots, r)$ is a full left branch of $\hat \mu$,
then the definition of reference pair means that there is some intermediate
position $m$ such that the $\func{sgn}$ map looks like this

\begin{center}
\begin{tabular}{c|ccccccc}
$j$ & $\ell$ & $\cdots$ & $m-1$ & $m$ & $m+1$ & $\cdots$ & $r$ \\ \hline
$\func{sgn}$ & $+$ & $+$ & $+$ & $-$ & $-$ & $-$ & $-$%
\end{tabular}
\end{center}

However, we note that it is possible that they are all plusses ($m=r+1$) or
they are all minuses $m=\ell$. With this notation, we can say that $\rho$ is
allowable if $\rho(\ell)< \cdots < \rho(m-1)$ and $\rho(m) < \cdots <
\rho(r) $, or equivalently, in the list 
\begin{equation*}
(\rho^{-1}(\ell), \ldots, \rho^{-1}(r))
\end{equation*}
the values $(\ell, \ldots, m-1)$ appear in that order and the values $(m,
\ldots, r)$ appear in that order.

\begin{proposition}
\label{P:wild-pool} An equivalence class of wildly relatable tamed pairs 
\begin{equation*}
Q = \{ (\mu, \func{sgn}) \}
\end{equation*}
contains a unique reference pair $(\hat \mu, \hat{\func{sgn}})$. By the
definition of wildly relatable, for every $(\mu,\func{sgn})\in Q$, there is
a unique permutation $\rho$ of $\{2, \ldots, k+1\}$ such that 
\begin{equation*}
(\mu,\func{sgn}) = W(\rho)(\hat \mu, \hat{\func{sgn}})
\end{equation*}
and this $\rho$ is allowable. The collection $P$ of all $\rho$ arising in
this way from $Q$ is exactly the set of \emph{all} allowable $\rho$ with
respect to the reference pair $(\hat \mu, \hat{\func{sgn}})$.
\end{proposition}

Now, recall \eqref{E:gamma-expansion2}: 
\begin{equation*}
\gamma ^{(1)}(t_{1})=\sum_{(\mu ,\func{sgn})\text{ tamed}}\int_{T_{D}(\mu
)}J_{\mu ,\func{sgn}}(\gamma ^{(k+1)})(t_{1},\underline{t}_{k+1})\,d%
\underline{t}_{k+1}
\end{equation*}%
In this sum, group together equivalence classes $Q$ of wildly relatable $%
(\mu ,\func{sgn})$. 
\begin{equation}
\gamma ^{(1)}(t_{1})=\sum_{\text{classes }Q}\sum_{(\mu ,\sigma )\in
Q}\int_{T_{D}(\mu )}J_{\mu ,\func{sgn}}(\gamma ^{(k+1)})(t_{1},\underline{t}%
_{k+1})\,d\underline{t}_{k+1}  \label{E:gamma-expansion3}
\end{equation}%
Each class $Q$ can be represented by a unique reference $(\hat{\mu},\hat{%
\func{sgn}})$, and as in Proposition \ref{P:wild-pool} for each $(\mu ,\func{%
sgn})\in Q$, there is an allowable $\rho \in P$ (with respect to $(\hat{\mu},%
\hat{\func{sgn}})$) such that 
\begin{equation*}
(\mu ,\func{sgn})=W(\rho )(\hat{\mu},\hat{\func{sgn}})
\end{equation*}%
Since $W$ is an action, we can write 
\begin{equation*}
(\hat{\mu},\hat{\func{sgn}})=W(\rho ^{-1})(\mu ,\func{sgn})
\end{equation*}%
Into the action $W(\rho ^{-1})$, let us input the identity time permutation
and define $\sigma $ as the output time permutation, i.e. 
\begin{equation*}
(\hat{\mu},\sigma ,\hat{\func{sgn}})=W(\rho ^{-1})(\mu ,\text{id},\func{sgn})
\end{equation*}%
where, in accordance with Definition \ref{D:wild}, $\sigma =\rho ^{-1}$.
Since $\rho $ is allowable, this implies that for each left brach $(\ell
,\ldots ,r)$ with $m$ as defined above, $\sigma ^{-1}(\ell )<\cdots <\sigma
^{-1}(m-1)$ and $\sigma ^{-1}(m)<\cdots <\sigma ^{-1}(r)$. In other words, $%
(\ell ,\ldots ,m-1)$ and $(m,\ldots ,r)$ appear in order inside the list of
values $(\sigma (\ell ),\ldots ,\sigma (r))$. By Proposition \ref%
{P:wild-preserves} 
\begin{equation*}
\int_{T_{D}(\mu )}J_{\mu ,\func{sgn}}(\gamma ^{(k+1)})(t_{1},\underline{t}%
_{k+1})\,d\underline{t}_{k+1}=\int_{\sigma \lbrack T_{D}(\hat{\mu})]}J_{\hat{%
\mu},\hat{\func{sgn}}}(\gamma ^{(k+1)})(t_{1},\underline{t}_{k+1})\,d%
\underline{t}_{k+1}
\end{equation*}%
Now as we sum this over all $(\mu ,\func{sgn})\in Q$, we are summing over
all $\rho \in P$ and hence over all $\sigma =\rho ^{-1}$ meeting the
condition mentioned above. Hence the integration domains on the right side
union to a set that we will denote 
\begin{equation*}
T_{R}(\hat{\mu},\hat{\func{sgn}})\overset{\mathrm{def}}{=}\bigcup_{\rho \in
P}\sigma (T_{D}(\hat{\mu}))
\end{equation*}%
that can be described as follows: for each left branch $(\ell ,\ldots ,r)$,
with 
\begin{equation*}
z=\mu (\ell )=\cdots =\mu (r)
\end{equation*}%
and $m$ the division index between plus and minus nodes (as defined above), $%
T_{R}(\hat{\mu},\hat{\func{sgn}})$ is described by the inequalities 
\begin{equation}
t_{m-1}\leq \cdots \leq t_{\ell }\leq t_{z}\quad \text{and}\quad t_{r}\leq
\cdots \leq t_{m}\leq t_{z}  \label{eqn:reference limits}
\end{equation}%
Plugging into \eqref{E:gamma-expansion3}, we obtain

\begin{proposition}
\label{Prop:ref form}The Duhamel expansion to coupling order $k$ can be
grouped into at most $8^{k}$ terms: 
\begin{equation}
\gamma ^{(1)}(t_{1})=\sum_{\text{reference }(\hat{\mu},\hat{\func{sgn}}%
)}\int_{T_{R}(\hat{\mu},\hat{\func{sgn}})}J_{\mu ,\func{sgn}}(\gamma
^{(k+1)})(t_{1},\underline{t}_{k+1})\,d\underline{t}_{k+1}
\label{E:gamma-expansion4}
\end{equation}%
where each integration domain $T_{R}(\hat{\mu},\hat{\func{sgn}})$ is as
defined in (\ref{eqn:reference limits}).
\end{proposition}

Returning to Example \ref{EX:wild}, $(\mu_1,\func{sgn}_1)$ is the reference
pair. To combine the Duhamel integrals as above, we convert all other five
tamed forms $(\mu_j,\func{sgn}_j)$ to $(\mu_1,\func{sgn}_1)$ via wild moves.
The resulting combined time integration set will be read off from the $%
(\mu_1,\func{sgn}_1)$ tree as 
\begin{equation*}
t_3\leq t_2 \leq t_1 \,, \qquad t_4 \leq t_1 \,, \qquad t_5 \leq t_2 \,,
\qquad t_6 \leq t_4 \,, \qquad t_7 \leq t_4
\end{equation*}

Proposition \ref{Prop:ref form} and the integration domain (\ref%
{eqn:reference limits}) is compatible with the $U$-$V$ space techniques we
proved in \S \ref{Sec:UVandTrilinear}. This fact may not be so clear at the
moment as they are written with many shorthands. We will prove this fact in 
\S \ref{Sec:WhyWeNeedExtKM}.

\section{The Uniqueness for GP Hierarchy (\protect\ref{Hierarchy:T^4 cubic
GP}) - Actual Estimates\label{Sec:GP Uniqueness-2}}

The main goal of this section is to prove Proposition \ref%
{Prop:KeyUniquenessEstimate} on estimating $J_{\mu _{m},sgn}^{(k+1)}$. Of
course, by writting $J_{\mu _{m},sgn}^{(k+1)}$, we mean the reference form
now. We first present an example in \S \ref{sec:key example} to convey the
basic ideas of the proof. We then, in \S \ref{Sec:WhyWeNeedExtKM},
demonstrate why we need the extended KM board game and prove that
Proposition \ref{Prop:ref form} and the integration domain (\ref%
{eqn:reference limits}) is compatible with the $U$-$V$ space techniques.
Once that is settled, the main idea idea in \S \ref{sec:key example} will
work for the general case. Thus we estimate the general case in \S \ref%
{Sec:pf of the general case}.

The time integration limits in \S \ref{Sec:Wild5} will be put to use with
Lemmas \ref{Lem:T4TrilinearWithUV} and \ref{Lem:T4TrilinearWithUVH1}. With
the trivial estimate $\left\Vert u\right\Vert _{Y^{s}}\lesssim \left\Vert
u\right\Vert _{X^{s}}$, Lemmas \ref{Lem:T4TrilinearWithUV} and \ref%
{Lem:T4TrilinearWithUVH1} read 
\begin{eqnarray}
&&\left\Vert \int_{a}^{t}e^{-i(t-t^{\prime })\Delta }\left(
u_{1}u_{2}u_{3}\right) (\bullet ,t^{\prime })\,dt^{\prime }\right\Vert
_{X^{-1}}  \label{estimate:Duhamel mlinear-1} \\
&\leqslant &C\Vert u_{1}\Vert _{X^{-1}}\left( T^{\frac{1}{7}}M_{0}^{\frac{3}{%
5}}\Vert P_{\leqslant M_{0}}u_{2}\Vert _{X^{1}}+\Vert P_{>M_{0}}u_{2}\Vert
_{X^{1}}\right) \Vert u_{3}\Vert _{X^{1}},  \notag
\end{eqnarray}%
\begin{eqnarray}
&&\left\Vert \int_{a}^{t}e^{-i(t-t^{\prime })\Delta }\left(
u_{1}u_{2}u_{3}\right) (\bullet ,t^{\prime })\,dt^{\prime }\right\Vert
_{X^{1}}  \label{estimate:Duhamel mlinear-2} \\
&\leqslant &C\Vert u_{1}\Vert _{X^{1}}\left( T^{\frac{1}{7}}M_{0}^{\frac{3}{5%
}}\Vert P_{\leqslant M_{0}}u_{2}\Vert _{X^{1}}+\Vert P_{>M_{0}}u_{2}\Vert
_{X^{1}}\right) \Vert u_{3}\Vert _{X^{1}},  \notag
\end{eqnarray}%
\begin{equation}
\left\Vert \int_{a}^{t}e^{-i(t-t^{\prime })\Delta }\left(
u_{1}u_{2}u_{3}\right) (\bullet ,t^{\prime })\,dt^{\prime }\right\Vert
_{X^{-1}}\leqslant C\Vert u_{1}\Vert _{X^{-1}}\Vert u_{2}\Vert _{X^{1}}\Vert
u_{3}\Vert _{X^{1}},  \label{estimate:Duhamel mlinear-3}
\end{equation}%
\begin{equation}
\left\Vert \int_{a}^{t}e^{-i(t-t^{\prime })\Delta }\left(
u_{1}u_{2}u_{3}\right) (\bullet ,t^{\prime })\,dt^{\prime }\right\Vert
_{X^{1}}\leqslant C\Vert u_{1}\Vert _{X^{1}}\Vert u_{2}\Vert _{X^{1}}\Vert
u_{3}\Vert _{X^{1}}.  \label{estimate:Duhamel mlinear-4}
\end{equation}%
If $u_{j}=e^{it^{\prime }\Delta }f_{j}$ with for some $j$ and some $f_{j}$
independent of $t$ and $t^{\prime }$, we can replace the $X^{s}$ norm of $%
u_{j}$ in (\ref{estimate:Duhamel mlinear-1})-(\ref{estimate:Duhamel
mlinear-4}) with the $H^{s}$ norm of $f_{j}$. We do not use "$\lesssim $" in
(\ref{estimate:Duhamel mlinear-1})-(\ref{estimate:Duhamel mlinear-4}) as we
are going to use them repeatly and the constants are going to accumulate.

\subsection{An Example of How to Estimate\label{sec:key example}}

We estimate the integral in Example \ref{example:one could indeed combine
more} 
\begin{equation*}
I=\int_{t_{4}=0}^{t_{1}}\int_{t_{2}=0}^{t_{1}}%
\int_{t_{3}=t_{4}}^{t_{1}}U^{(1)}(t_{1}-t_{2})B_{1,2}^{+}U^{(2)}(t_{2}-t_{3})B_{1,3}^{-}U^{(3)}(t_{3}-t_{4})B_{3,4}^{+}\gamma ^{(4)}d%
\underline{t}_{4}
\end{equation*}%
where the integration limits are already computed in \S \ref{Sec:Wild1}. Its
reference tree is exactly the tree corresponding to $I_{2}$ in Example \ref%
{example:one could indeed combine more}.

Plug in (\ref{eqn:qdF Rep for zero initial solution}), we find the integrand
is in fact 
\begin{equation*}
I=\int_{t_{4}=0}^{t_{1}}dt_{4}\int d\mu _{t_{4}}\left( \phi \right)
\int_{t_{2}=0}^{t_{1}}dt_{2}\int_{t_{3}=t_{4}}^{t_{1}}U_{1,2}(\left\vert
U_{2,4}\phi \right\vert ^{2}U_{2,4}\phi )(x_{1})U_{1,3}(\overline{%
U_{3,4}\phi }\overline{U_{3,4}\phi }U_{3,4}\left( \left\vert \phi
\right\vert ^{2}\phi \right) )(x_{1}^{\prime })dt_{3}
\end{equation*}%
We denote the cubic term $\left\vert \phi \right\vert ^{2}\phi $ generated
in the most inner coupling with $\mathcal{C}_{R}^{(4)}$ where the subscript $%
R$ means \textquotedblleft rough\textquotedblright\ as it has no propagator
inside to smooth things out. That is, 
\begin{equation*}
I=\int_{t_{4}=0}^{t_{1}}dt_{4}\int d\mu _{t_{4}}\left( \phi \right)
\int_{t_{2}=0}^{t_{1}}dt_{2}\int_{t_{3}=t_{4}}^{t_{1}}U_{1,2}(\left\vert
U_{2,4}\phi \right\vert ^{2}U_{2,4}\phi )(x_{1})U_{1,3}(\overline{%
U_{3,4}\phi }\overline{U_{3,4}\phi }U_{3,4}\mathcal{C}_{R}^{(4)})(x_{1}^{%
\prime })dt_{3}
\end{equation*}

For (\ref{eqn:a typical term in uniqueness}) with a general $k$, we will use 
$\mathcal{C}_{R}^{(k+1)}$ to denote this most inner cubic term. Notice that, 
$\mathcal{C}_{R}^{(k+1)}$ is always independent of time and is hence
qualified to be an $f_{j}$ in estimates (\ref{estimate:Duhamel mlinear-1}) -
(\ref{estimate:Duhamel mlinear-4}).

In the 2nd coupling, if we denote 
\begin{equation*}
D_{\phi ,R}^{(3)}=U_{-t_{3}}(\overline{U_{3,4}\phi }\overline{U_{3,4}\phi }%
U_{3,4}\mathcal{C}_{R}^{(4)})(x_{1}^{\prime })
\end{equation*}%
we have 
\begin{equation*}
\int_{t_{3}=t_{4}}^{t_{1}}U_{1,3}(\overline{U_{3,4}\phi }\overline{%
U_{3,4}\phi }U_{3,4}\mathcal{C}_{R}^{(4)})(x_{1}^{\prime
})dt_{3}=\int_{t_{3}=t_{4}}^{t_{1}}U_{1}D_{\phi ,R}^{(3)}dt_{3}
\end{equation*}%
In general, let us use $D^{(l+1)}$, which is $D_{\phi ,R}^{(3)}$ here, to
denote the cubic term together with the $U(-t_{l+1})$ during the $l$-th
coupling where $l<k$. We add a $\phi $ subscript if the cubic term,
generated at the $l$-th coupling, has contracted a $U\phi $. We add a $R$
subscript if the cubic term, generated at the $l$-th coupling, has
contracted the rough cubic term $\mathcal{C}_{R}^{(k+1)}$ or a $%
D_{R}^{(j+1)} $ for some $j$. The coupling process makes sure that every
time integral corresponds to one and only one cubic term, thus the notation
of $D$ is well-defined. We suppress all the $t_{k+1}$-dependence, which is
the $t_{4}$-dependence here, in all the $D$ markings as we will not explore
any smoothing giving by the $dt_{k+1}$ integral. Finally, notice that $%
D^{(l+1)}$ always carry the $t_{l+1}$ variable and will make a Duhamel term
whenever it is hit by a $U(t_{j})\ $where $j\neq l+1$.

Then, using the same marking strategy at the 1st coupling, we reach 
\begin{equation*}
I=\int_{t_{4}=0}^{t_{1}}dt_{4}\int d\mu _{t_{4}}\left( \phi \right) \left(
\int_{t_{2}=0}^{t_{1}}U_{1}D_{\phi }^{(2)}\left( x_{1}\right) dt_{2}\right)
\left( \int_{t_{3}=t_{4}}^{t_{1}}U_{1}D_{\phi ,R}^{(3)}(x_{1}^{\prime
})dt_{3}\right)
\end{equation*}%
We can now start estimating. Taking the norm inside, 
\begin{eqnarray*}
&&\left\Vert \left\langle \nabla _{x_{1}}\right\rangle ^{-1}\left\langle
\nabla _{x_{1}^{\prime }}\right\rangle ^{-1}I\right\Vert _{L_{t_{1}}^{\infty
}L_{x,x^{\prime }}^{2}} \\
&\leqslant &\int_{0}^{T}\int dt_{4}d\left\vert \mu _{t_{4}}\right\vert
\left( \phi \right) \left\Vert (\left\langle \nabla _{x_{1}}\right\rangle
^{-1}\int_{t_{2}=0}^{t_{1}}U_{1}D_{\phi }^{(2)}\left( x_{1}\right)
dt_{2})(\left\langle \nabla _{x_{1}^{\prime }}\right\rangle
^{-1}\int_{t_{3}=t_{4}}^{t_{1}}U_{1}D_{\phi ,R}^{(3)}(x_{1}^{\prime
})dt_{3})\right\Vert _{L_{t_{1}}^{\infty }L_{x,x^{\prime }}^{2}}
\end{eqnarray*}%
the $L_{t_{1}}^{\infty }L_{x,x^{\prime }}^{2}$ norm \textquotedblleft
factors\textquotedblright\ in the sense that 
\begin{eqnarray*}
&&\left\Vert \left\langle \nabla _{x_{1}}\right\rangle ^{-1}\left\langle
\nabla _{x_{1}^{\prime }}\right\rangle ^{-1}I\right\Vert _{L_{t_{1}}^{\infty
}L_{x,x^{\prime }}^{2}} \\
&\leqslant &\int_{0}^{T}\int \left\Vert \int_{t_{2}=0}^{t_{1}}U_{1}D_{\phi
}^{(2)}\left( x_{1}\right) dt_{2}\right\Vert _{L_{t_{1}}^{\infty
}H_{x}^{-1}}\left\Vert \int_{t_{3}=t_{4}}^{t_{1}}U_{1}D_{\phi
,R}^{(3)}(x_{1}^{\prime })dt_{3}\right\Vert _{L_{t_{1}}^{\infty
}H_{x^{\prime }}^{-1}}dt_{4}d\left\vert \mu _{t_{4}}\right\vert \left( \phi
\right)
\end{eqnarray*}%
The term $D_{\phi }^{(2)}$ carries no $R$ subscript, we can bump it to $%
H^{1} $ and then use the embedding (\ref{E:H^1/2-embedding}), we have 
\begin{eqnarray*}
&&\left\Vert \left\langle \nabla _{x_{1}}\right\rangle ^{-1}\left\langle
\nabla _{x_{1}^{\prime }}\right\rangle ^{-1}I\right\Vert _{L_{t_{1}}^{\infty
}L_{x,x^{\prime }}^{2}} \\
&\leqslant &\int_{0}^{T}\int \left\Vert
\int_{t_{3}=t_{4}}^{t_{1}}U_{1}D_{\phi ,R}^{(3)}(x_{1}^{\prime
})dt_{3}\right\Vert _{X^{-1}}\left\Vert \int_{t_{2}=0}^{t_{1}}U_{1}D_{\phi
}^{(2)}\left( x_{1}\right) dt_{2}\right\Vert _{X^{1}}dt_{4}d\left\vert \mu
_{t_{4}}\right\vert \left( \phi \right)
\end{eqnarray*}%
Applying (\ref{estimate:Duhamel mlinear-2}) to the 1st coupling and replace
all $\left\Vert U\phi \right\Vert _{X^{s}}$ by $\left\Vert \phi \right\Vert
_{H^{s}}$, we have 
\begin{eqnarray*}
&&\left\Vert \left\langle \nabla _{x_{1}}\right\rangle ^{-1}\left\langle
\nabla _{x_{1}^{\prime }}\right\rangle ^{-1}I\right\Vert _{L_{t_{1}}^{\infty
}L_{x,x^{\prime }}^{2}} \\
&\leqslant &C\int_{0}^{T}\int \left\Vert
\int_{t_{3}=t_{4}}^{t_{1}}U_{1}D_{\phi ,R}^{(3)}(x_{1}^{\prime
})dt_{3}\right\Vert _{X^{-1}}\left\Vert \phi \right\Vert _{H^{1}}^{2} \\
&&\times \left( T^{\frac{1}{7}}M_{0}^{\frac{3}{5}}\Vert P_{\leqslant
M_{0}}\phi \Vert _{H^{1}}+\Vert P_{>M_{0}}\phi \Vert _{H^{1}}\right)
^{1}dt_{4}d\left\vert \mu _{t_{4}}\right\vert \left( \phi \right)
\end{eqnarray*}%
Utilizing (\ref{estimate:Duhamel mlinear-1}) to the 2nd coupling and replace
all $\left\Vert U\phi \right\Vert _{X^{s}}$ by $\left\Vert \phi \right\Vert
_{H^{s}}$, we have 
\begin{eqnarray*}
&&\left\Vert \left\langle \nabla _{x_{1}}\right\rangle ^{-1}\left\langle
\nabla _{x_{1}^{\prime }}\right\rangle ^{-1}I\right\Vert _{L_{t_{1}}^{\infty
}L_{x,x^{\prime }}^{2}} \\
&\leqslant &C^{2}\int_{0}^{T}\int \left\Vert \phi \right\Vert
_{H^{1}}^{3}\left( T^{\frac{1}{7}}M_{0}^{\frac{3}{5}}\Vert P_{\leqslant
M_{0}}\phi \Vert _{H^{1}}+\Vert P_{>M_{0}}\phi \Vert _{H^{1}}\right)
^{2}\left\Vert \mathcal{C}_{R}^{(4)}\right\Vert _{H^{-1}}dt_{4}d\left\vert
\mu _{t_{4}}\right\vert \left( \phi \right) \\
&=&C^{2}\int_{0}^{T}\int \left\Vert \phi \right\Vert _{H^{1}}^{3}\left( T^{%
\frac{1}{7}}M_{0}^{\frac{3}{5}}\Vert P_{\leqslant M_{0}}\phi \Vert
_{H^{1}}+\Vert P_{>M_{0}}\phi \Vert _{H^{1}}\right) ^{2}\left\Vert
\left\vert \phi \right\vert ^{2}\phi \right\Vert _{H^{-1}}dt_{4}d\left\vert
\mu _{t_{4}}\right\vert \left( \phi \right)
\end{eqnarray*}%
Employing the 4D Sobolev 
\begin{equation}
\left\Vert \left\vert \phi \right\vert ^{2}\phi \right\Vert
_{H^{-1}}\leqslant C\left\Vert \phi \right\Vert _{H^{1}}^{3},
\label{Sobolev:4D}
\end{equation}%
on the rough coupling, we get to 
\begin{eqnarray*}
&&\left\Vert \left\langle \nabla _{x_{1}}\right\rangle ^{-1}\left\langle
\nabla _{x_{1}^{\prime }}\right\rangle ^{-1}I\right\Vert _{L_{t_{1}}^{\infty
}L_{x,x^{\prime }}^{2}} \\
&\leqslant &C^{3}\int_{0}^{T}dt_{4}\int d\left\vert \mu _{t_{4}}\right\vert
\left( \phi \right) \left\Vert \phi \right\Vert _{H^{1}}^{6}\left( T^{\frac{1%
}{7}}M_{0}^{\frac{3}{5}}\Vert P_{\leqslant M_{0}}\phi \Vert _{H^{1}}+\Vert
P_{>M_{0}}\phi \Vert _{H^{1}}\right) ^{2}.
\end{eqnarray*}%
Plug in the support property of the measure (see (\ref{set:spt of qdF
Measure})), 
\begin{eqnarray}
\left\Vert \left\langle \nabla _{x_{1}}\right\rangle ^{-1}\left\langle
\nabla _{x_{1}^{\prime }}\right\rangle ^{-1}I\right\Vert _{L_{t_{1}}^{\infty
}L_{x,x^{\prime }}^{2}} &\leqslant &C^{3}C_{0}^{6}\left( T^{\frac{1}{7}%
}M_{0}^{\frac{3}{5}}C_{0}+\varepsilon \right) ^{2}\int_{0}^{T}dt_{4}\int
d\left\vert \mu _{t_{4}}\right\vert \left( \phi \right)
\label{E:estimate of example gamma(1)} \\
&\leqslant &C^{3}C_{0}^{6}\left( T^{\frac{1}{7}}M_{0}^{\frac{3}{5}%
}C_{0}+\varepsilon \right) ^{2}2T.  \notag
\end{eqnarray}%
and we are done.

\subsection{The Extended Klaineriman-Machedon Board Game is Compatible\label%
{Sec:WhyWeNeedExtKM}}

In \S \ref{sec:key example}, the $U$-$V$ estimates worked perfectly with the
integration limits obtained via the extended Klainerman-Machedon board game
in \S \ref{Sec:ElaboratedKMBoardGame}. One certainly wonders if the extended
KM board game is necessary and if the extended KM board game is compatible
with the estimates in the general case.

In the beginning of \S \ref{Sec:Wild1}, we briefly mentioned the problem one
would face without the extended Klainerman-Machedon board game. We can now
explain by a concrete example. For comparison, rewrite $I_{1}$, in Example %
\ref{example:one could indeed combine more} with the above notation,

\begin{eqnarray*}
I_{1}
&=&\int_{t_{4}=0}^{t_{1}}\int_{t_{2}=t_{4}}^{t_{1}}%
\int_{t_{3}=0}^{t_{2}}U^{(1)}(t_{1}-t_{2})B_{1,2}^{-}U^{(2)}(t_{2}-t_{3})B_{1,3}^{+}U^{(3)}(t_{3}-t_{4})B_{3,4}^{-}\gamma ^{(4)}d%
\underline{t}_{4} \\
&=&\int_{t_{4}=0}^{t_{1}}dt_{4}\int d\mu _{t_{4}}\left( \phi \right)
\int_{t_{2}=t_{4}}^{t_{1}}\left( U_{1}D_{\phi }^{(2)}\left( x_{1}\right) 
\left[ \int_{t_{3}=0}^{t_{2}}U_{1}D_{\phi ,R}^{(3)}(x_{1}^{\prime })dt_{3}%
\right] \right) dt_{2}.
\end{eqnarray*}%
One sees that the the $dt_{3}$ integral is encapsulated inside the $dt_{2}$
integral, or the $x$ and $x^{\prime }$ parts do not factor, even with the
carefully worked out time integration limits in the original
Klainerman-Machedon board game. Hence, one cannot apply the $U$-$V$
estimates like the above. To be very precise for readers curious about this,
as there are only two integrals that got entangled, $I_{1}$ could in fact be
estimated using \cite[(4.25), p.60]{KTV}, based on the idea of integration
by parts. However, if one allows the coupling level to be large, it is not
difficult to find, at any stage of a long coupling, multiple encapsulations,
which have more than three factors entangled together and cannot be
estimated by the ideas of integration by parts. We are not presenting such a
construction as the formula would be unnecessarily long and does not give
new ideas. Finally, we remark that, such an entanglement problem, generated
by the time integral reliance of the $U$-$V$ spaces techniques, does not
show up in the couplings with only $B^{+}$ or only $B^{-}$ or does not have
to emerge in the $\mathbb{R}^{3}/\mathbb{R}^{4}/\mathbb{T}^{3}$ cases in
which $U$-$V$ spaces are not necessary.

We now prove how the extended KM board game is compatible with the $U$-$V$
techniques. Given a reference tree, we will create a Duhamel tree (we write $%
D$-tree for short) to supplement the reference tree. The $D$-tree
supplements the given reference tree in the sense that the $D$-tree shows
the arrangment of the cubic terms $D^{(j)}$, defined in \S \ref{sec:key
example}, completely, while one could also read the integation limits off
from it as in the given reference tree. The whole point of the $D$-tree is
to get these two pieces of information in the same picture as the proof of
compatibility then follows trivially. Of course, from now on, we assume (\ref%
{eqn:qdF Rep for zero initial solution}) has already been plugged in and we
are doing the $dt_{k+1}$ integral, which is from $0$ to $t_{1}$, last.

\begin{algorithm}
\label{alg:D-tree}\leavevmode In the $D$ tree, we will write each node
prefaced by a $D$. Each node $D^{(j)}$ will have a left child, middle child,
and right child:
\end{algorithm}

\begin{equation}
\begin{tikzpicture} \node{$D^{(j)}$} child{node{$ \text{
\tikz[baseline=(X.base)] \node (X) [draw, shape=circle, inner sep=0]
{${\strut \;}$}; } $} edge from parent node[right]{\tiny{ls}}} child{node{$
\text{ \tikz[baseline=(X.base)] \node (X) [draw, shape=circle, inner sep=0]
{${\strut \;}$}; } $} edge from parent
node[right]{\tiny{\hspace{-2px}r$+$}}} child{node{$ \text{
\tikz[baseline=(X.base)] \node (X) [draw, shape=circle, inner sep=0]
{${\strut \;}$}; } $} edge from parent node[right]{\tiny{r$-$}}};
\end{tikzpicture}  \label{tree-D-tree example}
\end{equation}%
The labeling of \emph{ls}, \emph{$r+$}, \emph{$r-$} for the left, middle,
and right child, is a shorthand mnemonic for the procedure for determining
the children of $D^{(j)}$ by inspecting of the reference tree. Apply the
following steps for $j=1$ (with no left child), then repeat the steps for
all $D^{(j)}$ that appear as children. Continue to repeat the steps below
until all vertices without children are $F$:

\begin{enumerate}
\item To determine the left child of $D^{(j)}$, locate node $j$ in the
reference tree and apply the \textquotedblleft left same\textquotedblright\
rule. If the $j$ node in the reference tree is $+$, and $j+$ has a left
child $\ell +$ (of the same sign $+$), then place $D^{(l)}$ as the left
child in the $D$-tree. If the $j$ node in the reference tree is $-$, and $j-$
has a left child $\ell -$ (of the same sign $-$), then place $D^{(l)}$ as
the left child in the $D$-tree. If the $j$ node in the reference tree does
not have a left child of the same sign, then place $F$ as the left child of $%
D^{(j)}$ in the $D$-tree.

\item To determine the middle and right child of $D^{(j)}$, locate the node $%
j$ in the reference tree. Examine the right child of $j$ (if it exists), and
consider its full left branch 
\begin{equation*}
p_{1}+,\;\ldots ,\;p_{\alpha }+,\;n_{1}-,\;\ldots ,\;n_{\beta }-
\end{equation*}%
(It is possible here that $\alpha =0$ (no $+$ nodes on this left branch) and
it is also possible that $\beta =0$ (no $-$ nodes on this left branch). In
the $D$-tree, as the middle child of $D^{(j)}$, place $D^{(p_{1})}$, and as
the right child of $D^{(j)}$, place $D^{(n_{1})}$. If either or both are
missing ($\alpha =0$ or $\beta =0$ respectively), then place $F$ instead.
\end{enumerate}

A quick and simple example is the $D$-tree for the integral in \S \ref%
{sec:key example}

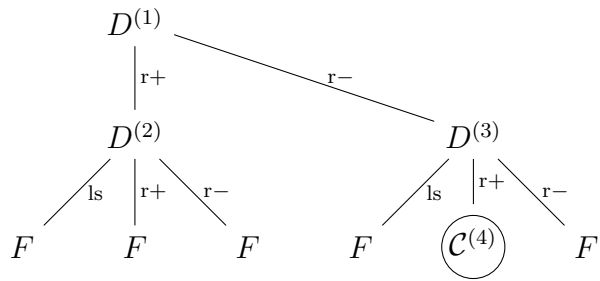
\begin{equation*}
\begin{tikzpicture} \node{$D^{(1)}$} child[missing] child[missing]
child[missing] child{node{$D^{(2)}$} child{node{$F$} edge from parent
node[right]{\tiny{ls}}} child{node{$F$} edge from parent
node[right]{\tiny{\hspace{-2px}r$+$}}} child{node{$F$} edge from parent
node[right]{\tiny{r$-$}}} edge from parent
node[right]{\tiny{\hspace{-2px}r$+$}} } child[missing] child[missing]
child{node{$D^{(3)}$} child{node{$F$} edge from parent
node[right]{\tiny{ls}}} child{node{$ \text{ \tikz[baseline=(X.base)] \node
(X) [draw, shape=circle, inner sep=0] {${\strut \mathcal{C}^{(4)}}$}; } $}
edge from parent node[right]{\tiny{\hspace{-2px}r$+$}}} child{node{$F$} edge
from parent node[right]{\tiny{r$-$}}} edge from parent
node[right]{\tiny{\;\;r$-$}} }; \end{tikzpicture}
\end{equation*}

Here is a longer example.

\begin{example}
\label{example:ref-tree to D tree-1}Consider the following reference tree
\end{example}

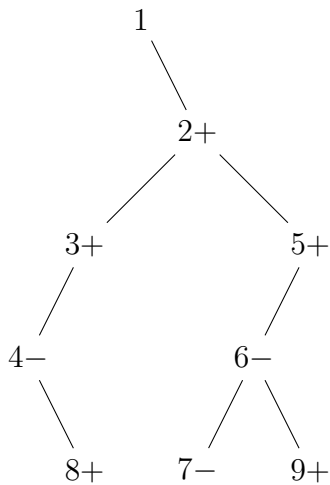
\begin{equation*}
\begin{tikzpicture} \node{$1$} child[missing] child{node{$2+$}
child{node{$3+$} child{node{$4-$} child[missing] child{node{$8+$}} }
child[missing] } child[missing] child{node{$5+$} child{node{$6-$}
child{node{$7-$}} child{node{$9+$}} } child[missing] } }; \end{tikzpicture}
\end{equation*}%
Its supplemental $D$-tree is

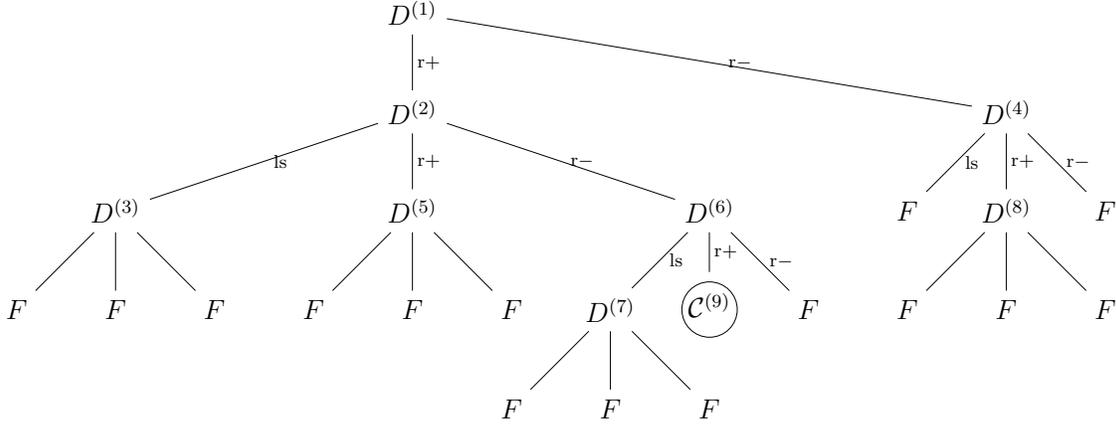
\resizebox{6in}{!}{
\begin{tikzpicture} \node{$D^{(1)}$} child[missing] child[missing]
child[missing] child[missing] child[missing] child[missing]
child{node{$D^{(2)}$} child{node{$D^{(3)}$} child{node{$F$}}
child{node{$F$}} child{node{$F$}} edge from parent node[right]{\tiny{ls}} }
child[missing] child[missing] child{node{$D^{(5)}$} child{node{$F$}}
child{node{$F$}} child{node{$F$}} edge from parent
node[right]{\tiny{\hspace{-2px}r$+$}} } child[missing] child[missing]
child{node{$D^{(6)}$} child{node{$D^{(7)}$} child{node{$F$}}
child{node{$F$}} child{node{$F$}} edge from parent node[right]{\tiny{ls}} }
child{node{$ \text{ \tikz[baseline=(X.base)] \node (X) [draw, shape=circle,
inner sep=0] {${\strut \mathcal{C}^{(9)}}$}; } $} edge from parent
node[right]{\tiny{\hspace{-2px}r$+$}}} child{node{$F$} edge from parent
node[right]{\tiny{r$-$}}} edge from parent node[right]{\tiny{r$-$}} } edge
from parent node[right]{\hspace{-2px}\tiny{r$+$}} } child[missing]
child[missing] child[missing] child[missing] child[missing]
child{node{$D^{(4)}$} child{node{$F$} edge from parent
node[right]{\tiny{ls}}} child{node{$D^{(8)}$} child{node{$F$}}
child{node{$F$}} child{node{$F$}} edge from parent
node[right]{\tiny{\hspace{-2px}r$+$}} } child{node{$F$} edge from parent
node[right]{\tiny{r$-$}}} edge from parent
node[right]{\hspace{4px}\tiny{r$-$}} }; 
\end{tikzpicture}
}

Every bottom node of the form $D^{(j)}$ (as opposed to $F$) has implicitly
three $F$ children, except for the $D^{(k+1)}$ node, which is special (in
our case here, it is $D_{9}$). In this case, the $D$-tree was generated as
follows. Take $D^{(2)}$, for example, in the reference tree.

\begin{itemize}
\item To determine the left of child $D^{(2)}$ in the $D$-tree, we look at
the reference tree and follow the \textquotedblleft left
same\textquotedblright\ rule. The left child of $2+$ is $3+$, so we place $%
D^{(3)}$ as the left child of $D^{(2)}$ in the $D$ tree. (If it were instead 
$3-$, we would place $F$ in the $D$ tree, since the signs are different).

\item To determine the middle child of $D^{(2)}$ in the $D$-tree, we look at
the reference tree and follow the \textquotedblleft right $+$%
\textquotedblright\ rule. That is, we take the right child, and consider
it's left branch: $5+$, $6-$, and $7-$, and note the first $+$ node which $%
5+ $. We assign $D^{(5)}$ as the middle child of $D^{(2)}$. If there were no 
$+$ node in the left branch, we would have assigned $F$.

\item To determine the right child of $D^{(2)}$ in the $D$-tree, we look at
the reference tree and follow the \textquotedblleft right $-$%
\textquotedblright\ rule. That is, we take the right child, and consider
it's left branch: $5+$, $6-$, and $7-$, and note the first $-$ node which is 
$6-$. We assign $D^{(6)}$ as the right child of $D_{2}$. If there were no $-$
node in the left branch, we would have assigned $F$.
\end{itemize}

\begin{proof}[Proof of Compatibility]
With the $D$-tree, we can now read (\ref{eqn:reference limits}) better. This
is because the rule for assigning upper limits of time integration is
actually the same rule for constructing children in the $D$-tree. By the
construction of the $D$-tree, we can write the form of each $D^{(j)},$ $%
j\neq k+1$ and the integration limit for $t_{j}$. If $D^{(j)}$ has children $%
L$, $M$, $R$ (for left, middle, and right) and has parent $D^{(l)}$ in the $%
D $-tree, then (ignoring the role of complex conjugates) 
\begin{equation*}
D^{(j)}(t_{j})=U(-t_{j})[\left( U_{j}L\right) \left( U_{j}M\right) \left(
U_{j}R\right) ]
\end{equation*}%
and the integration of $t_{j}$ is exactly from $0$ to $t_{l}$. One can
directly see from the picture (\ref{tree-D-tree example}) that all Duhamel
terms inside a $D^{(j)}$ must have the same integration limit and they
factor. Therefore, there is no entanglement in each stage of the coupling
process. An induction then shows that there is no entanglement for any
coupling of finite length / stages. Or in other words, the extended KM board
game is compatible with the $U$-$V$ techniques.
\end{proof}

For completeness, we finish Example \ref{example:ref-tree to D tree-1} with
the integration limits.

\begin{example}
Continuing Example \ref{example:ref-tree to D tree-1}, we have 
\begin{equation}
D^{(2)}=U(-t_{2})[U_{2}D^{(3)}\cdot U_{2}D^{(5)}\cdot U_{2}D^{(6)}]
\label{E:coll5}
\end{equation}%
and the three terms inside this expression are: 
\begin{equation*}
D^{(3)}=U(-t_{3})[U_{3}F(t_{9})\cdot U_{3}F(t_{9})\cdot U_{3}F(t_{9})]
\end{equation*}%
\begin{equation*}
D^{(5)}=U(-t_{5})[U_{5}F(t_{9})\cdot U_{5}F(t_{9})\cdot U_{5}F(t_{9})]
\end{equation*}%
\begin{equation*}
D^{(6)}=U(-t_{6})[U_{6}D^{(7)}\cdot U_{6}D^{(9)}\cdot U_{6}F(t_{9})]
\end{equation*}%
where $F(t_{i})=U(-t_{i})\phi $. On the other hand, we have 
\begin{equation}
D^{(4)}=U(-t_{4})[U_{4}F(t_{9})\cdot U_{4}D^{(8)}\cdot U_{4}F(t_{9})]
\label{E:coll6}
\end{equation}%
Now, read the time integration limits from the reference tree or the $D$%
-tree, we have $t_{2}$ and $t_{4}$ have upper limit $t_{1}$, while $t_{3}$, $%
t_{5}$, and $t_{6}$ all have upper limit $t_{2}$, etc. Start by writing $%
\int_{t_{9}=0}^{t_{1}}$ on the outside. Take all $t_{j}$ integrals for $j=2$
or for which $D^{(j)}$ is a descendant of $D^{(2)}$. This is 
\begin{equation}
\int_{t_{2}=0}^{t_{1}}\int_{t_{3}=0}^{t_{2}}\int_{t_{5}=0}^{t_{2}}%
\int_{t_{6}=0}^{t_{2}}\int_{t_{7}=0}^{t_{6}}  \label{E:coll1}
\end{equation}%
Then collect all $t_{j}$ integrals for $j=4$ or for which $D^{(j)}$ is a
descendant of $D^{(4)}$. This is 
\begin{equation}
\int_{t_{4}=0}^{t_{1}}\int_{t_{8}=0}^{t_{4}}  \label{E:coll2}
\end{equation}%
Notice that \eqref{E:coll1} and \eqref{E:coll2} split by Fubini, since none
of the limits of integration in \eqref{E:coll1} appear in \eqref{E:coll2},
and vice versa. So we can write this piece of $\gamma ^{(1)}$ as 
\begin{equation}
\gamma ^{(1)}(t_{1})=\int_{t_{9}=0}^{t_{1}}\left[ \int_{t_{2}=0}^{t_{1}}%
\int_{t_{3}=0}^{t_{2}}\int_{t_{5}=0}^{t_{2}}\int_{t_{6}=0}^{t_{2}}%
\int_{t_{7}=0}^{t_{6}}U_{1}D^{(2)}(t_{2},x_{1})\right] \left[
\int_{t_{4}=0}^{t_{1}}\int_{t_{8}=0}^{t_{4}}U_{1}D^{(4)}(t_{4},x_{1}^{\prime
})\right]  \label{E:coll4}
\end{equation}%
Write out $D^{(2)}$ as in \eqref{E:coll5} and $D^{(4)}$ as in \eqref{E:coll6}%
. Notice that we can distribute the integrals $\int_{t_{3}=0}^{t_{2}}%
\int_{t_{5}=0}^{t_{2}}\int_{t_{6}=0}^{t_{2}}$ onto the $D^{(3)}$, $D^{(5)}$
and $D^{(6)}$ terms respectively: 
\begin{eqnarray*}
&&\int_{t_{2}=0}^{t_{1}}\int_{t_{3}=0}^{t_{2}}\int_{t_{5}=0}^{t_{2}}%
\int_{t_{6}=0}^{t_{2}}\int_{t_{7}=0}^{t_{6}}U_{1}D^{(2)}(t_{2},x_{1}) \\
&=&\int_{t_{2}=0}^{t_{1}}U_{1,2}\left[ \left(
\int_{t_{3}=0}^{t_{2}}U_{2}D^{(3)}(t_{3})\right) \cdot \left(
\int_{t_{5}=0}^{t_{2}}U_{2}D^{(5)}(t_{5})\right) \cdot \left(
\int_{t_{6}=0}^{t_{2}}\int_{t_{7}=0}^{t_{6}}U_{2}D^{(6)}(t_{6})\right) %
\right]
\end{eqnarray*}%
We have kept the $t_{7}$ integral together with $t_{6}$ since $D^{(7)}$ is a
child of $D^{(6)}$ in the $D$-tree. We can see all the Duhamel structures
are fully compatible with the $U$-$V$ techniques. The rest is similar and we
neglect further details.
\end{example}

\subsection{Estimates for General $k$\label{Sec:pf of the general case}}

As the compatiblity between the extended KM board game and the $U$-$V$
techniques has been proved in \S \ref{Sec:WhyWeNeedExtKM}, we can now apply
the $U$-$V$ techniques in \S \ref{sec:key example} to the general case. We
see from \S \ref{sec:key example} that estimates (\ref{estimate:Duhamel
mlinear-1}) and (\ref{estimate:Duhamel mlinear-2}) provide gains whenever
the $l$-th coupling contracts a $U\phi $. For large $k$, there are at least $%
\frac{2}{3}k$ of the couplings carry such property and thus allow gains.

\begin{definition}
\label{def:type of coupling}For $l<k$, we say the $l$-th coupling is an
unclogged coupling, if the corresponding cubic term $\mathcal{C}^{(l+1)}$ or 
$D^{(l+1)}$ has contracted at least one $U\phi $ factor. If the $l$-th
coupling is not unclogged, we will call it a congested coupling.
\end{definition}

\begin{lemma}
\label{Lem:unclogged couplings}For large $k$, there are at least $\frac{2}{3}
k$ unclogged couplings in $k$ couplings when one plugs (\ref{eqn:qdF Rep for
zero initial solution}) into (\ref{eqn:a typical term in uniqueness}).
\end{lemma}

\begin{proof}
Assume there are $j$ congested couplings, then there are $(k-1-j)$ unclogged
couplings. Before the $(k-1)$-th coupling, there are $2k-1$ copies of $U\phi 
$ available. After the 1st coupling, all of these $2k-1$ copies of $U\phi $,
except one, must be inside some Duhamel term. Since the $j$ congested
couplings do not consume any $U\phi $, to consume all $2k-2$ copies of $%
U\phi $, we have to have 
\begin{equation}
2k-2\leqslant 3(k-1-j)  \label{eqn:total phi vs unclogged}
\end{equation}
because a unclogged coupling can, at most, consume $3$ copies of $U\phi $.
Inequality $(\ref{eqn:total phi vs unclogged})$ certainly holds only if $j< 
\frac{k}{3}$. Hence, there are at least $\frac{2k}{3}$ unclogged couplings.
\end{proof}

We can now present the algorithm which proves the general case.

\begin{itemize}
\item[Step 0] Plug (\ref{eqn:qdF Rep for zero initial solution}) into (\ref%
{eqn:a typical term in uniqueness}). Mark $\mathcal{C}_{R}^{(k+1)}$ and all $%
D^{(l+1)}$ for $l=1,...,k-1$ per the general rule given in the example / \S %
\ref{sec:key example}. Then we would have reached 
\begin{eqnarray*}
&&\left\Vert \left\langle \nabla _{x_{1}}\right\rangle ^{-1}\left\langle
\nabla _{x_{1}^{\prime }}\right\rangle
^{-1}\int_{I_{2}}...\int_{I_{k}}J_{\mu _{m},sgn}^{(k+1)}(\gamma
^{(k+1)})(t_{1},\underline{t}_{k+1})d\underline{t}_{k+1}\right\Vert
_{L_{t_{1}}^{\infty }L_{x,x^{\prime }}^{2}} \\
&\leqslant &\int_{0}^{T}dt_{k+1}\int d\left\vert \mu _{t_{k+1}}\right\vert
\left( \phi \right) \left\Vert \left( \left\langle \nabla
_{x_{1}}\right\rangle ^{-1}f^{(1)}(t_{1},x_{1})\right) \left( \left\langle
\nabla _{x_{1}^{\prime }}\right\rangle ^{-1}g^{(1)}(t_{1},x_{1}^{\prime
})\right) \right\Vert _{L_{t_{1}}^{\infty }L_{x,x^{\prime }}^{2}}
\end{eqnarray*}%
which \textquotedblleft factors\textquotedblright\ into 
\begin{eqnarray*}
&\leqslant &\int_{0}^{T}dt_{k+1}\int d\left\vert \mu _{t_{k+1}}\right\vert
\left( \phi \right) \left\Vert \left\langle \nabla _{x_{1}}\right\rangle
^{-1}f^{(1)}(t_{1},x_{1})\right\Vert _{L_{t_{1}}^{\infty
}L_{x}^{2}}\left\Vert \left\langle \nabla _{x_{1}^{\prime }}\right\rangle
^{-1}g^{(1)}(t_{1},x_{1}^{\prime })\right\Vert _{L_{t_{1}}^{\infty
}L_{x^{\prime }}^{2}} \\
&\leqslant &\int_{0}^{T}dt_{k+1}\int d\left\vert \mu _{t_{k+1}}\right\vert
\left( \phi \right) \left\Vert f^{(1)}\right\Vert _{X^{-1}}\left\Vert
g^{(1)}\right\Vert _{X^{-1}}
\end{eqnarray*}%
for some $f^{(1)}$ and $g^{(1)}$. Of course, only one of $f^{(1)}$ and $%
g^{(1)}$ can carry the cubic rough term $\mathcal{C}_{R}^{(k+1)}$ as there
is only one, so bump the other one into $X^{1}$. Go to Step 1.

\item[Step 1] Set counter $l=1$ and go to Step 2.

\item[Step 2] If $D^{(l+1)}$ is a $D_{\phi ,R}^{(l+1)}$, apply estimate (\ref%
{estimate:Duhamel mlinear-1}), put the factor carrying $\mathcal{C}
_{R}^{(k+1)}$, which would be a $U\mathcal{C}_{R}^{(k+1)}$ or a $%
D_{R}^{(j+1)}$ for some $j$, in $X^{-1}$, and replace all the $X^{1}$ norm
of $U\phi $ by the $H^{1}$ norm of $\phi $. If the ending estimate includes $%
\left\Vert U\mathcal{C}_{R}^{(k+1)}\right\Vert _{X^{-1}}$, replace it by $%
\left\Vert \mathcal{C}_{R}^{(k+1)}\right\Vert _{H^{-1}}$. Then go to Step 6.
If $D^{(l+1)}$ is not a $D_{\phi ,R}^{(l+1)}$, go to Step 3.

\item[Step 3] If $D^{(l+1)}$ is a $D_{\phi }^{(l+1)}$, apply estimate (\ref%
{estimate:Duhamel mlinear-2}), replace all the $X^{1}$ norm of $U\phi $ by
the $H^{1}$ norm of $\phi $. Then go to Step 6. If $D^{(l+1)}$ is not a $%
D_{\phi }^{(l+1)}$, go to Step 4.

\item[Step 4] If $D^{(l+1)}$ is a $D_{R}^{(l+1)}$, apply estimate (\ref%
{estimate:Duhamel mlinear-3}), put the factor carrying $\mathcal{C}
_{R}^{(k+1)}$, which would be a $U\mathcal{C}_{R}^{(k+1)}$ or a $%
D_{R}^{(j+1)}$ for some $j$, in $X^{-1}$, and replace all the $X^{1}$ norm
of $U\phi $ by the $H^{1}$ norm of $\phi $. If the ending estimate includes $%
\left\Vert U\mathcal{C}_{R}^{(k+1)}\right\Vert _{X^{-1}}$, replace it by $%
\left\Vert \mathcal{C}_{R}^{(k+1)}\right\Vert _{H^{-1}}$. Then go to Step 6.
If $D^{(l+1)}$ is not a $D_{R}^{(l+1)}$, go to Step 5.

\item[Step 5] If $D^{(l+1)}$ is a $D^{(l+1)}$, apply estimate (\ref%
{estimate:Duhamel mlinear-4}), replace all the $X^{1}$ norm of $U\phi $ by
the $H^{1}$ norm of $\phi $. Then go to Step 6. If $D^{(l+1)}$ is not a $%
D_{\phi }^{(l+1)}$, go to Step 6.

\item[Step 6] Set counter $l=l+1$. If $l<k$, go to Step 2, otherwise go to
Step 7.

\item[Step 7] Replace all the leftover $\left\Vert U\phi \right\Vert
_{X^{1}} $ by $\left\Vert \phi \right\Vert _{H^{1}}$. There is actually at
most one leftover $\left\Vert U\phi \right\Vert _{X^{1}}$ which is exactly $%
f^{(1)}$ or $g^{(1)}$ from the beginning and only happens when the sign $%
J_{\mu _{m},sgn}^{(k+1)}$ under consideration is all $+$ or all $-$. As it
is not inside any Duhamel, it is not taken care of by Steps 1-6. Go to Step
8.

\item[Step 8] After Step 7, we are now at the $k$-th coupling and would have
applied (\ref{estimate:Duhamel mlinear-1}) and (\ref{estimate:Duhamel
mlinear-2}) at least $\frac{2}{3}k$ times, thus we are looking at 
\begin{eqnarray*}
&&\left\Vert \left\langle \nabla _{x_{1}}\right\rangle ^{-1}\left\langle
\nabla _{x_{1}^{\prime }}\right\rangle
^{-1}\int_{I_{2}}...\int_{I_{k}}J_{\mu _{m},sgn}^{(k+1)}(\gamma
^{(k+1)})(t_{1},\underline{t}_{k+1})d\underline{t}_{k+1}\right\Vert
_{L_{t_{1}}^{\infty }L_{x,x^{\prime }}^{2}} \\
&\leqslant &C^{k-1}\int_{0}^{T}dt_{k+1}\int d\left\vert \mu
_{t_{k+1}}\right\vert \left( \phi \right) \left\Vert \phi \right\Vert
_{H^{1}}^{\frac{4}{3}k-1}\left( T^{\frac{1}{7}}M_{0}^{\frac{3}{5}}\Vert
P_{\leqslant M_{0}}\phi \Vert _{H^{1}}+\Vert P_{>M_{0}}\phi \Vert
_{H^{1}}\right) ^{\frac{2}{3}k}\left\Vert \left\vert \phi \right\vert
^{2}\phi \right\Vert _{H^{-1}}
\end{eqnarray*}%
Apply the 4D Sobolev (\ref{Sobolev:4D}) to the rough factor, 
\begin{equation*}
\leqslant C^{k}\int_{0}^{T}dt_{k+1}\int d\left\vert \mu
_{t_{k+1}}\right\vert \left( \phi \right) \left\Vert \phi \right\Vert
_{H^{1}}^{\frac{4}{3}k+2}\left( T^{\frac{1}{7}}M_{0}^{\frac{3}{5}}\Vert
P_{\leqslant M_{0}}\phi \Vert _{H^{1}}+\Vert P_{>M_{0}}\phi \Vert
_{H^{1}}\right) ^{\frac{2}{3}k}
\end{equation*}%
Put in the support property (\ref{set:spt of qdF Measure}), 
\begin{eqnarray*}
&\leqslant &\int_{0}^{T}dt_{k+1}\int d\left\vert \mu _{t_{k+1}}\right\vert
\left( \phi \right) C^{k}C_{0}^{\frac{4}{3}k+2}\left( T^{\frac{1}{7}}M_{0}^{%
\frac{3}{5}}C_{0}+\varepsilon \right) ^{\frac{2}{3}k} \\
&\leqslant &2TC^{k}C_{0}^{\frac{4}{3}k+2}\left( T^{\frac{1}{7}}M_{0}^{\frac{3%
}{5}}C_{0}+\varepsilon \right) ^{\frac{2}{3}k} \\
&\leqslant &2TC_{0}^{2}\left( CC_{0}^{3}T^{\frac{1}{7}}M_{0}^{\frac{3}{5}%
}+CC_{0}^{2}\varepsilon \right) ^{\frac{2}{3}k}
\end{eqnarray*}%
as claimed.
\end{itemize}

Thence, we have proved Proposition \ref{Prop:KeyUniquenessEstimate} and
hence Theorem \ref{Thm:GP uniqueness}. As mentioned before, the main theorem
/ Theorem \ref{THM:MainUniquenessTHM} then follows from Theorem \ref{Thm:GP
uniqueness} and Lemma \ref{Cor:UTFLforNLS} which checks condition (c) of
Theorem \ref{Thm:GP uniqueness} for solutions of (\ref{Hierarchy:T^4 cubic
GP}) generated by (\ref{NLS:T^4 cubic NLS}) via (\ref{eqn:solution to GP
generated by NLS}).

\appendix

\section{$\mathbb{T}^{4}$ is Special in the Multilinear Estimates Aspect 
\label{Sec A:T4 is specal}}

After reading the main part of the paper, it should now be clear that the
proof of Theorem \ref{THM:MainUniquenessTHM} goes through if the $\mathbb{T}
^{4}$ energy-critical problem is replaced by the corresponding problems on $%
\mathbb{R}^{3}$, $\mathbb{R}^{4}$ or $\mathbb{T}^{3}$. Interestingly, as the
stronger $L_{t}^{1}H_{x}^{s}$ versions of (\ref%
{eqn:H-1-T4TrilinearEstimateWithoutGain}) and (\ref%
{eqn:H1-T4TrilinearEstimateWithoutGain}) are true on $\mathbb{R}^{3}$, $%
\mathbb{R}^{4}$ and $\mathbb{T}^{3}$, one does not need to use the $U$-$V$
spaces for the corresponding problems at all. In this appendix, we will 1st
prove that the stronger $L_{t}^{1}H_{x}^{s}$ estimates fail on $\mathbb{T}
^{4} $ in \S \ref{Sec:OldEstimatesFail} and then they hold on $\mathbb{R}
^{4} $ in \S \ref{sec:R4TrilinearEstimate}. Together with \cite{HTX1} and 
\cite{CJInvent} in which the $L_{t}^{1}H_{x}^{s}$ estimates were proved for $%
\mathbb{R}^{3}$ and $\mathbb{T}^{3}$, we see that the $\mathbb{T}^{4}$ case
is indeed special in the aspect of multilinear estimates and one has to use
the $U$-$V$ spaces.

\subsection{Failure of the Stronger Sobolev Multilinear Estimates\label%
{Sec:OldEstimatesFail}}

\begin{claim}
At least one of the following two estimates 
\begin{eqnarray}
\left\Vert \dprod\limits_{j=1}^{3}e^{it\Delta }f_{j}\right\Vert
_{L_{t}^{1}H_{x}^{-1}} &\lesssim &\Vert f_{1}\Vert _{H^{-1}}\Vert f_{2}\Vert
_{H^{1}}\Vert f_{3}\Vert _{H^{1}}  \label{eqn:T4FailedH-1Estimate} \\
\left\Vert \dprod\limits_{j=1}^{3}e^{it\Delta }f_{j}\right\Vert
_{L_{t}^{1}H_{x}^{1}} &\lesssim &\Vert f_{1}\Vert _{H^{1}}\Vert f_{2}\Vert
_{H^{1}}\Vert f_{3}\Vert _{H^{1}}  \label{eqn:T4FailedH1Estimate}
\end{eqnarray}
fails on $\mathbb{T}^{4}$
\end{claim}

\begin{proof}
Assume (\ref{eqn:T4FailedH-1Estimate}) and (\ref{eqn:T4FailedH1Estimate}) do
hold, then interpolating between them yields 
\begin{equation*}
\left\Vert \dprod\limits_{j=1}^{3}e^{it\Delta }f_{j}\right\Vert
_{L_{t}^{1}L_{x}^{2}}\lesssim \Vert f_{1}\Vert _{L^{2}}\Vert f_{2}\Vert
_{H^{1}}\Vert f_{3}\Vert _{H^{1}}
\end{equation*}%
Put in $f_{j}=P_{\leqslant M}f$, we have 
\begin{equation*}
\left\Vert e^{it\Delta }P_{\leqslant M}f\right\Vert
_{L_{t}^{3}L_{x}^{6}}^{3}\lesssim \Vert P_{\leqslant M}f\Vert _{L^{2}}\Vert
P_{\leqslant M}f\Vert _{H^{1}}\Vert P_{\leqslant M}f\Vert _{H^{1}}\lesssim
M^{2}\Vert P_{\leqslant M}f\Vert _{L^{2}}^{3}
\end{equation*}%
or 
\begin{equation*}
\left\Vert e^{it\Delta }P_{\leqslant M}f\right\Vert
_{L_{t}^{3}L_{x}^{6}}\lesssim M^{\frac{2}{3}}\Vert P_{\leqslant M}f\Vert
_{L^{2}}.
\end{equation*}%
The above estimate is a $\mathbb{T}^{4}$ scale-invariant estimate carrying
the $L_{t}^{3}$ exponent. As noted in \cite{B2} and revisited in \cite{BD},
such an estimate fails or one needs an extra $\varepsilon $ for it to hold.
Thus at least one of (\ref{eqn:T4FailedH-1Estimate}) and (\ref%
{eqn:T4FailedH1Estimate}), (could be both of them), fails.
\end{proof}

\subsection{$\mathbb{R}^{4}$ Trilinear Estimates\label%
{sec:R4TrilinearEstimate}}

\begin{lemma}
\label{Lem:R4TrilinearWithFreqLocal}On $\mathbb{R}^{4},$ 
\begin{eqnarray}
\hspace{-0.3in}\left\Vert e^{it\Delta }f_{1}e^{it\Delta }f_{2}e^{it\Delta
}f_{3}\right\Vert _{L_{T}^{1}H^{-1}} &\lesssim &\Vert f_{1}\Vert
_{H^{-1}}\Vert f_{2}\Vert _{H^{1}}\Vert f_{3}\Vert _{H^{1}}
\label{eqn:R4HighFreqMultilinearEstimate} \\
\left\Vert e^{it\Delta }f_{1}e^{it\Delta }f_{2}e^{it\Delta }f_{3}\right\Vert
_{L_{T}^{1}H^{1}} &\lesssim &\Vert f_{1}\Vert _{H^{1}}\Vert f_{2}\Vert
_{H^{1}}\Vert f_{3}\Vert _{H^{1}}
\label{eqn:R4HighFreqMultilinearEstimateH1}
\end{eqnarray}

In particular, (\ref{eqn:R4HighFreqMultilinearEstimate}) and (\ref%
{eqn:R4HighFreqMultilinearEstimateH1}) imply (\ref%
{eqn:H-1-T4TrilinearEstimateWithoutGain}) and (\ref%
{eqn:H1-T4TrilinearEstimateWithoutGain}). That is, (\ref%
{eqn:R4HighFreqMultilinearEstimate}) and (\ref%
{eqn:R4HighFreqMultilinearEstimateH1}) are indeed stronger than (\ref%
{eqn:H-1-T4TrilinearEstimateWithoutGain}) and (\ref%
{eqn:H1-T4TrilinearEstimateWithoutGain}).
\end{lemma}

Before proving estimates (\ref{eqn:R4HighFreqMultilinearEstimate}) and (\ref%
{eqn:R4HighFreqMultilinearEstimateH1}), we give a brief comment on how (\ref%
{eqn:R4HighFreqMultilinearEstimate}) and (\ref%
{eqn:R4HighFreqMultilinearEstimateH1}) imply (\ref%
{eqn:H-1-T4TrilinearEstimateWithoutGain}) and (\ref%
{eqn:H1-T4TrilinearEstimateWithoutGain}) as one could get more than one
version of them. Estimates (\ref{eqn:R4HighFreqMultilinearEstimate}) and (%
\ref{eqn:R4HighFreqMultilinearEstimateH1}) will be proved using H\"{o}lder,
Strichartz, ... with $\geqslant 2$ time exponents. Therefore, using $%
\left\Vert \int_{0}^{t}e^{i(t-s)\Delta }f(s)ds\right\Vert _{X^{1}}\leqslant
\left\Vert f\right\Vert _{L_{t}^{1}H_{x}^{1}}$ and the inclusion that $%
\left\Vert f\right\Vert _{U^{p}}\lesssim $ $\left\Vert f\right\Vert _{U^{2}}$
for $p\geqslant 2$, one reduces (\ref{eqn:H-1-T4TrilinearEstimateWithoutGain}
) and (\ref{eqn:H1-T4TrilinearEstimateWithoutGain}) on $\mathbb{R}^{4}$ from
(\ref{eqn:R4HighFreqMultilinearEstimate}) and (\ref%
{eqn:R4HighFreqMultilinearEstimateH1}), by applying the atomic structure of $%
U^{p}$ on the nuts and bolts. We omit the details but remark that one would
get a $U^{1}$ estimate instead if one applies the atomic structure directly
on the $L_{t}^{1}$ estimate. That is, one could have multiple version of
multilinear estimates yielding existence. Let us, for the moment, consider
the $\mathbb{T}^{3}$ quintic problem as an example since $\mathbb{R}^{3}/%
\mathbb{R}^{4}$ are simpler and $\mathbb{T}^{4}$ does not allow the
ambiguity to be mentioned. Instead of using the $\mathbb{T}^{3}$ version of (%
\ref{eqn:H-1-T4TrilinearEstimateWithoutGain}) and (\ref%
{eqn:H1-T4TrilinearEstimateWithoutGain}), one could use the $\mathbb{T}^{3}$
version of (\ref{eqn:R4HighFreqMultilinearEstimate}) and (\ref%
{eqn:R4HighFreqMultilinearEstimateH1}) which do not need $U$-$V$ techniques
or the $U^{1}$ version of their implications to show local existence for the 
$\mathbb{T}^{3}$ quintic problem in three similar but different subspaces of 
$H^{1}$. The only way to know if these three versions yield the same
solution is an unconditional uniqueness theorem.

We will need the ordinary $\mathbb{R}^{4}$ Strichartz estimate including the
endpoint, and the bilinear Strichartz estimate.

\begin{lemma}[Bilinear $\mathbb{R}^{4}$ Strichartz \protect\cite{B1}]
For $M_{1}\geqslant M_{2}$, 
\begin{equation}
\left\Vert P_{M_{1}}e^{it\Delta }f_{1}\cdot P_{M_{2}}e^{it\Delta
}f_{2}\right\Vert _{L_{t,x}^{2}}\lesssim \frac{M_{2}^{\frac{3}{2}}}{M_{1}^{%
\frac{1}{2}}}\left\Vert P_{M_{1}}f_{1}\right\Vert _{L_{x}^{2}}\left\Vert
P_{M_{2}}f_{2}\right\Vert _{L_{x}^{2}}.  \label{eqn:R4 Bilinear Strichartz}
\end{equation}
\end{lemma}

\subsubsection{Proof of Lemma \protect\ref{Lem:R4TrilinearWithFreqLocal}}

We only prove the $H^{-1}$ estimate (\ref{eqn:R4HighFreqMultilinearEstimate}%
) as the $H^{1}$ estimate (\ref{eqn:R4HighFreqMultilinearEstimateH1}) is
easier. We follow the proof of Lemma \ref{Lem:T4TrilinearWithUV} and start
with 
\begin{equation*}
I=\sum_{M_{1},M_{2},M_{3},M}I_{M_{1},M_{2},M_{3},M}
\end{equation*}%
where 
\begin{equation*}
I_{M_{1},M_{2},M_{3},M}=\iint_{x,t}u_{1}u_{2}u_{3}vdxdt
\end{equation*}%
with $u_{j}=P_{M_{j}}e^{it\Delta }f_{j}$ and $v=P_{M}g$. It will work for (%
\ref{eqn:R4HighFreqMultilinearEstimate}) in $\mathbb{R}^{4}$ because we have
the endpoint Strichartz estimate. The analysis is mainly divided into two
main cases.

\noindent \emph{Case 1}. $M_{1}\sim M_{2}\geq M_{3}$ and $M_{1}\sim
M_{2}\geq M$

\noindent \emph{Case 2}. $M_{1}\sim M\geq M_{2}\geq M_{3}\footnote{%
This Case 2 is more complicated than Lemma \ref{Lem:T4TrilinearWithUV}
because $g\in L_{t}^{\infty }H_{x}^{1}$}$

\paragraph{\noindent \textit{Case 1 of the }$\mathbb{R}^{4}$ case}

We cancel the $M$ decomposition by summing in $M$. That is, we consider 
\begin{equation*}
I_{M_{1},M_{2},M_{3}}=\sum_{M}I_{M_{1},M_{2},M_{3},M}=%
\iint_{x,t}u_{1}u_{2}u_{3}gdxdt
\end{equation*}%
in this Case 1. Recalling $M_{1}$, $M_{2}$, $M_{3}$ are subject to the
condition $M_{1}\sim M_{2}\geq M_{3}$, we then have 
\begin{equation*}
I_{M_{1},M_{2},M_{3}}\leq \Vert u_{1}\,u_{2}\,u_{3}\,\,v\Vert
_{L_{t,x}^{1}}\lesssim \Vert u_{1}u_{3}\Vert _{L_{t,x}^{2}}\Vert u_{2}\Vert
_{L_{t}^{2}L_{x}^{4}}\Vert g\Vert _{L_{t}^{\infty }L_{x}^{4}}
\end{equation*}%
By bilinear Strichartz (\ref{eqn:R4 Bilinear Strichartz}), and the ordinary $%
\mathbb{R}^{4}$ endpoint Strichartz, we have 
\begin{eqnarray*}
I_{M_{1},M_{2},M_{3}} &\leqslant &\frac{M_{3}^{\frac{3}{2}}}{M_{1}^{\frac{1}{%
2}}}\left\Vert P_{M_{1}}f_{1}\right\Vert _{L_{x}^{2}}\left\Vert
P_{M_{3}}f_{3}\right\Vert _{L_{x}^{2}}\left\Vert P_{M_{2}}f_{2}\right\Vert
_{L_{x}^{2}}\left\Vert g\right\Vert _{L_{t}^{\infty }H_{x}^{1}} \\
&\leqslant &\frac{M_{3}^{\frac{1}{2}}}{M_{1}^{\frac{1}{2}}}\left\Vert
P_{M_{1}}f_{1}\right\Vert _{L_{x}^{2}}\left\Vert P_{M_{3}}f_{3}\right\Vert
_{H_{x}^{1}}\left\Vert P_{M_{2}}f_{2}\right\Vert _{L_{x}^{2}}\left\Vert
g\right\Vert _{L_{t}^{\infty }H_{x}^{1}}
\end{eqnarray*}%
Swapping an $M_{1}$ and $M_{2}$ factor 
\begin{equation*}
\leqslant \frac{M_{3}^{\frac{1}{2}}}{M_{1}^{\frac{1}{2}}}\left\Vert
P_{M_{1}}f_{1}\right\Vert _{H_{x}^{-1}}\left\Vert P_{M_{3}}f_{3}\right\Vert
_{H_{x}^{1}}\left\Vert P_{M_{2}}f_{2}\right\Vert _{H_{x}^{1}}\left\Vert
g\right\Vert _{L_{t}^{\infty }H_{x}^{1}}
\end{equation*}%
Carrying out the sum in this case, we obtain 
\begin{equation*}
I_{1}\lesssim \sup_{M_{3}}\Vert P_{M_{3}}f_{3}\Vert _{H_{x}^{1}}\left\Vert
g\right\Vert _{L_{t}^{\infty }H_{x}^{1}}\sum_{\substack{ M_{1}  \\ M_{1}\sim
M_{2}}}\left\Vert P_{M_{1}}f_{1}\right\Vert _{H_{x}^{-1}}\left\Vert
P_{M_{2}}f_{2}\right\Vert _{H_{x}^{1}}
\end{equation*}%
Cauchy-Schwarz and summing in $M_{1}$ concludes the proof of Case 1 of the $%
\mathbb{R}^{4}$ case .

\paragraph{\noindent \textit{Case 2 of the }$\mathbb{R}^{4}$ case}

Recall that 
\begin{equation*}
M_{1}\sim M>M_{2}\geq M_{3}
\end{equation*}
in Case 2. We will assume $v=P_{M_{1}}g$ for convenience. We split Case 2
into two more subcases: Case 2A in which $\frac{M_{2}}{M_{1}}\leqslant \frac{
1}{M_{2}}$ and Case 2B in which $\frac{M_{2}}{M_{1}}\geqslant \frac{1}{M_{2}}
$.

\subparagraph{Case 2A of \textit{the }$\mathbb{R}^{4}$ case}

We start with 
\begin{equation*}
I_{2A}\lesssim \sum_{\substack{ M_{1},M_{2},M_{3}  \\ M_{1}>M_{2}\geq M_{3}}}
\Vert (\langle \nabla \rangle ^{-1}u_{1})\,u_{2}\Vert _{L_{t,x}^{2}}\Vert
u_{3}\langle \nabla \rangle v\Vert _{L_{t,x}^{2}}
\end{equation*}
By Cauchy-Schwarz in the $M_{1}$ sum 
\begin{equation*}
I_{2A}\lesssim A_{I}B_{I}
\end{equation*}
where $A$ and $B$ are, for fixed $M_{2},M_{3}$ given by 
\begin{equation*}
A_{I}^{2}=\sum_{\substack{ M_{1}  \\ M_{1}\geq M_{2}}}\Vert \langle \nabla
\rangle ^{-1}u_{1}\,u_{2}\Vert _{L_{x,t}^{2}}^{2}\,,\qquad B_{I}^{2}=\sum 
_{\substack{ M_{1}  \\ M_{1}\geq M_{2}}}\Vert u_{3}\langle \nabla \rangle
v\Vert _{L_{x,t}^{2}}^{2}
\end{equation*}

For $A_{I}$, we apply the bilinear Strichartz (\ref{eqn:R4 Bilinear
Strichartz}), 
\begin{equation*}
A_{I}\lesssim M_{2}\Vert P_{M_{2}}f_{2}\Vert _{L^{2}}\left( \sum_{\substack{ %
M_{1}  \\ M_{1}\geq M_{2}}}\frac{M_{2}}{M_{1}}\Vert P_{M_{1}}f_{1}\Vert
_{H^{-1}}^{2}\right) ^{1/2}
\end{equation*}
Use $\frac{M_{2}}{M_{1}}\leqslant \frac{1}{M_{2}}$ in Case 2A, 
\begin{equation*}
A_{I}\lesssim M_{2}^{\frac{1}{2}}\Vert P_{M_{2}}f_{2}\Vert _{L^{2}}\Vert
f_{1}\Vert _{H^{-1}}\lesssim M_{2}^{-\frac{1}{2}}\Vert P_{M_{2}}f_{2}\Vert
_{H^{1}}\Vert f_{1}\Vert _{H^{-1}}
\end{equation*}

For $B_{I}$, we first write out the integral 
\begin{equation*}
B_{I}^{2}\lesssim \sum_{\substack{ M_{1}  \\ M_{1}\geq M_{2}}}
\iint_{x,t}|u_{3}|^{2}|\langle \nabla \rangle v|^{2}\,dx\,dt
\end{equation*}
H\"{o}lder in the $x$-integral and bring the $M_{1}$-sum inside the $t$
integral 
\begin{equation*}
B_{I}^{2}\lesssim \int_{t}\Vert u_{3}\Vert _{L_{x}^{\infty }}^{2}\sum 
_{\substack{ M_{1}  \\ M_{1}\geq M_{2}}}\int_{x}|\langle \nabla \rangle
v|^{2}\,dx\,dt\lesssim \int_{t}\Vert u_{3}\Vert _{L_{x}^{\infty }}^{2}\Vert
g\Vert _{H^{1}}^{2}\,dt
\end{equation*}
Now sup the $\Vert g\Vert _{H_{x}^{1}}$ term out of the $t$ integral and use
the $\mathbb{R}^{4}$ endpoint Strichartz 
\begin{equation*}
B_{I}\lesssim \Vert u_{3}\Vert _{L_{t}^{2}L_{x}^{\infty }}\Vert g\Vert
_{L_{t}^{\infty }H_{x}^{1}}\lesssim M_{3}\Vert u_{3}\Vert
_{L_{t}^{2}L_{x}^{4}}\Vert g\Vert _{L_{t}^{\infty }H_{x}^{1}}\lesssim \Vert
P_{M_{3}}f_{3}\Vert _{H_{x}^{1}}\Vert g\Vert _{L_{t}^{\infty }H_{x}^{1}}
\end{equation*}
Putting it all together (the $M_{1}$ sum has already been taken care of) 
\begin{eqnarray*}
I_{2A} &\lesssim &\Vert f_{1}\Vert _{H^{-1}}\Vert g\Vert _{L_{t}^{\infty
}H_{x}^{1}}\sum_{\substack{ M_{2},M_{3}  \\ M_{2}\geq M_{3}}}M_{2}^{-\frac{1 
}{2}}\Vert P_{M_{3}}f_{3}\Vert _{H_{x}^{1}}\Vert P_{M_{2}}f_{2}\Vert _{H^{1}}
\\
&\lesssim &\Vert f_{1}\Vert _{H^{-1}}\Vert g\Vert _{L_{t}^{\infty
}H_{x}^{1}}\Vert f_{3}\Vert _{H_{x}^{1}}\sum_{M_{2}}\left( M_{2}^{-\frac{1}{%
2 }}\log M_{2}\right) \Vert P_{M_{2}}f_{2}\Vert _{H^{1}} \\
&\lesssim &\Vert f_{1}\Vert _{H^{-1}}\Vert g\Vert _{L_{t}^{\infty
}H_{x}^{1}}\Vert f_{3}\Vert _{H_{x}^{1}}\Vert f_{2}\Vert _{H^{1}}
\end{eqnarray*}

\subparagraph{Case 2B of \textit{the }$\mathbb{R}^{4}$ case}

Switch the Cauchy-Schwarz combination, 
\begin{equation*}
I_{2B}\lesssim \sum_{\substack{ M_{1},M_{2},M_{3}  \\ M_{1}>M_{2}\geq M_{3}}}
\Vert (\langle \nabla \rangle ^{-1}u_{1})\,u_{3}\Vert _{L_{t,x}^{2}}\Vert
u_{2}\langle \nabla \rangle v\Vert _{L_{t,x}^{2}}
\end{equation*}
We then Cauchy-Schwarz in the $M_{1}$ sum to get to 
\begin{equation*}
I_{2B}\lesssim A_{II}B_{II}
\end{equation*}
where $A$ and $B$ are, for fixed $M_{2},M_{3}$ given by 
\begin{equation*}
A_{II}^{2}=\sum_{\substack{ M_{1}  \\ M_{1}\geq M_{2}}}\Vert \langle \nabla
\rangle ^{-1}u_{1}u_{3}\Vert _{L_{x,t}^{2}}^{2}\,,\qquad B_{II}^{2}=\sum 
_{\substack{ M_{1}  \\ M_{1}\geq M_{2}}}\Vert u_{2}\langle \nabla \rangle
v\Vert _{L_{x,t}^{2}}^{2}
\end{equation*}
For $A_{II}$, we apply the bilinear Strichartz (\ref{eqn:R4 Bilinear
Strichartz}), 
\begin{equation*}
A_{B}\lesssim M_{3}^{\frac{3}{2}}\Vert P_{M_{3}}f_{2}\Vert _{L^{2}}\left(
\sum_{\substack{ M_{1}  \\ M_{1}\geq M_{2}}}\frac{1}{M_{1}}\Vert
P_{M_{1}}f_{1}\Vert _{H^{-1}}^{2}\right) ^{1/2}
\end{equation*}
Since $\frac{M_{2}}{M_{1}}\geqslant \frac{1}{M_{2}}$ which implies $%
M_{2}\leqslant M_{1}\leqslant M_{2}^{2}$ in Case 2B, 
\begin{equation*}
A_{II}\lesssim M_{3}^{\frac{3}{2}}\Vert P_{M_{3}}f_{3}\Vert _{L^{2}}\Vert
\Vert P_{M_{2}\leq \bullet \leq M_{2}^{2}}f_{1}\Vert _{H^{-\frac{3}{2}
}}\lesssim M_{3}^{\frac{1}{2}}\Vert P_{M_{3}}f_{3}\Vert _{H^{1}}\Vert
P_{M_{2}\leq \bullet \leq M_{2}^{2}}f_{1}\Vert _{H^{-\frac{3}{2}}}
\end{equation*}

For $B_{II}$, just like $B_{I}$, we get to 
\begin{equation*}
B\lesssim \Vert P_{M_{2}}f_{2}\Vert _{H_{x}^{1}}\Vert g\Vert _{L_{t}^{\infty
}H_{x}^{1}}
\end{equation*}
Putting it all together (since the $M_{1}$ sum has already been carried out) 
\begin{eqnarray*}
I_{2B} &\lesssim &\Vert g\Vert _{L_{t}^{\infty }H_{x}^{1}}\sum_{\substack{ %
M_{2},M_{3}  \\ M_{2}\geq M_{3}}}M_{3}^{\frac{1}{2}}\Vert
P_{M_{3}}f_{3}\Vert _{H_{x}^{1}}\Vert P_{M_{2}}f_{2}\Vert _{H_{x}^{1}}\Vert
P_{M_{2}\leq \bullet \leq M_{2}^{2}}f_{1}\Vert _{H^{-\frac{3}{2}}} \\
&\lesssim &\Vert g\Vert _{L_{t}^{\infty }H_{x}^{1}}\Vert f_{3}\Vert
_{H_{x}^{1}}\sum_{M_{2}}M_{2}^{\frac{1}{2}}\Vert P_{M_{2}}f_{2}\Vert
_{H^{1}}\Vert P_{M_{2}\leq \bullet \leq M_{2}^{2}}f_{1}\Vert _{H^{-\frac{3}{%
2 }}}
\end{eqnarray*}
Cauchy-Schwarz in $M_{2}$, 
\begin{equation*}
\lesssim \Vert g\Vert _{L_{t}^{\infty }H_{x}^{1}}\Vert f_{3}\Vert
_{H_{x}^{1}}\left( \sum_{M_{2}}\Vert P_{M_{2}}f_{2}\Vert _{H^{1}}^{2}\right)
^{\frac{1}{2}}\left( \sum_{M_{2}}M_{2}\Vert P_{M_{2}\leq \bullet \leq
M_{2}^{2}}f_{1}\Vert _{H^{-\frac{3}{2}}}^{2}\right) ^{\frac{1}{2}}
\end{equation*}
For the 2nd $M_{2}$ sum, decompose $M_{2}\leq \bullet \leq M_{2}^{2}$ into
dyadic pieces, and call it $M_{1}$ again, 
\begin{eqnarray*}
\sum_{M_{2}}M_{2}\Vert P_{M_{2}\leq \bullet \leq M_{2}^{2}}f_{1}\Vert _{H^{- 
\frac{3}{2}}}^{2} &=&\sum_{M_{2}}M_{2}\sum_{\substack{ M_{1}  \\ M_{2}\leq
M_{1}\leq M_{2}^{2}}}M_{1}^{-\frac{1}{2}}\Vert P_{M_{1}}f_{1}\Vert
_{H^{-1}}^{2} \\
&\leqslant &\sum_{M_{1}}M_{1}^{-1}\Vert P_{M_{1}}f_{1}\Vert
_{H^{-1}}^{2}\sum _{\substack{ M_{2}  \\ M_{2}\leq M_{1}}}M_{2}=\Vert
f_{1}\Vert _{H^{-1}}^{2}
\end{eqnarray*}
That is 
\begin{equation*}
I_{2B}\lesssim \Vert g\Vert _{L_{t}^{\infty }H_{x}^{1}}\Vert f_{3}\Vert
_{H_{x}^{1}}\Vert f_{2}\Vert _{H^{1}}\Vert f_{1}\Vert _{H^{-1}}
\end{equation*}
as needed.

\end{document}